\documentclass[preprint,nopreprintline,12pt,authoryear]{elsarticle}
\usepackage[margin=0.9in]{geometry}
\usepackage[utf8]{inputenc}
\usepackage{amsthm, amsfonts, amssymb, amsmath,enumitem, bbm, mathabx}
\usepackage{cases, amsthm}
\usepackage{mathtools}
\mathtoolsset{showonlyrefs}
\usepackage{wasysym}
\usepackage{graphicx}
\usepackage{float}
\usepackage{bbm}
\usepackage{caption}
\usepackage{subcaption}
\usepackage{algorithm}
\usepackage{algorithmic}
 \usepackage{setspace}
\let\Algorithm\algorithm
\renewcommand\algorithm[1][]{\Algorithm[#1]\setstretch{1.6}}
\usepackage[colorlinks=true, pdfstartview=FitV,
linkcolor=blue, citecolor=blue, urlcolor=blue]{hyperref}
\newtheorem{definition}{Definition}[section]
\newtheorem{assumption}{Assumption}[section]
\newtheorem{theorem}[definition]{Theorem}
\newtheorem{proposition}[definition]{Proposition}
\newtheorem{lemma}[definition]{Lemma}
\newtheorem{corollary}[definition]{Corollary}
\newtheorem{remark}[definition]{Remark}
\newtheorem{exmp}{Example}[section]
\newcommand{\N}{\mathbbm{N}}

\newcommand{\medbias}{\mathrm{Med}\mbox{-}\mathrm{bias}}

\newcommand{\unif}{\text{U}}

\DeclareMathOperator*{\argmin}{arg\,min}
\DeclareMathOperator*{\argmaxp}{arg\,max^+}
\DeclareMathOperator*{\argminp}{arg\,min^+}

\newcommand{\sign}{\mathrm{sgn}}

\newcommand{\al}{\alpha}

\newcommand{\E}{\mathbbm{E}}

\newcommand{\eps}{\xi}

\newcommand{\Xo}[1]{X_{#1:n}}
\newcommand{\Yo}[1]{Y_{[#1:n]}}
\newcommand{\xio}[1]{\xi_{[#1:n]}}
\newcommand{\slLCM}{\mathbf{slLCM}}
\newcommand{\slGCM}{\mathbf{slGCM}}

\def\R{\mathbb{R}}
\def\Pr{\mathbb{P}}

\def\E{\mathbb{E}}

\def\R{\mathbb{R}}
\def\Pr{\mathbb{P}}

\def\E{\mathbb{E}}
\begin{document}
\begin{frontmatter}
\title{Generalized Asymptotic Limit Theory and Inference for Isotonic Regression}
\author[upenn]{Soham Mallick\corref{cor1}}
\author[cmu]{Siddhaarth Sarkar}
\author[cmu]{Arun Kumar Kuchibhotla}
\cortext[cor1]{Corresponding author. Email: kcillam@wharton.upenn.edu}
\affiliation[upenn]{op={}, organization={Department of Statistics \& Data Science, University of Pennsylvania}}
\affiliation[cmu]{op={}, organization={Department of Statistics \& Data Science, Carnegie Mellon University}}
\begin{abstract}
Monotonicity is a natural shape constraint in nonparametric regression problems, arising for instance when predicting factory yield as a monotone function of labor hours. The widely used isotonic least squares estimator (LSE) does not require any tuning parameters and its rate of convergence and pointwise limiting distribution are well studied, assuming a specific local shape for the true monotone function. We introduce a general condition on the local behavior of this true function, uncovering a far richer family of asymptotic distributions than previously known. Valid inference in the classical framework has remained challenging due to the need to estimate nuisance parameters, and no existing methods address inference in our broader setup. We resolve this by showing the symmetry of these new limiting distributions, which allows the HulC procedure of Kuchibhotla, Balakrishnan, and Wasserman (2024) to produce asymptotically valid confidence intervals. More generally, our framework enables inference that remains uniformly valid over a suitably regular class of true functions.
\end{abstract}
\end{frontmatter}
\section{Introduction}
Monotonicity is a common and useful shape constraint in nonparametric regression. Suppose we have independent and identically distributed random vectors $(X_i, Y_i)\in\mathbb{R}^2, 1\le i\le n$ with the joint distribution $P_n$ such that $f_{0,n}(x) = \mathbb{E}[Y_i|X_i = x]$ is a non-decreasing function of $x$. We index the joint distribution and the conditional expectation with $n$ to indicate that they can change with the sample size.  The goal of isotonic regression is to estimate this monotone conditional mean function without imposing any parametric form. Using $f_{0,n}(\cdot)$, define the ``errors'' $\xi_i = Y_i - f_{0,n}(X_i), 1\le i\le n$. Clearly, $\mathbb{E}[\xi_i|X_i] = 0$ for all $1\le i\le n$. We do not assume any special properties on the distribution of $(\xi_i, X_i)$ such as independence, but might require boundedness of certain (conditional) moments.
There exist several nonparametric estimators for $f_{0,n}(\cdot)$, and arguably, the most natural estimator is the nonparametric Least Squares Estimator (LSE) defined as
\[
\widehat{f}_n := \argmin_{f\in\mathcal{M}}\,\frac{1}{n}\sum_{i=1}^n (Y_i - f(X_i))^2,
\]
where $\mathcal{M}$ is the class of all non-decreasing functions from $\mathbb{R}$ to $\mathbb{R}$. Although defined as an infinite-dimensional optimization problem, the estimator can be obtained via a finite-dimensional convex optimization problem because the objective function only depends on $f$ at covariate values $X_1, \ldots, X_n$. Let $\Xo{1} \le \Xo{2} \le \cdots \le \Xo{n}$ denote the order statistics of $X_1, \ldots, X_n$, with ties broken randomly, and let $\Yo{1}, \ldots, \Yo{n}$ be the corresponding responses. Finally, let $\xio{1}, \ldots, \xio{n}$ denote the corresponding errors. This means that
\[
\{(X_i, Y_i, \xi_i), 1\le i\le n\} ~=~ \{(\Xo{i}, \Yo{i}, \xio{i}), 1\le i\le n\}.
\]
With this notation, the nonparametric LSE $\widehat{f}_n$ can be obtained via the finite-dimensional problem
\begin{equation}\label{eq:isotonic-finite-dimensional}
\widehat{\theta}_n = \argmin_{\substack{\theta\in\mathbb{R}^n:\\\theta_1 \le \theta_2 \le \cdots \le \theta_n}}\,\frac{1}{n}\sum_{i=1}^n (\Yo{i} - \theta_i)^2.
\end{equation}
Then $\widehat{f}_n(\cdot)$ can be taken to be any monotone function that satisfies $\widehat{f}_n(\Xo{i}) = \widehat{\theta}_{n,i}$. A traditional choice for $\widehat{f}_n(\cdot)$ is a left-continuous, piecewise constant monotone function and satisfies a closed-form min-max (and max-min) formula:
\begin{equation}\label{eq:min-max-formula}
\begin{split}
\widehat{f}_n(x) &= \max_{i:\Xo{i}\leq x}\min_{j:\Xo{j}\geq x}\dfrac{1}{j-i+1}\sum_{k=i}^j \Yo{k},\\
&=\min_{i:\Xo{i}\leq x}\max_{j:\Xo{j}\geq x}\dfrac{1}{j-i+1}\sum_{k=i}^j \Yo{k}.
\end{split}
\end{equation}
The finite-dimensional problem~\eqref{eq:isotonic-finite-dimensional} can be solved with a computational complexity of $O(n)$ using the well-known pool adjacent violators algorithm (PAVA). See \cite{ayer1955empirical}, \cite{miles1959complete}, \cite{kruskal1964nonmetric} and \cite{best1990active} for a proof of the min-max formulation and discussions on the PAVA algorithm.
Coming to the asymptotic properties of isotonic regression estimator $\widehat{f}_n$, we note that it is one of the most studied shape-constrained non-parametric estimators. Define $\|g\| = (\int_{\mathbb{R}} g^2(x)dx)^{1/2}$ for all functions $g:\mathbb{R}\to\mathbb{R}$. It is known that, under certain regularity conditions, $\|\widehat{f}_n - f_{0,n}\| = O_p(n^{-1/3})$ no matter the smoothness of $f_{0,n}(\cdot)$~\citep{balabdaoui2019least}. Moreover, if $f_{0,n}$ is a ``simple'' function in that it is a piecewise constant monotone function with $k$-pieces, then $\|\widehat{f}_n - f_{0,n}\| = \tilde{O}_p(k/n)$~\citep{zhang2002risk,han2018robustness}.\footnote{We use the notation $\tilde{O}(\cdot)$ to ignore poly-log factors.} See~\cite{guntuboyina2018nonparametric} for a detailed review of these rates and adaptivity properties. More precise results about the pointwise behavior of $\widehat{f}_n(\cdot)$ are also known. \cite{wright1981asymptotic} is a primary reference in this context. Under some regularity conditions on the joint distribution $P_n$, if $|f_{0,n}(x) - f_{0,n}(x_0)| = A|x - x_0|^{\beta}(1 + o(1))$ as $x\to x_0$ for some $\beta > 0$, then~\cite{wright1981asymptotic} shows that
\begin{equation}\label{eq:pointwise-asym-result}
n^{\beta/(2\beta + 1)}(\widehat{f}_n(x_0) - f_{0,n}(x_0)) ~\overset{d}{\to}~ \left(\frac{\sigma^{2\beta}(x_0)A}{(\beta + 1)h^{\beta}(x_0)}\right)^{1/(2\beta + 1)}\mathbb{C}_{\beta},
\end{equation}
where $\sigma^2(x) = \mathbb{E}[\xi_i^2|X_i = x]$, $h(\cdot)$ is the Lebesgue density of the covariate distribution, and $\mathbb{C}_{\beta}$ is the slope from the left at zero of the greatest convex minorant of $t\mapsto B(t) + |t|^{\beta + 1}$ with $B(\cdot)$ representing the two-sided Brownian motion. Additionally, under the same set of assumptions, \cite{wright1984asymptotic} showed that the non-degenerate limit of the quantile isotonic regression estimator (with appropriate scaling) is also $\mathbb{C}_\beta$, suggesting that our proposed results can be extended beyond the least squares estimator. Similar extensions for $M$-estimation based monotone regression estimators are obtained in~\cite{alvarez2012m}.
It should be stressed here that, unlike the rate of convergence results for the global $L_2$-norm, the limiting distribution result~\eqref{eq:pointwise-asym-result} is derived for the fixed distribution setting (i.e., $P_n = P_0$ and $f_{0,n} = f_0$ for all $n\ge1$). In particular, it is not obvious what the limiting distribution would be if $f_{0,n}(x) = x/n^{1/5} + x^3/6, ~x\ge 0$, for example. Moreover, the proof of Wright's result is written with an implicit assumption that $\xi_i/\sigma(X_i)$ is independent of $X_i$ for all $1\le i\le n$, which we believe is a restrictive assumption; this assumption was made explicitly in~\cite{leurgans1982asymptotic}.

One of the primary goals of this paper is to extend~\eqref{eq:pointwise-asym-result} to the triangular array setting under weaker assumptions on the local behavior of $f_{0,n}(\cdot)$.
 Finally, inference in isotonic regression is a non-trivial problem. Firstly, the adaptive behavior of the LSE $\widehat{f}_n(\cdot)$ (adapting to the local ``flatness'' parameter $\beta$) implies that the rate of convergence of the estimator is, in general, unknown. Secondly, even if $\beta$ is assumed known, the limiting distribution involves two additional nuisance functions, namely the conditional variance $\sigma(\cdot)$ and the covariate density $h(\cdot)$, estimation of which involves tuning parameters and requires more assumptions on the joint distribution of the data. It is well-known that bootstrap is inconsistent for valid inference~\citep{guntuboyina2018nonparametric} and subsampling is not readily applicable for this problem as the rate of convergence is unknown. Preliminary simulations in~\cite{kuchibhotla2021hulc} show that subsampling with an estimated rate of convergence~\citep{bertail1999subsampling} has unreliable performance for finite samples, but HulC maintains good coverage even for smaller sample sizes of order $50$ for almost all ranges of $\beta > 0$. HulC~\citep{kuchibhotla2021hulc} relies on asymptotic median unbiasedness of the estimator $\widehat{f}_n(x_0)$, which was assumed in that paper. In the current paper, we prove this property and in fact, prove a stronger claim that $\mathbb{C}_{\beta}$ is a symmetric distribution for all $\beta > 0$. 
 
 Subsampling and HulC are two generic methods for inference, not tailored to monotone regression. Several attempts exist in the literature that take advantage of the structure of the monotone regression estimator to perform inference. Two prominent works in this regard are~\cite{deng2021confidence} and~\cite{cattaneo2023bootstrap}. The underlying assumptions on $f_{0,n}$ in both papers are stronger than those in~\cite{wright1981asymptotic}. \cite{deng2021confidence} propose a pivotal statistics $\widehat{r}_n(\widehat{f}_n(x_0) - f_{0,n}(x_0))$ with a data-dependent quantity $\widehat{r}_n$ and construct confidence intervals using a conservative quantile of the limiting distribution. However, they require independence of $\xi_i$ and $X_i$ as well as homoscedasticity for asymptotically valid inference. \cite{cattaneo2023bootstrap} propose a modified bootstrap to construct asymptotically valid confidence intervals without homoscedasticity or the independence assumption, but require a known upper bound on $\beta$. Another inference strategy is presented in \cite{banerjee2001likelihood} and \cite{banerjee2007likelihood}, which is based on the likelihood ratio. Although this method provides a tuning parameter-free inference strategy, it is only applicable when the first derivative is non-zero ($f'_{0,n}(x_0) >0$). A generalization of this result to higher order smoothness is provided in \cite{deng2020isotonic} but it still requires the knowledge of the number of vanishing derivatives.
\paragraph{Main contributions.}
There are two main contributions of this paper: (1) prove new limiting distribution results for $\widehat{f}_n(x_0)$ under a triangular array setting that allows $f_{0,n}(\cdot)$ to change with sample size $n$ and without assuming independence of errors and covariates; and (2) provide an asymptotically uniformly valid confidence interval for all ranges of local flatness parameter $\beta$, assuming it is finite.
Towards the first goal, we prove new pointwise asymptotic limits of the isotonic regression estimator by replacing Wright's assumption with $|f_{0,n}(x_n) - f_{0,n}(x_0)| = A|x_n - x_0|^{\beta}(1 + o(1))$ for some sequence $x_n$ converging to $x_0$ as $n\to\infty$. In fact, this is a very special case of our main result, which only assumes the existence of a diverging sequence $s_n$ such that $\sqrt{n/s_n}(f_{0,n}(x_0 + c/s_n) - f_{0,n}(x_0))$ converges to some $\psi(c)$ as $n\to\infty$. Under this assumption, we obtain the limiting distribution as the slope from the left at zero of the greatest convex minorant of a drifted two-sided Brownian motion with a specific convex function as the drift. More surprisingly, our results imply that there always exists a sequence of data-generating processes such that any non-negative convex function that is zero at zero can be obtained as the drift. To the best of our knowledge, such richness in the limiting distribution theory of isotonic regression is unexplored. In a way, this rich distribution theory is similar to that of sample quantiles and quantile regression, as showcased in~\cite{knight1998limiting,knight2002limiting}. Towards the second goal of inference, we prove symmetry of the limiting distribution under a wide range of joint distributions which implies asymptotic uniform validity of HulC under the setting of~\cite{wright1981asymptotic}, no matter what $\beta \in (0, \infty)$ is. The proposed inference, hence, is adaptive in the sense that it does not require the knowledge of $\beta\in(0, \infty)$.
\paragraph{Necessity of Triangular Arrays.} The goal of extending the existing asymptotic limit theory to triangular array setting is a necessary ingredient to develop asymptotically uniformly valid inference. To briefly illustrate this point, for each non-decreasing function $f$, suppose $\mathcal{P}(f)$ is a collection of distributions of $(X, Y)$ with $\mathbb{E}[Y|X = x] = f(x)$. Let $\mathcal{F}$ be a subset of the collection of all non-decreasing functions. For $q> 0$, a point $x_0$ in the support of $X$ and $n\ge1$, there exists $f_n\in\mathcal{F}$ and $P_n\in\mathcal{P}(f_n)$ such that
\[
\sup_{f\in\mathcal{F}}\sup_{P\in\mathcal{P}(f)}\,\mathbb{P}_{f,P}\left(|\widehat{f}_n(x_0) - f_0(x_0)| \le q\right) \le (1 + 1/n)\mathbb{P}_{f_n, P_n}\left(|\widehat{f}_n(x_0) - f_0(x_0)| \le q\right).
\]
Hence, for the left hand side to be less than $\alpha$ (asymptotically), one requires the probability under $P_n\in \mathcal{P}(f_n)$ to be less than $\alpha$ (asymptotically). The behavior of the probability on the right hand side {\em cannot} be understood from existing asymptotic limit theory for monotone regression and our triangular array framework is needed for uniformly valid inference (over a broad class of functions and distributions).
\paragraph{Organization.} The remaining article is organized as follows. In Section~\ref{sec:new-results}, we prove new asymptotic limit theorems for the isotonic estimator. As a way to illustrate our proof for dealing with errors without assuming independence from covariates, we also present a simple result about the boundedness of the LSE. In Section~\ref{sec:inference-HulC}, we provide a general result on the continuity and symmetry of random variables, defined as slope from left of greatest convex minorants of a drifted Brownian motion. Then we describe HulC and prove its asymptotic validity along with a finite sample bound on the miscoverage for a (large) subclass of monotone regression problems. In Section~\ref{sec:simulation}, we present some numerical illustrations verifying our theoretical claims and also provide a comparison of HulC's performance with that of existing methods of subsampling~\citep{bertail1999subsampling},~\cite{cattaneo2023bootstrap}, and~\cite{deng2021confidence}. Finally, we summarize the paper along with a discussion of potential future directions in Section~\ref{sec:discussions}. All proofs and additional simulations of interest are relegated to the appendices.\footnote{Code for all the simulations in the paper are available on \url{https://github.com/Arun-Kuchibhotla/HulC}}
\section{Results for Isotonic Regression}\label{sec:new-results}
\subsection{Boundedness of the LSE}
The first new result we prove for isotonic regression is the uniform boundedness of the estimator in probability. Although not directly related to the pointwise asymptotic distribution, this simple result helps introduce techniques related to concomitants, which are necessary to allow for arbitrary dependence between errors and covariates. Proposition~\ref{prop:uniform-boundedness} formally states the uniform boundedness result and does not require the monotonicity of the true conditional mean function.
\begin{proposition}\label{prop:uniform-boundedness}
Suppose $(X_i, Y_i)\in\mathbb{R}^d, 1\le i\le n$ are independent and identically distributed as $(X, Y)$. Suppose $x\mapsto \mathbb{P}(X \le x)$ is a continuous function. Set $\mu(x) = \mathbb{E}[Y|X = x]$, and $\xi = Y - \mu(X)$. Assume that $\eta(x) = (\mathbb{E}[|\xi|^{p}|X = x])^{1/p} < \infty$ for some $1 < p \le 2$. Let $\widehat{f}_n(\cdot)$ be the isotonic regression estimator given in~\eqref{eq:min-max-formula}. Then for all $n\ge1$,
\begin{equation}\label{eq:moment-bound-supremum-norm}
\left(\mathbb{E}[\|\widehat{f}_n\|_{\infty}^p]\right)^{1/p} \le \|\mu\|_{\infty} + C_p\|\eta\|_{\infty},
\end{equation}
where $C_p := {2^{2 + 1/p}}/{(2^{p-1} - 1)^{1/p}}$.
\end{proposition}
Note that~\eqref{eq:moment-bound-supremum-norm} implies that $\|\widehat{f}_n\|_{\infty} = O_p(1)$ as $n\to\infty$.
A complete proof of Proposition~\ref{prop:uniform-boundedness} is given in~\ref{appsec:proof-of-uniform-boundedness}. The proof is based on a peeling technique and Lemma 1 of~\cite{bhattacharya1974convergence}.
Proposition~\ref{prop:uniform-boundedness} extends Lemma 9 of~\cite{han2018robustness} in two ways: (1) we do not require the independence between $(\xi_1, \ldots, \xi_n)$ and $(X_1, \ldots, X_n)$; and (2) we do not require the finiteness of second moments of errors. From our proof in~\ref{appsec:proof-of-uniform-boundedness}, one could potentially weaken the assumption of uniform boundedness of $\eta(\cdot)$ to a finite moment condition.
\subsection{Asymptotic Distribution of the LSE}
In this section, we present a general result stating the limiting distribution of isotonic LSE under an assumption on the local behavior of $f_{0,n}$.
Fix $x_0\in\mathbb{R}$ in the interior of the support of the covariate distribution.  \label{subsec:assumptions}
\begin{enumerate}[label=(A\arabic*)]
\item $(X_i, Y_i), 1\le i\le n$ are independent and identically distributed random vectors satisfying $Y_i = f_{0,n}(X_i) + \xi_i$ for a monotone non-decreasing function $f_{0,n}(\cdot)$ and errors $\xi_i$'s that satisfy $\mathbb{E}[\xi_i|X_i] = 0$ for all $1\le i\le n$ and $\mathbb{E}[\xi_i^2|X_i] = \sigma^2_n(X_i) \le \overline{\sigma}^2$ (for all $n\ge1$). Moreover, $\sigma^2_n(\cdot)$ is a continuous function in a neighborhood of $x_0$ and $\sigma^2_n(x_0)$ has a limit as $n\to\infty$, i.e., there exists a neighborhood $\mathcal{S}(x_0)$ (independent of $n$) such that\label{assump:data-model-assumption}
\[
\limsup_{n\to\infty}\sup_{x\in\mathcal{S}(x_0):|x - x_0| \le \delta_n}|\sigma^2_n(x) - \sigma^2_n(x_0)| = 0,\quad\mbox{and}\quad \lim_{n\to\infty}\sigma^2_n(x_0) = \sigma_0^2 \in (0, \infty),
\]
for any sequence $\{\delta_n\}_{n\ge1}$ such that $\delta_n\to0$ as $n\to\infty$.
\item The cumulative distribution function, $H_n$, of $X_i$'s is continuous on $\mathbb{R}$ and is continuously differentiable in a neighborhood of $x_0$ with $H'_n(x_0) = h_n(x_0)>0$. Formally, there exists a neighborhood
 $\mathcal{N}(x_0)$ (independent of $n$) such that \label{assump:continuous-distribution-of-covariates}
 \[
 \limsup_{n\to\infty}\sup_{x\in\mathcal{N}(x_0):\,|x - x_0| \le \delta_n}|h_n(x) - h_n(x_0)| = 0,
 \]
 \[
 \liminf_{n\to\infty}\inf_{x\in\mathcal{N}(x_0)}h_n(x) > 0,
 \]
 for any sequence $\{\delta_n\}_{n\ge1}$ such that $\delta_n\to0$ as $n\to\infty$. Moreover, $\lim_{n\to\infty} h_n(x_0)$ exists and equals, say, $h_0 \in (0, \infty)$.
\item \label{assump:continuity-of-f_{0,n}}There exists a sequence $\{s_n\}_{n\geq 1}$ satisfying $s_n\to\infty$ and $s_n/n \to 0$ such that for any sequence $\{c_n\}_{n\ge1}$ with $c_n\to c\in\mathbb{R}\setminus\{0\}$,
\begin{equation}\label{eq:assumption-psi-limit}
\psi(c):=\displaystyle\lim_{n\to\infty}\sqrt{\frac{n}{s_n}}\left(f_{0,n}\left(x_0+\frac{c_n}{s_n}\right)-f_{0,n}(x_0)\right),
\end{equation}
exists for all $c\in\mathbb{R}\setminus\{0\}$ such that
\[
\limsup_{|c|\to\infty} |c(\psi(3c/2) - \psi(c))| = \infty
\]
and $|\psi|$ is not identically equal to $\infty$ on $\mathbb{R}\backslash\{0\}$. (Define $\psi(0) = 0$.)
\end{enumerate}
Define the function
\begin{equation}\label{eq:definition-Psi}
\Psi(t) = \mathbf{1}\{t \ge 0\}\int_0^t \psi(s)ds - \mathbf{1}\{t < 0\}\int_{t}^{0} \psi(s)ds, \quad\mbox{for all}\quad t\in\mathbb{R}.
\end{equation}
\begin{theorem}\label{thm:asymptotic-distribution-of-LSE}
For any $x_0$ in the interior of the support of the covariate distribution, under assumptions~\ref{assump:data-model-assumption},~\ref{assump:continuous-distribution-of-covariates}, and~\ref{assump:continuity-of-f_{0,n}}, we have
\[
\sqrt{\frac{n}{s_n}}\frac{h_0}{\sigma^2_0}(\hat{f}_n(x_0)-f_{0,n}(x_0))~\overset{d}{\to}~\mathbf{slGCM}[\mathcal{B}_0(t):t\in\mathbb{R}](0),
\]
where $\mathcal{B}_0(t) := B(t) + \Psi(h_0t/\sigma_0^2)$ with $B(t)$ being the two-sided Wiener--L{\'e}vy process on $\mathbb{R}$. Here, for any function $g(\cdot)$, $\mathbf{slGCM}[g(t):\,t\in \mathcal{I}](0)$ represents the slope from the left at zero of the greatest convex minorant of $t\mapsto g(t)$ on $\mathcal{I}$.
\end{theorem}
The basic sketch of a proof of Theorem~\ref{thm:asymptotic-distribution-of-LSE} is presented in Section~\ref{subsec:proof-main-result}, with a complete proof relegated to~\ref{appsec:proof-of-thm-asymptotic-dist-LSE}.
\begin{remark}[Representation of Limiting Distribution]\label{rem:representation-of-limiting-distribution}
The limiting distribution can be written in multiple ways. Similar to Exercise 3.27 of~\cite{groeneboom2014nonparametric}, one can show that for any $D > 0$,
\[
\mathbf{slGCM}[\mathcal{B}_0(t):t\in\mathbb{R}](0) ~\overset{d}{=}~ \sqrt{D}\mathbf{slGCM}\left[\mathcal{B}(t; D): t\in\mathbb{R}\right](0),
\]
where $\mathcal{B}(t; D) = B(t) + \sqrt{D}\Psi(h_0t/(D\sigma_0^2))$. The proof is as follows.
\begin{align*}
\sqrt{D}\,\mathbf{slGCM}[\mathcal{B}_0(t):\,t\in\mathbb{R}](0) &\overset{d}{=} \mathbf{slGCM}[\sqrt{D}B(t) + \sqrt{D}\Psi(h_0t/\sigma_0^2):\,t\in\mathbb{R}](0)\\
&\overset{d}{=} \mathbf{slGCM}[B(tD) + \sqrt{D}\Psi(h_0t/\sigma_0^2):\,t\in\mathbb{R}](0)\\
&\overset{d}{=} D\mathbf{slGCM}[B(s) + \sqrt{D}\Psi(h_0s/(D\sigma_0^2)):\,s\in\mathbb{R}](0),
\end{align*}
where the second equality follows from the fact that $(\sqrt{D}B(t))_{t\in\mathbb{R}} \overset{d}{=} (B(tD))_{t\in\mathbb{R}}$ and the third equality follows from the fact that for any continuous function $g$, $\mathbf{slGCM}[(1/a)g(at):\,t\in\mathbb{R}](0)\overset{d}{=}\mathbf{slGCM}[g(t):\,t\in\mathbb{R}](0)$.
\end{remark}
\begin{remark}[On the assumptions]\label{rem:relaxation-of-assumptions}
Assumption~\ref{assump:continuity-of-f_{0,n}} is the most crucial of all our assumptions, determining the shape of the limiting distribution. Our requirement of $\limsup_{|c|\to\infty}|c(\psi(3c/2) - \psi(c))| = \infty$ is made to ensure that the isotonic estimator given by the min-max formula~\eqref{eq:min-max-formula} is with probability converging to one equal to the estimator computed on a subset of the data. More precisely, this requirement implies $\psi(\cdot)$ is non-constant and non-zero as $|c|\to\infty$ and is used to ensure that $\Gamma_j(\cdot)$ defined in Proposition~\ref{prop:assumptions-implications} diverge to $\infty$ as $|c|\to\infty$; see Lemma~\ref{lem:prob-of-unequalness-goes-to-zero} in~\ref{appsec:proof-of-thm-asymptotic-dist-LSE} for the connection to isotonic estimator computed on a subset of data. On the other hand, this requirement is not readily verifiable if $\psi(c)$ is $\infty$ or $-\infty$ beyond a bounded set. This happens in the setting of locally asymmetric behavior of $f_{0,n}(\cdot)$ as discussed in Example~\ref{exmp:locally-asymmetric}. It suffices for our proof that $\psi(\cdot)$ is bounded away from zero at $\infty, -\infty$, and certain integrals of $f_{0,n}(x) - f_{0,n}(x_0)$ on the neighborhood of $x_0$; see~\eqref{eq:Markov-inequality} and Lemma~\ref{lem:control-of-v_n}.
\end{remark}
\begin{remark}[On the Limiting Distributions]\label{rem:limiting-distributions}
Theorem~\ref{thm:asymptotic-distribution-of-LSE} expands the collection of known limiting distributions for the isotone LSE significantly. The only known case in the literature~\citep{wright1981asymptotic} is when $\psi(c) = \mathfrak{C}|c|^{\beta}$ and $s_n = n^{1/(2\beta+1)}$ for some constant $\mathfrak{C}$, that too only in the case where the functions $f_{0,n}(\cdot)$ are not changing with $n$. Assumption~\ref{assump:continuity-of-f_{0,n}} which dictates the local behavior and the shape of the limiting distribution bears a close resemblance to the one made in the study of sample quantiles or quantile regression in non-regular settings. In the sample quantile problem, the cumulative distribution function of the random variables takes the place of $f_{0,n}(\cdot)$ in~\eqref{eq:assumption-psi-limit}; see, for example, Eq. (9) of~\cite{knight2002limiting} and Eq. (4) of~\cite{knight1998limiting}. It is an interesting open problem to characterize the set of all possible limiting distributions for the isotone LSE; in the case of sample quantiles, assumptions like~\ref{assump:continuity-of-f_{0,n}} allow complete characterization of limiting distributions~\citep{knight2002limiting}.
\end{remark}
Before proceeding to the proof of Theorem~\ref{thm:asymptotic-distribution-of-LSE}, we study the implications of Theorem~\ref{thm:asymptotic-distribution-of-LSE}.
Firstly, it is easy to see that for any non-decreasing function $\psi(\cdot)$ satisfying $\psi(0) = 0$, we can construct a sequence of functions $\{f_{0,n}\}_{n\ge1}$ such that assumption~\ref{assump:continuity-of-f_{0,n}} is true. In fact, one simple example of such $f_{0,n}(\cdot)$ is
\begin{equation}\label{eq:construction-of-f_0n}
f_{0,n}(x) = \alpha_0 + \sqrt{\frac{s_n}{n}}\psi(s_n(x - x_0)),\quad\mbox{for all}\quad x\in\mathbb{R}, n\ge1.
\end{equation}
Here $\alpha_0 = f_{0,n}(x_0)$ and $s_n = o(n)$. This shows that the local rate of convergence of the LSE can be slow or fast depending on the local behavior of the conditional mean function. Secondly, because $\Psi'(c) = \psi(c)$ for all continuity points $c$ of $\psi(\cdot)$, we get that the drift for the Brownian motion in the limiting distribution can also be made equal to arbitrary non-negative convex functions that are zero at zero.
Next, we show a generalization of Theorem \ref{thm:asymptotic-distribution-of-LSE} to the case where we have $k_n$ observations for each $n$, with the true function being $f_{0,n}$. Such a result is useful for ``subsampling''-type inference procedures.
\begin{corollary}\label{cor:triangular-asymptotic-distribution-of-LSE}
    Suppose $(X_i, Y_i), 1\le i\le n$ are independent and identically distributed random vectors satisfying $Y_i = f_{0,n}(X_i) + \xi_i,~i\in[n]$ for a monotone non-decreasing function $f_{0,n}(\cdot)$. Let $x_0$ in the interior of the support of the covariate distribution, and suppose assumptions \ref{assump:data-model-assumption}, \ref{assump:continuous-distribution-of-covariates}, and \ref{assump:continuity-of-f_{0,n}} hold and consider the isotonic LSE computed on a subset of these sample with size $k_n$, such that $k_n/n\to\tau_1\in(0,\infty)$ and $s_{k_n}/s_n\to\tau_2\in(0,\infty)$. Then,
   \[
    \sqrt{\frac{k_n}{s_{k_n}}}\frac{h_0}{\sigma^2_0}(\hat{f}_{k_n}(x_0)-f_{0,n}(x_0))~\overset{d}{\to}~\mathbf{slGCM}[\tilde{\mathcal{B}}_0(t):t\in\mathbb{R}](0),
    \]
    where $\tilde{\mathcal{B}}_0(t) := B(t) + \sqrt{\tau_1\tau_2}\Psi(h_0t/(\sigma_0^2\tau_2))$ with $B(t)$ being the two-sided Wiener--L{\'e}vy process on $\mathbb{R}$ and $\Psi$ is as defined in \eqref{eq:definition-Psi}.
    \end{corollary}
\noindent A proof is presented in~\ref{appsec:proof-of-cor-traingular-asymptotic-distribution-of-LSE}.
\subsection{Applications of Theorem~\ref{thm:asymptotic-distribution-of-LSE}}
\label{sec:applications-asymptotic-distribution-of-LSE}
Now, we show by way of examples the different cases that can be handled by Theorem~\ref{thm:asymptotic-distribution-of-LSE}. It will be clear that the main result of~\cite{wright1981asymptotic} is a very special case of Theorem~\ref{thm:asymptotic-distribution-of-LSE}.
\begin{exmp}[\cite{wright1981asymptotic}]\label{exmp:Wright-1981}
Suppose, as in the main result of~\cite{wright1981asymptotic}, $f_{0,n}(\cdot)$ does not change with the sample size $n$ and satisfies
\[|f_{0,n}(x)-f_{0,n}(x_0)|=A|x-x_0|^{\theta}(1+o(1))~~~\text{as }x\to x_0,~A>0.\]
Then, $f_{0,n}(x) - f_{0,n}(x_0) = A|x - x_0|^{\theta}\mathrm{sign}(x - x_0)(1 + o(1))$ as $x\to x_0$. With $s_n = n^{1/(2\theta + 1)}$, we can verify assumption~\ref{assump:continuity-of-f_{0,n}} from this condition to get
\begin{equation}\label{eq:psi-and-s_n-Wright}
\psi(c)=A\cdot\mathrm{sign}(c)|c|^{\theta},\quad s_n = n^{1/(2\theta+1)}.
\end{equation}
\[
\Psi(t) = \frac{A|t|^{\theta+1}}{\theta+1}.
\]
To illustrate, we present the plot of limiting distribution when $\theta = 2, A = 1$. Formally, we consider the data-generating process
\begin{equation}\label{eq:data-gen-process}
Y = f_{0,n}(X) + \xi,\quad\mbox{where}\; X\sim \mathrm{Unif}(-1, 1), ~\xi\sim N(0, 1),
\end{equation}
and $f_{0,n}(\cdot)$ is obtained using~\eqref{eq:construction-of-f_0n} with $\alpha_0 = 0$ and $s_n, ~\psi(\cdot)$ as in~\eqref{eq:psi-and-s_n-Wright}. In Figure~\ref{fig:psi_eg1}, we present a scatterplot of a sample of size $n = 500$, the plot of $B(t) + |(t/2)|^{3}/3$ along with its greatest convex minorant ($B(\cdot)$ is the two-sided Brownian motion), and a histogram of $10^5$ observations from the random variable $\slGCM[B(t) + |t|^3/24: t\in\mathbb{R}](0)$, which represents the limiting distribution from Theorem~\ref{thm:asymptotic-distribution-of-LSE} in this example.
\begin{figure}[!h]
    \centering
    \includegraphics[width = \textwidth]{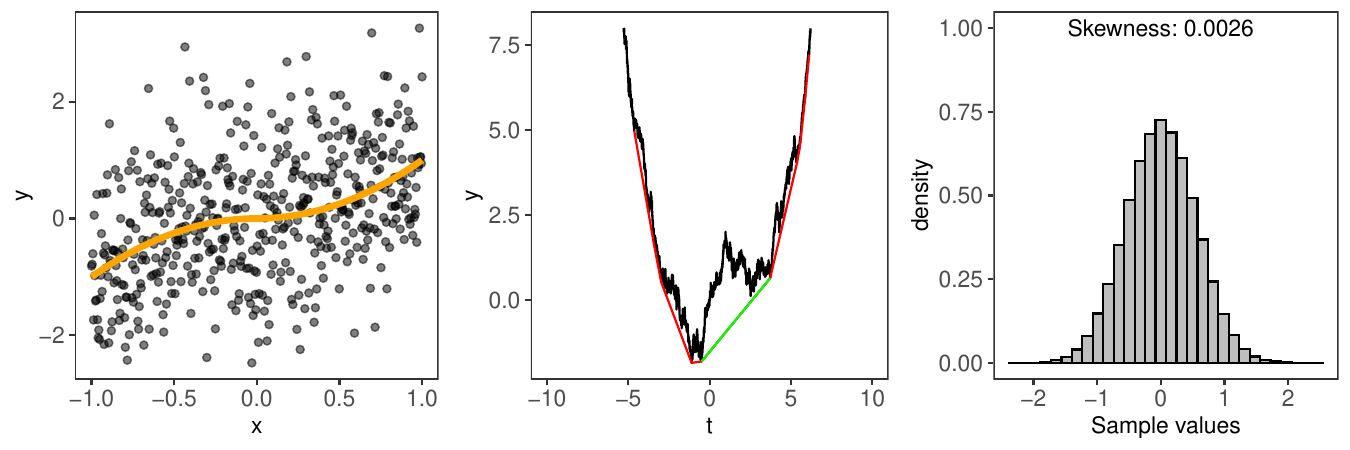}
    \caption{The left panel shows the scatter plot of data from $Y = f_{0,n}(X) + \xi,$ with $n = 500, X\sim\mathrm{Unif}(-1, 1), \xi|X\sim N(0, 1)$ and $f_{0,n}(x) = n^{-2/5}\psi(n^{1/5}X)$ along with the plot of $f_{0,n}(\cdot)$. The middle panel shows the plot of a two-sided Brownian motion with drift $|t|^3/24$ along with the greatest convex minorant; the green line shows the slope from the left of the greatest convex minorant at 0. The right panel shows the histogram of $10^5$ observations from the random variable $\slGCM[B(t) + |t|^3/24:t\in\mathbb{R}](0).$}
    \label{fig:psi_eg1}
\end{figure}
\end{exmp}
\begin{exmp}[Regularly varying functions]\label{exmp:regular-varying-functions}
This example is an extension of the assumption in~\cite{wright1981asymptotic} allowing for a slowly varying factor to the local polynomial behavior. This yields logarithmic factors in the rate of convergence.
Suppose the function $f_{0,n}$ is non-decreasing and
\[|f_{0,n}(x)-f_{0,n}(x_0)|=A|x-x_0|^{\theta}L(|x-x_0|)(1+o(1))~~~\text{as }x\to x_0,~A>0,\]
where $L:[0,\infty)\to[0,\infty)$ is a slowly varying function at 0, i.e., $L(\lambda t)/L(t) \to 1$ as $t\to0$ for all $\lambda > 0$.
This implies that
\[
f_{0,n}(x) - f_{0,n}(x_0) = A|x - x_0|^{\theta}L(|x - x_0|)\mathrm{sign}(x - x_0)(1 + o(1))\quad\mbox{as}\quad x\to x_0.
\]
By choosing $s_n$ such that $n^{1/2}L(1/s_n)/s_n^{\theta+1/2} \to 1$ as $n\to\infty$, we can take
\[
\psi(c) = A|c|^{\theta}\mathrm{sign}(c),\quad c\in\mathbb{R}.
\]
\noindent Note that this $\psi(\cdot)$ is the same as the one from Example~\ref{exmp:Wright-1981}, but now the rate of convergence is different. It is easy to see that $s_n$ here satisfies $s_n = n^{1/(2\theta+1)}L^*(n)$ for some function $L^*(\cdot)$ that is slowly varying at $\infty$.
\end{exmp}
\begin{exmp}[Locally Asymmetric Functions]\label{exmp:locally-asymmetric}
In previous examples, the assumption on $f_{0,n}$ implies that the behavior is antisymmetric around $x_0$, i.e., $\psi(-c) = -\psi(c)$ for all $c > 0$. Furthermore, the local polynomial behavior is also the same on either side. It is easy to construct examples where $f_{0,n}$ behaves like a quadratic on the left of $x_0$ and as a cubic on the right of $x_0$. In this example, we consider such cases.
Let $f_{0,n}$ be non-decreasing functions such that for some $\theta_1, \theta_2 > 0$, $A_1, A_2 > 0$ and some functions $L_1(\cdot), L_2(\cdot)$ slowly varying at $0$,
\[\lim_{x\downarrow x_0}\frac{f_{0,n}(x)-f_{0,n}(x_0)}{(x-x_0)^{\theta_1}L_1(x-x_0)}~\to~ A_1,\quad\mbox{and}\quad \lim_{x\uparrow x_0}\frac{f_{0,n}(x_0)-f_{0,n}(x)}{(x_0-x)^{\theta_2}L_2(x_0-x)}~\to~ A_2.\]
\noindent Define $s_n$ based on $\theta_1$ and $\theta_2$ as follows:
\begin{equation}\label{eq:rate-and-s_n-local-asymmetric}
\text{Choose $s_n$ s.t. }
\begin{cases}{n^{1/2}L_1(1/s_n)}/{s_n^{\theta_1+1/2}}\to 1, &\text{if }\theta_1\geq\theta_2,\\
{n^{1/2}L_2(1/s_n)}/{s_n^{\theta_2+1/2}}\to 1, &\text{if }\theta_1<\theta_2
\end{cases}
\end{equation}
Note that $s_n$ scales like $n^{1/(2\theta + 1)}$ up to a slowly varying (at $\infty$) factor depending on $n$, with $\theta := \max\{\theta_1, \theta_2\}$. This implies that if $\theta_1 \ge \theta_2$, then $n^{1/2}L_2(1/s_n)/s_n^{\theta_2 + 1/2} \to \infty\mathbf{1}_{\theta_1>\theta_2}+\mathbf{1}_{\theta_1=\theta_2}$
We claim that with this choice of $\{s_n\}_{n\ge1}$, assumption~\ref{assump:continuity-of-f_{0,n}} holds true with $\theta = \max\{\theta_1, \theta_2\}$,
\begin{equation}\label{eq:psi-definition-local-asymmetry}
\psi(c) = \begin{cases}
A_1c^{\theta}, &\mbox{if }c \ge 0, \theta_1 \ge \theta_2,\\
-A_2(-c)^{\theta}, &\mbox{if }c \le 0, \theta_1 \le \theta_2,\\
\infty, &\mbox{if }c > 0, \theta_1 < \theta_2,\\
-\infty, &\mbox{if }c < 0, \theta_1 > \theta_2.
\end{cases}
\end{equation}
This implies that
\begin{align*}
\Psi(t) &= \frac{|t|^{\theta+1}}{\theta+1}\left[{A_1}\mathbf{1}\{t \ge 0, \theta_1 \ge \theta_2\} + {A_2}\mathbf{1}\{t\le 0, \theta_1 \le \theta_2\}\right]\\
&\quad+\infty\left[\mathbf{1}\{t > 0, \theta_1 < \theta_2\}+\mathbf{1}\{t < 0, \theta_1 > \theta_2\}\right].
\end{align*}
\noindent A proof of the above fact is presented in~\ref{appsec:applications}. To connect to the classical case of the non-zero first derivative of $f_{0,n}$ at $x_0$, consider the case of $\theta_1 = \theta_2 = 1$ but with left and right derivatives of $f_{0,n}$ being unequal at $x_0$:
\[\lim_{x\downarrow x_0}\frac{f_{0,n}(x)-f_{0,n}(x_0)}{x-x_0}~\to~ A_1\quad\mbox{and}\quad \lim_{x\uparrow x_0}\frac{f_{0,n}(x)-f_{0,n}(x_0)}{x-x_0}~\to~ A_2.\]
From~\eqref{eq:rate-and-s_n-local-asymmetric}, it follows that $s_n = n^{1/3}$ and $\sqrt{n/s_n} = n^{1/3}$, which is the same rate of the case of non-zero first derivative in monotone regression. But the limiting distribution is $\mathbf{slGCM}(t\mapsto B(t) + \Psi(t))$, different from $\mathbf{slGCM}(t\mapsto B(t) + A_1t^2/2)$ or $\mathbf{slGCM}(t\mapsto B(t) + A_2t^2/2)$, where
\[
\Psi(t) = \frac{t^2}{2}\big[A_1\mathbf{1}\{t\ge0\} + A_2\mathbf{1}\{t < 0\}\big].
\]
\end{exmp}
\begin{exmp}[Near flat functions]\label{exmp:near-flat-functions}
Suppose the function $f_{0,n}$ is non-decreasing and is approximately an $N$-th degree polynomial around $x_0$:
\[
f_{0,n}(x)-f_{0,n}(x_0)=\sum_{j=1}^{N}\frac{a_{n,j}(x-x_0)^j}{j!}+R_n(x;~x_0) \quad\mbox{as}\quad x\to x_0,
\]
with the coefficients $\{a_{n,j}\}$ and $R_n(\cdot, \cdot)$ satisfying the following assumptions.
\begin{enumerate}[label=(A\arabic*$^\prime$)]
\item There exists a sequence $\{s_n\}_{n\ge1}$ satisfying $s_n\to\infty$ and $s_n/n\to0$ such that $n^{1/2}a_{n,j}/s_n^{j+1/2}\to a_j \in\mathbb{R}$ for all $1\le j\le N$ and $a_{j_0}>0$ for some $1\le j_0 \le N$.\label{eq:poly1}
\item For any sequence $\{c_n\}_{n\ge1}$ satisfying $c_n\to c$, $\sqrt{n/s_n}R_n\left(x_0+{c_n}/{s_n};x_0\right)\to 0$ as $n\to\infty$.\label{eq:poly2}
\end{enumerate}
For verification of assumption~\ref{assump:continuity-of-f_{0,n}}, observe that as $n\to\infty$,
\begin{align*}
&\sqrt{\frac{n}{s_n}}\left(f_{0,n}\left(x_0+\frac{c_n}{s_n}\right)-f_{0,n}(x_0)\right)\\
&=\sqrt{\frac{n}{s_n}}\sum_{j=1}^N \frac{a_{n,~j}}{j!}\left(\frac{c_n}{s_n}\right)^{j} + \sqrt{\frac{n}{s_n}}R_n\left(x_0+\frac{c_n}{s_n};~x_0\right)\\
&=\sum_{j=1}^N\frac{\sqrt{n}}{(s_n)^{j+\frac{1}{2}}}\cdot \frac{a_{n,j}}{j!}\cdot c_n^{j}+o(1)\\
&\rightarrow \psi(c) := \sum_{j=1}^N\frac{a_j}{j!}\cdot c^{j}.
\end{align*}
Also,
\[
    \Psi(t)=\int_0^t \psi(s)ds = \int_0^t \sum_{j=1}^N\frac{a_j}{j!}\cdot s^{j}=\sum_{j=1}^N\frac{a_j}{(j+1)!}\cdot t^{j+1}.
\]
We also note that the rate is unique, by Proposition~\ref{prop:assumptions-implications}. It is worth mentioning that not all choices of coefficients $\{a_j\}_{j\ge1}$ are allowed because $\psi(\cdot)$ has to be a non-decreasing function on $\mathbb{R}.$ But once such a $\psi(\cdot)$ is obtained, all possible $s_n$'s can be used to obtain different rates of convergences.
\end{exmp}
Note that the above example allows for functions whose first derivative is not zero but near zero. For example, consider
\begin{equation}\label{eq:f_0n-exmp4}
f_{0,n}(x) = \frac{x}{n^{1/5}} + \frac{x^3}{6}\quad\mbox{for}\quad x\in[-1, 1].
\end{equation}
Clearly, $f_{0,n}(0) = 0$, and $f_{0,n}'(0) = n^{-1/5} \neq 0$ but does converge to zero.
It is easy to check that with $s_n=n^{1/5}$ and $N=3$, $a_{n,1}=1/n^{1/5}$, $a_{n,2}=0$, $a_{n,3} = 1$, and $R_n(x;~x_0)=0$, both~\ref{eq:poly1} and~\ref{eq:poly2} are satisfied and $a_1=1$ and $a_2=a_3 = \cdots = 0$. Thus, the choice of $\psi$  and $s_n$ in this case is,
\begin{equation}\label{eq:psi-and-s_n-exmp4}
\psi(t) =  t,\ s_n=n^{1/5},
\end{equation}
and the correct rate of convergence in this case is $\sqrt{n/s_n}=n^{2/5}$.
\begin{remark}[Failure of Wright-type calibration]
A naive interpretation of Wright's result would suggest an $n^{1/3}$ rate of convergence, which leads to a wrong rate for~\eqref{eq:f_0n-exmp4}. One might wonder if it is possible to recover the correct rate by taking into account the dependence on the first derivative (i.e., $A$) from Wright's result~\eqref{eq:pointwise-asym-result}. The refined rate therefore would be $(n/f'_{0,n}(x_0))^{1/3}$, rather than just $n^{1/3}$. In this example, we know that $f'_{0,n}(0) = n^{-1/5}$ implying the rate $(n/n^{-1/5})^{1/3} = n^{2/5}$, which is indeed the correct rate of convergence derived from our result.
This recovery of the correct rate is only a coincidence and may not always be the case. We demonstrate this with a modification of $f_{0,n}$ as follows,
\begin{equation}\label{eq:f_0n-exmp4-modified}
f_{0,n}(x) = \frac{x}{n^{5/16}} + \frac{x^3}{6}\quad\mbox{for}\quad x\in[-1, 1].
\end{equation}
In this case, we still have $\beta = 1$ and  $f'_{0,n}(0) = n^{-5/16}$, giving us Wright's rate of convergence rate as $(n/n^{-5/16})^{1/3} = n^{7/16}$. However, \ref{eq:poly1} and~\ref{eq:poly2} can be satisfied by setting $s_n = n^{1/7}$ and $N = 3, a_{n,1} = n^{-5/16}, a_{n,2} = 0, a_{n,3} = 1$, and $R_n(x,x_0) = 0$. This leads to $a_1 = 0, a_3 = 1$ and giving us $\psi(t) = t^2$ and $s_n = n^{1/7}$. Therefore, the correct rate of convergence in this case is $\sqrt{n/s_n} = n^{3/7}$, which is clearly different from Wright's result. \par \vspace{1mm}
\begin{table}[h]
    \centering
    \setlength{\tabcolsep}{4pt}
    \begin{tabular}{lccccl}
    \hline
       & $\alpha_1, \alpha_2\in(0,1)$ &  $s_n$ & $(a_1, a_3)$ & Conv. rate & Asymp. Dist.\\[1 ex]
    \hline\hline
      First&   $\al_1 > \al_2$& $n^{\al_1}$ & $(1,0)$ & $n^{\frac{1}{2} - \frac{\al_1}{2}}$ & $\slGCM[B(t) + |t|^2/8](0)$ \\[1 ex]
    \hline
      Second &   $\al_1 < \al_2$ & $n^{\al_2}$ & $(0,1)$ & $n^{\frac{1}{2} - \frac{\al_2}{2}}$ & $\slGCM[B(t) + |t|^4/64](0)$ \\[1 ex]
    \hline
     Dual  &   $\al_1 = \al_2 = \al_\text{eq}$ & $n^{\al_\text{eq}}$ & $(1,1)$ & $n^{\frac{1}{2} - \frac{\al_\text{eq}}{2}}$ & $\slGCM[B(t) + |t|^2/8 + |t|^4/64](0)$ \\[1 ex]
    \hline
    \end{tabular}
    \caption{Asymptotic scenarios possible for $f_{0,n}$ in \eqref{eq:f_0n-exmp4-gen}. Wright's result and the true asymptotic behavior (convergence rate and asymptotic distribution) only match in the case where the first term dictates the asymptotic behavior (row 1). }
    \label{table:exmp-4-gen-cases}
\end{table}
To understand further on the difference between the two scenarios of $f_{0,n}$ in \eqref{eq:f_0n-exmp4} and \eqref{eq:f_0n-exmp4-modified}, consider a generalized form which includes both examples. Let $\alpha_1,\alpha_2 \in \mathbb{R}$ and consider the function $f_{0,n}$ of the following form:
\begin{align}
f_{0,n}(x) = n^{(3\alpha_1 - 1)/2}x + \frac{n^{(7\alpha_2 -1)/2}x^3}{6}\quad\mbox{for}\quad x \in[-1,1]\label{eq:f_0n-exmp4-gen}
\end{align}
In order that assumptions~\ref{eq:poly1} and~\ref{eq:poly2} are satisfied, we need the following conditions to hold for some sequence $\{s_n\}_{n\geq 0}$:
\begin{align*}
    \left(\frac{n^{\alpha_1}}{s_n}\right)^{3/2}\to a_1~\mbox{and}~\left(\frac{n^{\alpha_2}}{s_n}\right)^{7/2}\to a_3
\end{align*}
Depending on the range of values of $\alpha_1$ and $\alpha_2$, different asymptotic behavior are possible which is illustrated in Table~\ref{table:exmp-4-gen-cases}. As we can see, in the example of $f_{0,n}(x) = {x}/{n^{1/5}} + {x^3}/{6}$ in \eqref{eq:f_0n-exmp4}, we get that $\al_1 = 1/5, \al_2 = 1/7$ which indicates that the asymptotics is driven by the first-term. In this case, Wright's result and the true asymptotic distribution are the same. But in the second example of $f_{0,n}(x) = {x}/{n^{5/16}} + {x^3}/{6}$ in \eqref{eq:f_0n-exmp4-modified}, we get that $\al_1= 1/8, \al_2 = 1/7$. In this case, although the first term is non-zero, the second term is the one that dictates the asymptotics, and Wright's result fails to capture this true asymptotic behavior. This is also clear from the difference in the proposed asymptotic distribution to the actual one. Wright's result suggests that the asymptotic distribution is $\slGCM[B(t) + |t|^2/8:t\in\mathbb{R}](0)$ whereas in reality it is $\slGCM[B(t) + |t|^4/64 :t\in\mathbb{R}](0)$. More interestingly, consider the dual-term driven case, for example as follows,
\begin{align}
    f_{0,n}(x) = x/n^{2/7} + x^3/6,
\end{align}
 where $\al_1 = \al_2 = 1/7$. Although Wright's result captures the true asymptotic convergence rate of $n^{3/7}$, the asymptotic distribution is $\slGCM[B(t) + |t|^2/8 + |t|^4/64](0)$, which is entirely outside the scope of Wright's framework.
 \paragraph{A coverage experiment highlighting the failure of Wright-type calibration.}
We complement the above discussion with a finite-sample coverage experiment in the second-term driven regime (row~2 of Table~\ref{table:exmp-4-gen-cases}), where Wright's rate and limiting distribution are both misspecified.
Specifically, we generate data from \eqref{eq:data-gen-process} with $x_0=0$ and $f_{0,n}$ given by \eqref{eq:f_0n-exmp4-gen}, and fix $(\alpha_1,\alpha_2)=(0.10,0.60)$ so that $\alpha_2>\alpha_1$. We compare two $(1-\alpha)$ confidence intervals with $\alpha=0.05$:
(i) a Wright-type asymptotic confidence interval obtained by combining the quadratic-drift limit $\slGCM$ with Wright's scaling choice $s_n=n^{\alpha_1}$, and
(ii) HulC, constructed by splitting the sample into $B_\alpha$ batches (with $B_\alpha$ chosen as in Section~\ref{sec:simulation}) and taking the range of the batch-wise isotonic estimates at $0$.
Figure~\ref{fig:exp3_coverage_width} reports the empirical coverage and average interval length.
As predicted by the theory above, the Wright-type interval substantially undercovers in this regime because the correct scaling is governed by $\alpha_2$ rather than $\alpha_1$, whereas HulC remains stable across $n$.
\begin{figure}[!h]
    \centering
    \includegraphics[width=1\textwidth]{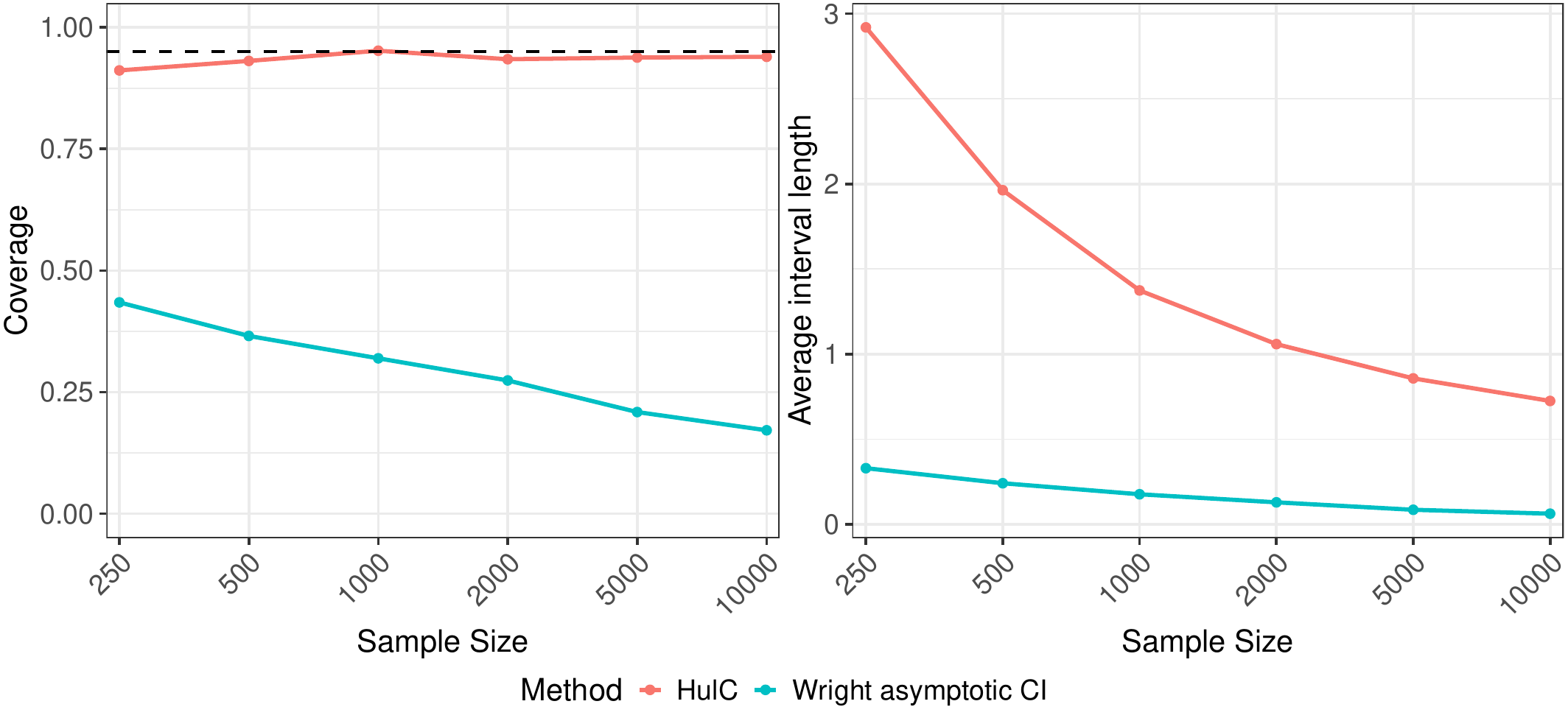}
    \caption{Empirical coverage and average width for Wright-type asymptotic confidence intervals and HulC at $x_0=0$, under the model \eqref{eq:data-gen-process} with $f_{0,n}$ as in \eqref{eq:f_0n-exmp4-gen} and $(\alpha_1,\alpha_2)=(0.10,0.60)$. Each point is based on $2000$ Monte Carlo replications, with $n\in\{250,500,1000,2000,5000,10000\}$ and nominal level $1-\alpha=0.95$ (dashed line).}
    \label{fig:exp3_coverage_width}
\end{figure}
\end{remark}
\subsection{Comments on the Assumptions and Proof of Theorem~\ref{thm:asymptotic-distribution-of-LSE}}\label{subsec:proof-main-result}
The following propositions consider a few implications of the assumptions, as well as their verification. Proposition~\ref{prop:verification-of-A3} shows that one can verify assumption~\ref{assump:continuity-of-f_{0,n}} by first verifying that the limit in~\eqref{eq:assumption-psi-limit} exists when $c_n = c$ for all $n\ge1$ and that $\psi(\cdot)$ is a continuous function.
\begin{proposition}\label{prop:verification-of-A3}
Suppose there exists a sequence $\{s_n\}_{n\ge1}$ satisfying $s_n\to\infty$ and $s_n/n\to0$ such that for every $c\in\mathbb{R}$, the limit of $\sqrt{n/s_n}(f_{0,n}(x_0 + c/s_n) - f_{0,n}(x_0))$ as $n\to\infty$ exists and equals $\psi(c)$. If $\psi(\cdot)$ is a continuous function, then under~\ref{assump:data-model-assumption} for any $c\in\mathbb{R}$ and any sequence $\{c_n\}_{n\ge1}$ satisfying $c_n\to c$ as $n\to\infty$,
\[
\psi(c) = \lim_{n\to\infty}\,\sqrt{\frac{n}{s_n}}\left(f_{0,n}\left(x_0 + \frac{c_n}{s_n}\right) - f_{0,n}(x_0)\right).
\]
\end{proposition}
\begin{proof}
Define
\[
g_n(c) := \sqrt{n/s_n}(f_{0,n}(x_0 + c/s_n) - f_{0,n}(x_0)).
\]
Because $f_{0,n}(\cdot)$ is a non-decreasing function, $g_n(\cdot)$ is a non-decreasing function for every $n\ge1$.
The result follows from the fact that pointwise convergence of a sequence of monotone functions to a continuous function implies continuous convergence; see Section 0.1 of~\cite{resnick2008extreme} and Proposition 2.1 of~\cite{resnick2007heavy} for details.
\end{proof}
The following proposition (proved in \ref{appsec:proof-of-prop-assumptions-implications}) provides various useful implications of our assumptions~\ref{assump:data-model-assumption} and~\ref{assump:continuity-of-f_{0,n}}. In addition, part 2 of Proposition~\ref{prop:assumptions-implications} proves that under~\ref{assump:continuity-of-f_{0,n}} the rate of convergence of the isotonic LSE is uniquely defined.
\begin{proposition}\label{prop:assumptions-implications}
Under assumptions~\ref{assump:data-model-assumption} and~\ref{assump:continuity-of-f_{0,n}}, $\psi(\cdot)$ and the sequence $\{s_n\}$ satisfy the following properties:
\begin{enumerate}
\item $\psi(\cdot)$ is a non-decreasing function and both $\psi(c), \Psi(c)/|c|$ are bounded away from zero as $|c|\to\infty$.
\item Suppose two sequences $\{s_n\}$ and $\{s_n'\}$ satisfy \ref{assump:continuity-of-f_{0,n}} (with possibly different limit functions $\psi$ and $\tilde{\psi}$), i.e., for any sequence $\{c_n\}$ converging to $c\in\mathbb{R}\backslash\{0\}$,
\[
\psi(c)=\displaystyle\lim_{n\to\infty}\sqrt{\frac{n}{s_n}}\left(f_{0,n}\left(x_0+\frac{c_n}{s_n}\right)-f_{0,n}(x_0)\right)~~\text{exists }\forall~c\in\mathbb{R}\setminus\{0\}\text{ and also, }
\]
\[
\tilde{\psi}(c)=\displaystyle\lim_{n\to\infty}\sqrt{\frac{n}{s'_n}}\left(f_{0,n}\left(x_0+\frac{c_n}{s'_n}\right)-f_{0,n}(x_0)\right) ~~\text{exists }\forall~c\in\mathbb{R}\setminus\{0\},
\]
and both $\psi, \tilde{\psi}$ satisfy properties listed in~\ref{assump:continuity-of-f_{0,n}},
then,
\[\lim_{n\to\infty}s_n/s_n'=a\in(0,~\infty) \text{ exists and }\tilde{\psi}(c)=\sqrt{a}\psi(ac)~~\forall~c\in\mathbb{R}\]
Moreover,
\[\tilde{\psi}\equiv\psi~\text{i.e., }\tilde{\psi}(c)=\psi(c)~\forall~c\in\mathbb{R}\iff a=1\]
Thus, if the two sequences $\{s_n\}$ and $\{s_n'\}$ yield the same limiting function $\psi$, then the limit of their ratio is 1. Thus, the rate is unique in the sense that any two such rates are asymptotically equivalent.
\item If $\psi(\cdot)$ is finite on $\mathbb{R}\setminus\{0\}$, then the limit
\[
\lim_{n\to\infty}\sqrt{\frac{n}{s_n}}\int_{0}^t c\left(f_{0,n}\left(x_0+x\frac{c_n}{s_n}\right)-f_{0,n}(x_0)\right)dx,
\]
exists and equals $\Psi(ct)$ for all $t\ge 0.$ Similarly, the limit
\[
\lim_{n\to\infty}\sqrt{\frac{n}{s_n}}\int_{t}^0 c\left(f_{0,n}\left(x_0+x\frac{c_n}{s_n}\right)-f_{0,n}(x_0)\right)dx,
\] also exists and equals $-\Psi(ct)$ for all $t < 0.$
\item Suppose $\psi(\cdot)$ is finite on $\mathbb{R}\setminus\{0\}$. Define
\begin{align*}
\Gamma_1(c) &:=
\begin{cases}
\psi(c) - \frac{\Psi(c)}{c}, &\mbox{if }c > 0,\\
-\psi(c) + \frac{\Psi(c)}{c}, &\mbox{if }c < 0,
\end{cases}\\
\Gamma_2(c) &:=
\begin{cases}
-\psi(c) + \frac{\Psi(2c) - \Psi(c)}{c}, &\mbox{if }c > 0,\\
\psi(c) - \frac{\Psi(2c) - \Psi(c)}{c}, &\mbox{if }c < 0.
\end{cases}
\end{align*}
Then $\Gamma_1(c), \Gamma_2(c)$ are non-negative for all $c\in\mathbb{R}$,
\[
\limsup_{|c|\to\infty}|c|\Gamma_1(c) = \limsup_{|c|\to\infty} |c|\Gamma_2(c) = \infty.
\]
\end{enumerate}
\end{proposition}
\begin{remark}[Interplay of rate and limiting distribution]
Note that, from part 2 of Proposition \ref{prop:assumptions-implications}, we get:
\[
\tilde{\Psi}(t) = \mathbf{1}\{t \ge 0\}\int_0^t \tilde{\psi}(s)ds - \mathbf{1}\{t < 0\}\int_{t}^{0} \tilde{\psi}(s)ds=\frac{1}{\sqrt{a}}\Psi(at), \quad\mbox{for all}\quad t\in\mathbb{R}.
\]
Using the above and Remark \ref{rem:representation-of-limiting-distribution} one can conclude that with $\tilde{\mathcal{B}}_0(t) := B(t) + \tilde{\Psi}(h_0t/\sigma_0^2)$,
\begin{align*}
&\sqrt{\frac{n}{s'_n}}\frac{h_0}{\sigma^2_0}(\hat{f}_n(x_0)-f_{0,n}(x_0))~\overset{d}{\to}~\mathbf{slGCM}[\tilde{\mathcal{B}}_0(t):t\in\mathbb{R}](0)\\
\iff\quad &\sqrt{\frac{n}{s_n}}\frac{h_0}{\sigma^2_0}(\hat{f}_n(x_0)-f_{0,n}(x_0))~\overset{d}{\to}~\mathbf{slGCM}[\mathcal{B}_0(t):t\in\mathbb{R}](0).
\end{align*}
Thus, it does not matter whether we choose $s_n$ or $s_n'$ for the calculation of the limiting distribution as it is only a different representation of the same limit result.
\end{remark}
\subsubsection{Proof of Theorem~\ref{thm:asymptotic-distribution-of-LSE}}
We imitate the proof as in \cite{wright1981asymptotic} and slightly modify some key steps. For the sake of readability, we provide a self-contained proof. To summarize, Lemmas \ref{lem:smoothness-of-H}, \ref{lem:unequalness-of-fn-star-and-fhat} and \ref{lem:brownian-motion} are exactly as in~\cite{wright1981asymptotic}, though we give detailed arguments which are not present in the original paper.
The monotone regression estimator $\hat{f}_n$ has a closed form expression as follows,
\[
\hat{f}_n(x)=\max_{X_{i:n}\leq x}\min_{X_{j:n}\geq x} \text{Av}_n([X_{i:n},X_{j:n}]),
\]
where $\text{Av}_n((a,b])=\frac{1}{|\{k: X_{k:n}\in(a,b]\}|}\sum_{k: X_{k:n}\in(a,b]} Y_{[k:n]}.$
Fix $c>0$. Since $H'_n(x_0)>0$ and $H'_n$ is continuous in a neighborhood $\mathcal{N}(x_0)$ of $x_0$, $H'_n(x)>0$ in an open neighborhood $\mathcal{N}'(x_0)\subseteq \mathcal{N}(x_0)$ of $x_0$. Hence, for large enough $n$, one can choose $\alpha_l(n)>0$ and $\alpha_u(n)>0$ such that:
\[H(x_0)-H(x_0-\alpha_l(n))=H(x_0+\alpha_u(n))-H(x_0)=\frac{2c}{s_n}.\]
For notational simplicity, we shall write $\alpha_l(n)$ and $\alpha_u(n)$ as $\alpha_l$ and $\alpha_u$ respectively, keeping in mind that these factors depend on $n$.
Define the monotone regression estimator on the restricted set $(x_0-\alpha_l,~x_0+\alpha_u)$ as:
\[f_{n,c}^{*}(x_0)=\max_{i:x_0-\alpha_l<X_{i:n}\le x_0}~\min_{j:x_0\le X_{j:n}<x_0+\alpha_u}\text{Av}_n([X_{i:n},~X_{j:n}]).\]
The proof starts by showing that $f_{n,c}^*(x_0)$ and $\hat{f}_n(x_0)$ are equal with probability converging to one as $n\to\infty$ and $c\to\infty$ (in this order); this is done in Lemma~\ref{lem:prob-of-unequalness-goes-to-zero}. Hence, by Lemma 4.2 of~\cite{rao1969estimation}, it suffices to study the limiting distribution of $f_{n,c}^*(x_0)$ for any fixed $c > 0$. For this, we first start by noting that $f_{n,c}^*(\cdot)$ is the isotonic LSE for the data $\{(X_i, Y_i):\, 1\le i\le n, X_{i}\in(x_0 - \alpha_l, x_0 + \alpha_u)\}$. Let $\gamma_n$ denote the number of observations with $X_i\in(x_0-\alpha_l,~x_0+\alpha_u)$. Let the observation points in $(x_0-\alpha_l,~x_0+\alpha_u)$ be $\{(X'_{(1)},Y'_{[1:\gamma_n]}),\ldots,(X'_{(\gamma_n)},~Y'_{[\gamma_n:\gamma_n]})\}$. Following Lemma 2.1 of~\cite{groeneboom2001estimation}, $f_{n,c}^*(X_{k:n}')$ can be obtained from the CUSUM diagram as the slope from the left of the greatest convex minorant (GCM) of $(0, 0), (k/\gamma_n, \sum_{l=1}^{k} Y_{[l:\gamma_n]}'/\gamma_n), 1\le k\le \gamma_n$ at $k/\gamma_n$. Note that the GCM does not change if we consider the linear interpolation of the points $(0, 0), (k/\gamma_n, \sum_{l=1}^{k} Y_{[l:\gamma_n]}'), 1\le l\le \gamma_n$. Let us call that linear interpolation process $\tilde{U}_n(\cdot)$ on $[0, 1]$. The idea now to prove the limiting distribution of $f_{n,c}^*$ is to show that $\tilde{U}_n(\cdot)$ converges as a stochastic process to a drifted Brownian motion. For proper scaling, we work with the linear interpolation of scaled points.
Let $D_0$ be any fixed positive real number. Define, $t_{n0}=0$ and
\[t_{nk}=2cD_0\cdot\frac{k}{\gamma_n}~~k\in\{1, 2, \ldots, \gamma_n\}.\]
 Define a process on $[0,~2cD_0]$ by $U_n(0)=0$ and
\[
U_n(t_{nk})=2cD_0\cdot\frac{1}{\gamma_n}\sum_{l=1}^k Y'_{[l:\gamma_n]},\quad k\in\{1, 2, \ldots, \gamma_n\}.
\]
Define the process $t\mapsto U_n(t)$ on $[0, 2cD_0]$ as the linear interpolation of $\{(t_{nk}, U_n(t_{nk})): 1\le k\le \gamma_n\}$.
This is, mathematically, given by
\begin{equation}\label{eq:linear-interpolation-scaled}
U_n(t) := \frac{2cD_0}{\gamma_n}\left(\sum_{l=1}^{\lfloor \gamma_nt/(2cD_0)\rfloor} Y_{[l:\gamma_n]}' + \left[\frac{\gamma_nt}{2cD_0}-\left\lfloor \frac{\gamma_nt}{2cD_0}\right\rfloor\right]Y_{[\lfloor \gamma_nt/(2cD_0)\rfloor+1:\gamma_n]}'\right).
\end{equation}
As noted in Eq. (6) of~\cite{wright1981asymptotic},
\begin{equation}\label{eq:isotone-slGCM}
    \begin{split}
        &\text{\textbf{slGCM}}(U_n(t) - f_{0,n}(x_0)(t):\, t\in[0, 2cD_0])(t_{n(j_n - 1)})\\
        &\le f_{n,c}^*(x_0) - f_{0,n}(x_0)\\
        &\le \text{\textbf{slGCM}}(U_n(t) - f_{0,n}(x_0)(t):\, t\in[0, 2cD_0])(t_{nj_n}),
    \end{split}
\end{equation}
where $j_n \in \{2, 3, \ldots, \gamma_n\}$ is such that $X_{(j_n - 1)}' < x_0 \le X_{(j_n)}'$.
Because the GCM is scale equivariant, the GCM of $s\mapsto \tilde{U}_n(s)$ is $2cD_0$ times the GCM of $s = t/2cD_0 \mapsto U_n(2cD_0s)$. Hence, the slope from the left of the GCM of $s\mapsto \tilde{U}_n(s)$ at $s = k/\gamma_n$ is the same as the slope from the left of the GCM of $t\mapsto U_n(t)$ at $t = 2cD_0k/\gamma_n$, for any $c, D_0 > 0$. Lemma~\ref{lem:brownian-motion} shows that $t\mapsto U_n(t) - \mathbb{E}[U_n(t)|\mathcal{X}]$ at a rate of $\sqrt{n/s_n}$ converges as a process to the standard Brownian motion on $[0, 2cD_0]$. Lemma~\ref{lem:deterministic-part-final} proves that $t\mapsto \mathbb{E}[U_n(t)|\mathcal{X}] - f_{0,n}(x_0)t$ uniformly converges at a rate of $\sqrt{n/s_n}$ to a scaled version of $\Psi(\cdot)$.
Thus, from Lemmas \ref{lem:brownian-motion} and \ref{lem:deterministic-part-final}, we conclude that for any $c > 0$, as $n\to\infty$,
\begin{align*}
\Upsilon_n(t)&:=\sqrt{\frac{2n/s_n}{D_0\sigma^2_0}}(U_n(t)-f_{0,n}(x_0)t)\\
&\overset{d}{\to} B(t)+\sqrt{\frac{2}{D_0\sigma^2_0}}\frac{D_0h_0}{2}\cdot\left(-\Psi\left(-\frac{2c}{h_0}\right)+\Psi\left(-\frac{2c}{h_0}\left(1-\frac{t}{cD_0}\right)\right)\right),
\end{align*}
as processes on the domain $t\in[0, 2cD_0]$.
As constants do not influence the slope, and making the change of variable $s=t-cD_0$,
\begin{align*}
    &\text{\textbf{slGCM}}(\Upsilon_n(t):t\in[0,2cD])(t_{nj(n)})\\
    &=\text{\textbf{slGCM}}(\Upsilon_n(s+cD_0):s\in[-cD_0,~cD_0])(t_{nj(n)}-cD_0)\\
    &\overset{d}{\to} \text{\textbf{slGCM}}\left(B(s)+\sqrt{\frac{D_0}{2}}\dfrac{h_0}{\sigma_0}\Psi\left(\dfrac{2s}{h_0D_0}\right):\, s\in[-cD_0, cD_0]\right)(0),
\end{align*}
following the same argument given after Eq. (10) of~\cite{wright1981asymptotic}; also, see Section 4 of~\cite{leurgans1982asymptotic} and Proof of Theorem 2.1 of~\cite{banerjee2007likelihood}.
Thus, we get the form of the limiting distribution at least on $[-cD_0,~cD_0]$. Let
\[
g(s)=\sqrt{\frac{D_0}{2}}\dfrac{h_0}{\sigma_0}\Psi\left(\dfrac{2s}{h_0D_0}\right).
\]
To show that this quantity converges to the limit on $\mathbb{R}$, we modify Lemma 6.2 in~\cite{rao1969estimation}. The proof of the lemma implies that it is enough to show that both $\mathbb{P}(B(-s)> g(-s) \text{ for some }s>cD_0)$ and $\mathbb{P}(B(s)> g(s) \text{ for some }s>cD_0)\to 0$ as $c\to\infty$, where $B(s)$ denote the two-sided Wiener--L{\'e}vy process on $\mathbb{R}$. This is immediate by noting that: $B(s)/s\to0$ a.s. and both $\Psi(s)/s$ and $\Psi(-s)/s$ are bounded away from zero as $s\to\infty$, as shown in part 1 of Proposition~\ref{prop:assumptions-implications}. Hence, we conclude that:
\[
\sqrt{\frac{2n/s_n}{D_0\sigma^2_0}}(\hat{f}_n(x_0)-f_{0,n}(x_0))\overset{d}{\to}\text{\textbf{slGCM}}[\mathcal{B}(t):\,t\in\mathbb{R}](0),
\]
where $\mathcal{B}(t)=B(t)+\frac{h_0}{\sigma_0}\sqrt{\frac{D_0}{2}}\Psi\left(\frac{2t}{h_0D_0}\right)$.
Now, due to Remark \ref{rem:representation-of-limiting-distribution}, any value of $D_0$ would work.
 Choose $D_0$ such that the coefficient of $\Psi$ becomes 1, i.e., set $D_0={2\sigma^2_0}/{h^2_0}$.
 As a result, the limiting distribution becomes
\[
\frac{h_0}{\sigma^2_0}\sqrt{\frac{n}{s_n}}(\hat{f}_n(x_0)-f_{0,n}(x_0))\overset{d}{\to}\text{\textbf{slGCM}}[\mathcal{B}_0(t)](0),
\]
where $\mathcal{B}_0(t)=B(t)+\Psi\left(\frac{h_0t}{\sigma^2_0}\right)$.
This shows the result.
\section{Inference for Isotonic Regression}\label{sec:inference-HulC}
Inference for isotonic regression is a difficult problem for several reasons, even if we restrict ourselves to only the previously known asymptotic limiting distributions in~\cite{wright1981asymptotic}. Firstly, the local polynomial exponent $\beta$ is unknown in practice. Secondly, even if $\beta$ is known, the limiting distribution involves other non-parametric quantities such as the density of covariates at $x_0$ and the conditional variance at $x_0$. The estimation of these non-parametric quantities can be significantly challenging and require more assumptions on the data-generating process for ``good'' estimation. With the richer asymptotic theory as implied by Theorem~\ref{thm:asymptotic-distribution-of-LSE} the problem only became more difficult because instead of a single parameter $\beta$, we now have a function $\psi(\cdot)$ that is unknown.
As mentioned in the introduction, only three general methods of inference exist, of which two methods are designed for the result of~\cite{wright1981asymptotic}. Subsampling with an estimated rate of convergence~\citep{bertail1999subsampling} is the only general method that yields asymptotically valid confidence intervals under the assumptions of Theorem~\ref{thm:asymptotic-distribution-of-LSE} if $f_{0,n}(x_0)$ is unchanged as $n\to\infty$. However, it should be clarified here that subsampling may not maintain uniform validity under the triangular array setting that allows $f_{0,n}(x_0)$ to change with sample size.
\subsection{Asymptotically Valid Inference using HulC}
\label{sec:intro-hulc}
The general method of HulC allows for the construction of asymptotically valid confidence intervals by splitting the data into a fixed number of non-overlapping batches and computing the estimator on each batch. In this method, nothing more than the computation of the estimator is needed. The detailed steps to construct a $(1-\alpha)$ confidence interval are as follows:
\begin{enumerate}
    \item Generate $U$, a standard uniform random variable. Set $B_{\alpha} = \lceil\log_2(2/\alpha)\rceil$ and
    \[
    B^* = \begin{cases}
    B_{\alpha} - 1, &\mbox{if } U \le (\alpha - 2^{B_{\alpha}})/(2^{B_{\alpha} - 1} - 2^{B_{\alpha}}),\\
    B_{\alpha}, &\mbox{otherwise}.
    \end{cases}
    \]
    \item Split the data $(X_i, Y_i), 1\le i\le n$ randomly into $B^*$ non-overlapping batches of approximately equal sizes and compute the isotonic LSE on each batch to obtain $\widehat{f}_{n,\alpha}^{(j)}(x_0), 1\le j\le B^*$.
    \item Return the confidence interval
    \begin{equation}\label{eq:HulC-isotonic}
    \widehat{\mathrm{CI}}_{n,\alpha}(x_0) = \left[\min_{1\le j\le B^*}\widehat{f}_{n,\alpha}^{(j)}(x_0),\,\max_{1\le j\le B^*}\widehat{f}_{n,\alpha}^{(j)}(x_0)\right].
    \end{equation}
\end{enumerate}
Following~\cite{kuchibhotla2021hulc}, we define the median bias of $\widehat{f}_n(x_0)$ for $f_{0,n}(x_0)$, when the underlying data is obtained from distribution $P_n$ as
\begin{equation}
\label{eq:median-bias-definition-whole}
\medbias_{P_n}(\widehat{f}_{n}(x_0), f_{0,n}(x_0)) := \left(\frac{1}{2} - \min\left\{\mathbb{P}_{P_n}(\widehat{f}_{n}(x_0) \le f_{0,n}(x_0)), \mathbb{P}_{P_n}(\widehat{f}_n(x_0) \ge f_{0,n}(x_0))\right\}\right)_+.
\end{equation}
We subscript the probability by $P_n$ to signify that the probability is computed under the distribution $P_n$ of the underlying data.
In fact, Theorem 2 of~\cite{kuchibhotla2021hulc} proves that
\begin{equation}\label{eq:HulC-miscoverage-bound}
\mathbb{P}_{P_n}\left(f_{0,n}(x_0) \notin \widehat{\mathrm{CI}}_{n,\alpha}(x_0)\right) ~\le~ \alpha\left(1 + 2(B_{\alpha}\Delta)^2e^{2B_{\alpha}\Delta}\right),
\end{equation}
where
\begin{equation}\label{eq:median-bias-definition}
\Delta \equiv \Delta_{n,\alpha} := \max_{1\le j\le B_{\alpha}}\medbias_{P_n}(\widehat{f}_{n,\alpha}^{(j)}(x_0),\,f_{0,n}(x_0)).
\end{equation}
Because $B_{\alpha}$ is independent of $n$, $\Delta_{n,\alpha}$ converging to zero as $n\to\infty$ implies that the HulC confidence interval~\eqref{eq:HulC-isotonic} has an asymptotic miscoverage error less than or equal to $\alpha$.
The rate at which $\Delta_{n,\alpha}$ converges to zero depends on how fast the finite sample distribution of the isotonic LSE reaches the limiting distribution. Some results in this direction have been obtained recently in~\cite{han2022berry} in the pointwise setting (i.e., same distribution as $n$ changes).
\paragraph{Effect of data splitting}
In the HulC construction, each $\widehat f^{(j)}_{n,\alpha}(x_0)$ is the isotonic LSE computed on a batch of size of order $n$ (more precisely, $\lfloor n/B^*\rfloor$ or $\lceil n/B^*\rceil$). Hence, the asymptotic distribution of each batch estimator is obtained by applying Corollary~\ref{cor:triangular-asymptotic-distribution-of-LSE} with $k_n=\lfloor n/B^*\rfloor$. Note that for the examples discussed in Section \ref{sec:applications-asymptotic-distribution-of-LSE}, $\lim_{n\to\infty}s_{k_n}/s_n = (\lim_{n\to\infty} k_n/n)^{1/(2\theta+1)} = (1/B^*)^{1/(2\theta+1)}$, which is a constant in $(0,\infty)$. In particular, if we can show that for each fixed $j$,
\[
\sqrt{\frac{k_n}{s_{k_n}}}\frac{h_0}{\sigma_0^2}\Big(\widehat f^{(j)}_{n,\alpha}(x_0)-f_{0,n}(x_0)\Big)~
\text{converges to a symmetric distribution},
\]
then, since the above limit has median zero,
\[
\medbias_{P_n}\!\left(\widehat f^{(j)}_{n,\alpha}(x_0),\,f_{0,n}(x_0)\right)\to 0
\quad\text{and hence}\quad
\Delta_{n,\alpha}\to 0,
\]
which together with~\eqref{eq:HulC-miscoverage-bound} yields asymptotic validity of $\widehat{\mathrm{CI}}_{n,\alpha}(x_0)$.\\

\noindent Consider the premise of Assumptions A and B in \cite{han2022berry} which are a stronger version of Wright's assumptions described in Example~\ref{exmp:Wright-1981}.
\begin{assumption}
    \label{assump:han-kato}
    Suppose $x_0\in(0,1)$ and let $X_1,X_2,\ldots,X_n$ be i.i.d. with distribution $H$ which has a density $h$ on $[0,1]$ that is continuous around $x_0$ and is bounded away from $0$ on $[0, 1]$. Further assume that for some $1<\beta<\infty$ and for all $x \in [0, 1]$,
    \[h(x) - h(x_0) = B(x-x_0)^\beta + R_h((x-x_0)^\beta)\]
    Let $\theta,\theta^\ast\in(0,\infty)$ be the first and second non-vanishing derivatives of $f_{0,n}\equiv f_0$ (fixed function) and we assume the following Taylor expansion holds for the function $f_0$, for all $x\in[0,1]$:
    \[f_0(x) = f_0(x_0) + A(x-x_0)^\theta\sign(x-x_0) +A^\ast(x-x_0)^{\theta^\ast}\mathrm{sgn}(x-x_0) + R((x-x_0)^{\theta^\ast})\]
    Here, $A, A^\ast, B$ are constants and $R, R_h:\R\to\R$ are functions such that $R(0)=R_h(0)=0$ and $\max\{R_h(\epsilon), R(\epsilon)\} = o(\epsilon),$ as $\epsilon\to 0$.
\end{assumption}
\noindent Under assumption \ref{assump:han-kato}, Theorem 2.2 of~\cite{han2022berry} proves that
\begin{align*}
&\sup_{u\in\mathbb{R}}\left|\mathbb{P}\left(n^{\theta/(2\theta+1)}c_\theta(\widehat{f}_{n}(x_0) - f_{0,n}(x_0)) \le u\right) - \mathbb{P}\left(\slGCM[\mathcal{B}(t):\,t\in\mathbb{R}](0) \le u\right)\right|\\
&\le k_1\max\left\{n^{-\theta/(2\theta + 1)}\log n,\,n^{-(\theta^* - \theta)/(2\theta+1)}, n^{-\beta/(2\theta + 1)}\right\}(\log n)^{k_2},
\end{align*}
where $k_1, k_2, c_\theta$ are constants independent of $n$ and $\mathcal{B}(t) = B(t) + |t|^{\alpha+1}$. Note that the same bound also holds for $\widehat{f}_{n,\alpha}^{(j)}(x_0)$ because the sample size for $\widehat{f}_{n,\alpha}^{(j)}$ is of the order of $n$ (for any fixed $\alpha$).
This implies that
\begin{align}
\label{eq:med-bias-berry-essen}
\mbox{Med-Bias}_{P}(\widehat{f}_n(x_0),\,f_0(x_0)) &\le ~ \mbox{Med-Bias}(\slGCM[\mathcal{B}(t):t\in\mathbb{R}](0),\,0) \\
&+ k_1\max\left\{n^{-\theta/(2\theta + 1)}\log n,\,n^{-(\theta^* - \theta)/(2\theta+1)}, n^{-\beta/(2\theta + 1)}\right\}(\log n)^{k_2}.
\end{align}
Hence, if
\begin{equation}\label{eq:limit-median-bias-zero}
\mbox{Med-Bias}(\slGCM[\mathcal{B}(t):t\in\mathbb{R}](0),\,0) = 0,
\end{equation}
then $\Delta_{n,\alpha}\to 0$ as $n\to\infty$.
Indeed, assuming~\eqref{eq:limit-median-bias-zero}, $\theta^* > 2\theta$, and $\beta > \theta$, we get that $\Delta_{n,\alpha} = O((n/B_{\alpha})^{-\theta/(2\theta + 1)})$ as $n\to\infty$ and hence,~\eqref{eq:HulC-miscoverage-bound} coupled with Corollary \ref{cor:triangular-asymptotic-distribution-of-LSE} yields
\[
\mathbb{P}\left(f_{0,n}(x_0) \notin \widehat{\mathrm{CI}}_{n,\alpha}(x_0)\right) ~\le~ \alpha\left(1 + O(1)B_{\alpha}^2(n/B_{\alpha})^{-2\theta/(2\theta+1)}\right),\quad\mbox{for all large enough }n.
\]
Therefore, the main assumption to verify to ensure validity of HulC confidence intervals is~\eqref{eq:limit-median-bias-zero}. It is easy to see that distributions symmetric around zero satisfy~\eqref{eq:limit-median-bias-zero}.
Motivated by this, we study general conditions in section \ref{subsec:symmetry-condtions} under which this symmetry property holds.
Under the setting of Theorem~\ref{thm:asymptotic-distribution-of-LSE}, Berry--Esseen bounds are non-existent and are of great interest in understanding the miscoverage error of HulC-type confidence intervals.
\paragraph{Width of Confidence Intervals.} Under the assumptions of Theorem~\ref{thm:asymptotic-distribution-of-LSE}, we know that $\widehat{f}_{n,\alpha}^{(j)}(x_0) - f_{0,n}(x_0), 1\le j\le B_{\alpha}$ all converge at a $\sqrt{n/(B_{\alpha}s_n)}$
rate to the limiting distribution. Hence, the maximum and the minimum of these estimators also converge at the same rate of $\sqrt{n/(B_{\alpha}s_n)}.$ This implies that
\[
\mbox{Width}(\widehat{\mathrm{CI}}_{n,\alpha}(x_0)) = O_p\left(\sqrt{\frac{s_n}{n}}\right)\quad\mbox{as}\quad n\to\infty.
\]
Because the confidence intervals attain this width without the knowledge of $s_n$ or $\psi$, they are adaptive and asymptotically valid confidence intervals. Construction of optimal adaptive confidence bands for monotone regression function is available in~\cite{dumbgen2003optimal}; also, see~\cite{bellec2018sharp} and~\cite{yang2019contraction}.
\subsection{Conditions for the Symmetry of Limiting Distributions}
\label{subsec:symmetry-condtions}
In this section, we show that the limiting distribution of the isotonic LSE is a continuous symmetric distribution if $\psi(\cdot)$ is an odd function (i.e., $\psi(-c) = -\psi(c)$, or equivalently, $\Psi(-c) = \Psi(c)$ for $c\ge0$). This allows for the validity of HulC confidence intervals~\cite{kuchibhotla2021hulc}. Some of the salient features of HulC are that (1) the methodology is completely data-driven when the limiting distribution has a zero median; (2) the miscoverage error converges to $\alpha$ even in the relative error, so that even with union bound over a large number of $x_0$.
In fact, we prove a general result that implies the symmetry of the distribution of minimizers of drifted Brownian motion, $\mathcal{B}(t)$, which implies the result for the subclass of limiting distributions of isotonic LSE. In passing, we note that the derivation of a necessary and sufficient condition for symmetry of the minimizers of a general stochastic process is an interesting open problem.
The following result studies the behavior of the slope from the left of the greatest convex minorant as well as the slope from the left of the least concave minorant of a drifted symmetric stochastic process with an arbitrary even drift function.
\begin{theorem}\label{thm:general-drifted-brownian-motion}
Suppose $d:\mathbb{R}\to\mathbb{R}$ be a non-stochastic continuous function such that $d(t)/|t|$ is bounded away from zero as $|t|\to\infty$ and
\begin{enumerate}[label=(AS\arabic*)]
    \item \label{assump:symmetry-even}$d$ is even, i.e, $d(t)=d(-t)$ for all $t\in\mathbb{R}$.
    \item \label{assump:symmetry:continuity}$|d(t_0+t)-d(t_0)|/\sqrt{t}\to 0$ as $t\downarrow 0~~\forall~t_0\in\mathbb{R}$.
\end{enumerate}
Consider two drifted Brownian motions $\mathcal{B}_1(t):=B_1(t)+d(t)$ and $\mathcal{B}_2(t):=B_2(t)-d(t)$, where $B_1(\cdot), B_2(\cdot)$ are independently generated standard two-sided Brownian motions on $\mathbb{R}$.
Fix any $t_0\in\mathbb{R}$. Let $\slGCM[\mathcal{B}_1](t_0)$ denote the slope from the left at $t_0$ of the greatest convex minorant of $\mathcal{B}_1(\cdot)$. Similarly, let $\slLCM[\mathcal{B}_2](t_0)$ denote the slope from the left at $t_0$ of the least concave majorant of $\mathcal{B}_2(\cdot)$, evaluated at $t_0$. Then, $\slGCM[\mathcal{B}_1](t_0), \slLCM[\mathcal{B}_2](-t_0)$ have well-defined representations in terms of the argmin functional and are continuous random variables. Also,
\[
\slGCM[\mathcal{B}_1](t_0)~\overset{d}{=}~\slLCM[\mathcal{B}_2](-t_0).
\]
In particular, for $t_0=0$,
\[
\slGCM[\mathcal{B}_1](0)~\overset{d}{=}~\slLCM[\mathcal{B}_2](0)~\overset{d}{=}-\slGCM[\mathcal{B}_1](0).
\]
\end{theorem}
The first assumption \ref{assump:symmetry-even} is needed to ensure symmetry while \ref{assump:symmetry:continuity}  is needed to ensure that the distribution functions are continuous. In fact, \ref{assump:symmetry:continuity} has been borrowed from Lemma SA-1 in \cite{cattaneo2023bootstrap}, which ensures continuity of the distribution function. A detailed proof is presented in~\ref{appsec:proof-of-thm-general-drifted-brownian-motion}. The proof is obtained by combining the switching relations~\citep[Lemma 3.2, Chap. 3]{groeneboom2014nonparametric} and the symmetry properties of the Brownian motion. The uniqueness and continuous distribution parts follow from the results of~\cite{kim1990cube} and~\cite{cattaneo2023bootstrap}, respectively.
It is easy to check that in our setup, $\Psi$ as defined in \eqref{eq:definition-Psi} is continuous and already satisfies \ref{assump:symmetry:continuity} whenever $\psi$ is finite on $\mathbb{R}$. Because the limiting distributions of the isotonic LSE are slopes from the left of the greatest convex minorants of drifted Brownian motions, Theorem~\ref{thm:general-drifted-brownian-motion} implies that the limiting distributions are continuous with respect to the Lebesgue measure and are symmetric around zero, whenever $\Psi(-t) = \Psi(t)$ for all $t\in\mathbb{R}$
. Isotonic LSE is only one of the many applications of Theorem~\ref{thm:general-drifted-brownian-motion}. Several non- and semi-parametric models involving monotonicity constraints have limiting distributions that are expressed in terms of the slope from the left of a GCM. For example, isotonic $M$-estimators in non-parametric regression have limiting distributions of this form, as proved in~\cite{wright1984asymptotic,alvarez2012m}. Smoothly weighted linear combinations of order statistics also have $\slGCM$ distributions~\citep[Corollary 3.2]{leurgans1982asymptotic}. Pointwise limiting distributions of the non-parametric monotone density estimator are also of the $\slGCM$ form as shown in Section 4 of~\cite{anevski2006general}; also, see Eq. (10) of~\cite{anevski2002monotone}. \cite{anevski2006general} also show that isotonic LSE under some classes of dependent data also have limiting distributions of the $\slGCM$ form. Several other examples involving monotone non-parametric functions can be found in~\citet[Section 3]{deng2021confidence} and~\citet[Section 3]{westling2020unified}. In all these examples, Theorem~\ref{thm:general-drifted-brownian-motion} applies and provides conditions for symmetry of the limiting distributions. Because most of the literature cited here only focuses on the type of assumption from~\cite{wright1981asymptotic} as discussed in our Example~\ref{exmp:Wright-1981}, all these limiting distributions are continuous and symmetric about $0$. We strongly believe that extensions analogous to Theorem~\ref{thm:asymptotic-distribution-of-LSE} are possible and show a much richer limit theory for all these problems as well. As described before, the symmetry of the limiting distribution around 0 implies that $f_{0,n}(x_0)$ is the asymptotic median of $\widehat{f}_{n}(x_0)$ and allows one to construct asymptotically valid confidence intervals requiring no additional estimation. In particular, one does not need to estimate $h_0, \sigma_0^2$, and $\psi(\cdot)$ in order to perform inference as long as $\Psi(-t) = \Psi(t)$ for all $t\in\mathbb{R}.$
\begin{remark}[Necessity of $\psi(\cdot)$ to be odd]
    Note that the condition $\psi(\cdot)$ is odd a.s. is the necessary and sufficient condition for $\Psi(\cdot)$ to be an even function (i.e., $\Psi(-t) = \Psi(t)$). It is worth pointing out that the median-bias being zero does not impose any continuity conditions on $\psi(\cdot)$. In particular, $\psi(\cdot)$ can be a discontinuous function (e.g., $\psi(c) = c\mathbf{1}\{|c| \ge 1\}$). Moreover, by definition $\Psi(\cdot)$ is the integral of $\psi(\cdot)$ and hence, is an absolutely continuous function.
    We do not know if the condition $d(-t) = d(t)$ for all $t\in\mathbb{R}$ of Theorem~\ref{thm:general-drifted-brownian-motion} is a necessary condition for the symmetry of $\slGCM$ at zero. It is an interesting open problem to explore necessary and sufficient conditions of $d(\cdot)$ for the symmetry or zero median of $\slGCM$. {From our plots of the limiting distributions in Example~\ref{exmp:locally-asymmetric}, it is clear that the limiting distribution can be asymmetric if $\Psi(t) \neq \Psi(-t)$ for some $t\in\mathbb{R}$.}
\end{remark}
\subsection{Uniform Validity of HulC Confidence Intervals}\label{subsec:uniform-ci}
\noindent In this section, we establish that, under a triangular-array setup, the validity of HulC-type confidence intervals holds more generally—uniformly over a family of joint distributions on $(X,Y)$ and a class of monotone functions $\mathcal{F}$, under a much less restrictive assumption than \ref{assump:continuity-of-f_{0,n}}, without requiring the existence of a limiting distribution. Suppose we have the following data generating procedure: $(X, Y)$ is a random vector satisfying $Y = f(X) + \xi$ for a monotone non-decreasing function $f\in\mathcal{F}$ and error $\xi$. To account for uniformity, we would need to slightly modify the assumptions \ref{assump:data-model-assumption} and \ref{assump:continuous-distribution-of-covariates} as follows. Let $\mathcal{H},\mathcal{E}$ be families of distributions for $X$ which satisfy the following assumptions:
\begin{enumerate}[label=(UA\arabic*)]
    \item The cumulative distribution function, $H$, of $X$ is in some class $\mathcal{H}$. Assume the class $\mathcal{H}$ is equicontinuous on $\mathbb{R}$ and is continuously differentiable in a neighborhood of $x_0$ with $H'(x_0) = h(x_0)>0$. Formally, there exists a neighborhood
     $\mathcal{N}(x_0)$ (independent of $H\in\mathcal{H}$ and $n\ge1$) such that \label{assump:continuous-distribution-of-covariates-uniform}
     \[
     \limsup_{n\to\infty}\sup_{H\in\mathcal{H},~x\in\mathcal{N}(x_0):\,|x - x_0| \le \delta_n}|h(x) - h(x_0)| = 0,
     \]
     \[
     \inf_{H\in\mathcal{H},~x\in\mathcal{N}(x_0)}h(x) > 0,
     \]
     for any sequence $\{\delta_n\}_{n\ge1}$ such that $\delta_n\to0$ as $n\to\infty$.
    \item  The error $\xi$ satisfies $\mathbb{E}[\xi|X] = 0$ and $\mathbb{E}[\xi^2|X] = \sigma^2(X) \le \overline{\sigma}^2$. Moreover, the class of functions $\mathcal{E}=\{\sigma^2(\cdot): \sigma^2(x) = \E[\xi^2\mid X=x]\}$ is equicontinuous in a neighborhood of $x_0$, i.e., there exists a neighborhood $\mathcal{S}(x_0)$ (independent of $n$) such that\label{assump:data-model-assumption-uniform}
    \[
    \limsup_{n\to\infty}\sup_{\sigma^2\in \mathcal{E},~x\in\mathcal{S}(x_0):|x - x_0| \le \delta_n}|\sigma^2(x) - \sigma^2(x_0)| = 0,
    \]
    \[
    \inf_{\sigma^2\in \mathcal{E}}\sigma^2(x_0) = \sigma_0^2 \in (0, \infty),
    \]
    for any sequence $\{\delta_n\}_{n\ge1}$ such that $\delta_n\to0$ as $n\to\infty$.
\end{enumerate}
Define the family of joint distributions $\mathcal{P}(f; x_0)\equiv \mathcal{P}(f)$ (depending on $f\in\mathcal{F}$) for $(X,Y)$ as:
\begin{align}
\label{eq:general-class}
    \mathcal{P}(f) = \{P_{X,Y}: Y = f(X) + \xi, ~X\sim H~\text{satisfies \ref{assump:continuous-distribution-of-covariates-uniform} and }\xi~\text{satisfies \ref{assump:data-model-assumption-uniform}}\}.
\end{align}
\noindent Consider the setup as in Section \ref{sec:intro-hulc}, where we split the data into $B^*$ batches and compute the isotonic LSE $\hat{f}^{(j)}_{n, \alpha}(x_0)$ on each of them. We begin by introducing a useful unbiasedness property, which will be central to establishing uniform validity.
\begin{definition}[Property $\mathcal{U}$]
 We say that the class of functions $\mathcal{F}$ has the median-unbiasedness property $\mathcal{U}$ if for any sequence $\{f_{0,n}\}_{n\geq 1}\subseteq \mathcal{F}$ and $P_n\in\mathcal{P}(f_{0,n})$, the maximum median bias of the isotonic LSE estimators, computed on each split of the data, as defined by $\Delta_{n,\alpha}$ in \eqref{eq:median-bias-definition}, satisfies:
\[\limsup_{n\to\infty} \Delta_{n,\alpha} \to 0.\]
\end{definition}
\begin{lemma}
\label{lem:unif-validity-under-property} Suppose $\mathcal{F}$ is a class with Property $\mathcal{U}$. Let $\widehat{\mathrm{CI}}_{n,\alpha}(x_0)$ be the HulC confidence interval based on the isotonic LSE estimators for an i.i.d. sample $\{(X_i,Y_i)\}_{i\in[n]}\sim P\in\mathcal{P}(f)$. Then, we have uniform validity:
    \[\liminf_{n\to\infty}  \inf_{f\in\mathcal{F}} \inf_{P\in\mathcal{P}(f)} \Pr_P\left(f(x_0)\in \widehat{\mathrm{CI}}_{n,\alpha}(x_0)\right)=1-\alpha\]
\end{lemma}
A proof is presented in~\ref{appsec:proof-of-unif-validity-under-property}.
Next, we explore sufficient conditions under which property $\mathcal{U}$ holds, and simultaneously relax assumption \ref{assump:continuity-of-f_{0,n}} to include function classes $\mathcal{F}$ and joint distributions $\mathcal{P}$ on $(X,Y)$ such that the isotonic LSE does not have a limiting distribution. First, for a monotone function $f$, any $x_0\in\mathbb{R}$ and a rate function $\rho:(0,\infty)\to(0,\infty)$
satisfying $\rho(h)\to 0$ as $h\downarrow 0$, we define the local drifts by
\[
\psi_{f,\, \rho}(c, h)\ :=\ \frac{f(x_0+ch)-f(x_0)}{\rho(h)},\qquad c\in\mathbb{R},\ \ 0<h\le 1.
\]
If assumption~\ref{assump:continuity-of-f_{0,n}} holds true for $f_{0,n}$ along some sequence $\{s_n\}$, then taking $h_n:=1/s_n$ and
$\rho(h_n):=\sqrt{s_n/n}$ yields $\psi_{f_{0,n}, \,\rho}(c_n, h_n)\to\psi(c)$ for some non-decreasing function
$\psi(\cdot)$ for all $c_n\to c$.
This implies that $\psi_{f_{0,n}, \,\rho}(c, h)$ is eventually bounded on bounded sets whenever $\psi(\cdot)$ is bounded on
bounded sets.
Conversely, if one is only given that the family $\{\psi_{f_{0,n}, \,\rho}(\cdot, h):0<h\le 1\}$ is bounded on bounded sets,
there may not exist a limiting distribution for $\hat f_n(x_0)-f_{0,n}(x_0)$, but by Helly's selection theorem (\cite{brunk1956some}, Theorem~2)
one can get convergence in distribution via subsequences.
\begin{proposition}
    \label{prop:med-unb-unif-validity}
Consider the following set of sufficient assumptions on a class $\mathcal{F}$.\\
For every $f\in\mathcal{F}(x_0)$, assume there exist a rate function $\rho_f:(0,\infty)\to(0,\infty)$ such that,
\begin{enumerate}[label=(S\arabic*)]
    \item For every $C\ge0$, there exists $\mathfrak{B}_C \in (0, \infty)$ such that\label{assump:locally-bounded}
    \[
    \sup_{f\in\mathcal{F}(x_0)}\ \sup_{0<h\le 1}\ \sup_{c\in[-C, C]}|\psi_{f,\,\rho_f}(c, h)|
    \le \mathfrak{B}_C
    \quad\mbox{for all}\quad C\ge0.
    \]
    \item For any $c>0$,\label{assump:locally-symmetric}
    \[
    \sup_{f\in\mathcal{F}(x_0)}\left|\limsup_{h\downarrow 0}\,\frac{\psi_{f,\,\rho_f}(c, h)}{\psi_{f,\,\rho_f}(-c, h)} +1
    \right| = 0.
    \]
    \item \label{assump:tail-nonconstancy}
    \[
    \lim_{C\to\infty}\
    \sup_{f\in\mathcal{F}(x_0)}\ \sup_{0<h\le 1}\ \sup_{|c|>C}\
    \frac{1}{\big|c\big(\psi_{f,\,\rho_f}(3c/2, h)-\psi_{f,\,\rho_f}(c, h)\big)\big|}
    \ =\ 0.
    \]
\end{enumerate}
 Under assumptions~\ref{assump:continuous-distribution-of-covariates-uniform},~\ref{assump:data-model-assumption-uniform},~\ref{assump:locally-bounded},~\ref{assump:locally-symmetric}, and~\ref{assump:tail-nonconstancy}, $\mathcal{F}$ has property $\,\mathcal{U}$ and hence, uniform validity holds for HulC-type confidence intervals.
\end{proposition}
\noindent A proof is presented in~\ref{appsec:proof-of-med-unb-unif-validity}. As an example, consider the following class which admits a Taylor expansion similar to Wright's assumption:
\begin{exmp}
    Let $\mathcal{F}\equiv\mathcal{F}(x_0, \underline{\theta},\bar{\theta},\underline{A},\bar{A})$ be the following class of monotone functions:
\begin{align}
    \mathcal{F}=\{f:&f~\text{is monotone and }\exists~\theta_f\in[\underline{\theta},\bar{\theta}]>0, A_f\in[\underline{A},\bar{A}]>0,~\text{such that}\\
    &~f(x) = f(x_0) + A_f|x-x_0|^{\theta_f}\sign(x-x_0) + R_f(|x-x_0|^{\theta_f})\}
    \label{eq:3-3-class}
\end{align}
Here, $R_f:\R\to\R$ is such that $R_f(0)=0~\forall~f\in\mathcal{F}$ and $\sup_{f\in\mathcal{F}} |R_f(\epsilon)|= o(\epsilon)$ as $\epsilon\to 0$.\\
\noindent It is easy to verify that the class $\mathcal F$ defined above satisfies the assumptions in
Proposition~\ref{prop:med-unb-unif-validity}. Indeed, for each $f\in\mathcal F$ take
$\rho_f(h)=h^{\theta_f}$. Then, uniformly over $f\in\mathcal F$ and $c$ in bounded sets,
\[
\psi_{f,\,\rho_f}(c,h)=\frac{f(x_0+ch)-f(x_0)}{h^{\theta_f}}
= A_f|c|^{\theta_f}\sign(c) + o(1)\qquad (h\downarrow 0),
\]
and the tail non-constancy condition follows since
$|c(\psi_{f,\, \rho_f}(3c/2)-\psi_{f,\, \rho_f}(c))|\asymp |c|^{\theta_f+1}\to\infty$. Thus, using Proposition \ref{prop:med-unb-unif-validity}, we get,
\[\liminf_{n\to\infty} \inf_{f\in\mathcal{F}} \inf_{P\in\mathcal{P}(f)}\Pr\left(f(x_0)\in \widehat{\mathrm{CI}}_{n,\alpha}(x_0)\right)=1-\alpha.\]
\end{exmp}
\section{Simulations}\label{sec:simulation}
In the following subsections, we demonstrate the salient features of the asymptotic result in Theorem~\ref{thm:asymptotic-distribution-of-LSE} and also study the performance of HulC for inference.  Sections~\ref{subsec:sim-s_n} and~\ref{subsec:sim-psi}, in addition to corroborating our theoretical results, show the effects of $s_n$ and $\psi$ (and corresponding $\Psi$) respectively. In section~\ref{subsec:sim-compare-CI}, we compare the performance of applicable inference procedures in monotone regressions setting, analyzing the coverage and width of confidence intervals obtained.
\subsection{Different Rates with Same Limiting Distribution} \label{subsec:sim-s_n}
In this section, we illustrate the fact that for the same limiting distribution, monotone LSE can have different rates of convergence. We mentioned that this is theoretically possible via~\eqref{eq:construction-of-f_0n}. We consider scenarios with the same $\psi(x) = x^2 \sign(x)$ but different choices of $s_n = n^\alpha$ with $\alpha \in \{  {1/6}, {2/6}, {3/6}, {4/6}, {5/6} \}$. The choices of $s_n$ are equally spaced on the log scale.  In similar spirit, we choose $n$ growing from $665 \approx e^{6.5}$ to $22026 \approx e^{10}$, with equally spaced increments in the log scale. For each choice of $(n,s_n)$, we construct $f_{0,n}(x)$ using~\eqref{eq:construction-of-f_0n} and generate $n$ IID observations from $X \sim \unif(-1,1)$ and $Y = f_{0,n}(X) + \xi$ where $\xi|X \sim N(0,1)$. For the $j$-th sample (i.e., $j$-th collection of $n$ observations), we calculate and store $\widehat{f}_{n}^{(j)}(0)$, for $1\le j\le 500$, thus obtaining 500 observations of $\widehat{f}_{n}(0)$, for each choice of $(n,s_n)$.
First, we compare the MSE of the estimator via its estimator, \begin{equation}
\widehat{\mathrm{MSE}}(n,s_n) = \frac{1}{500}\sum_{j = 1}^{500} (\widehat{f}_{n}^{(j)}(0) - {f}_{0,n}(0))^2,
\end{equation}
where $\widehat{f}_{n}^{(j)}(0)$ is the estimator at $0$ obtained by the $j$-th sample of $n$ observations.
\begin{figure}[h!]
    \centering
    \includegraphics[width = 0.9\textwidth]{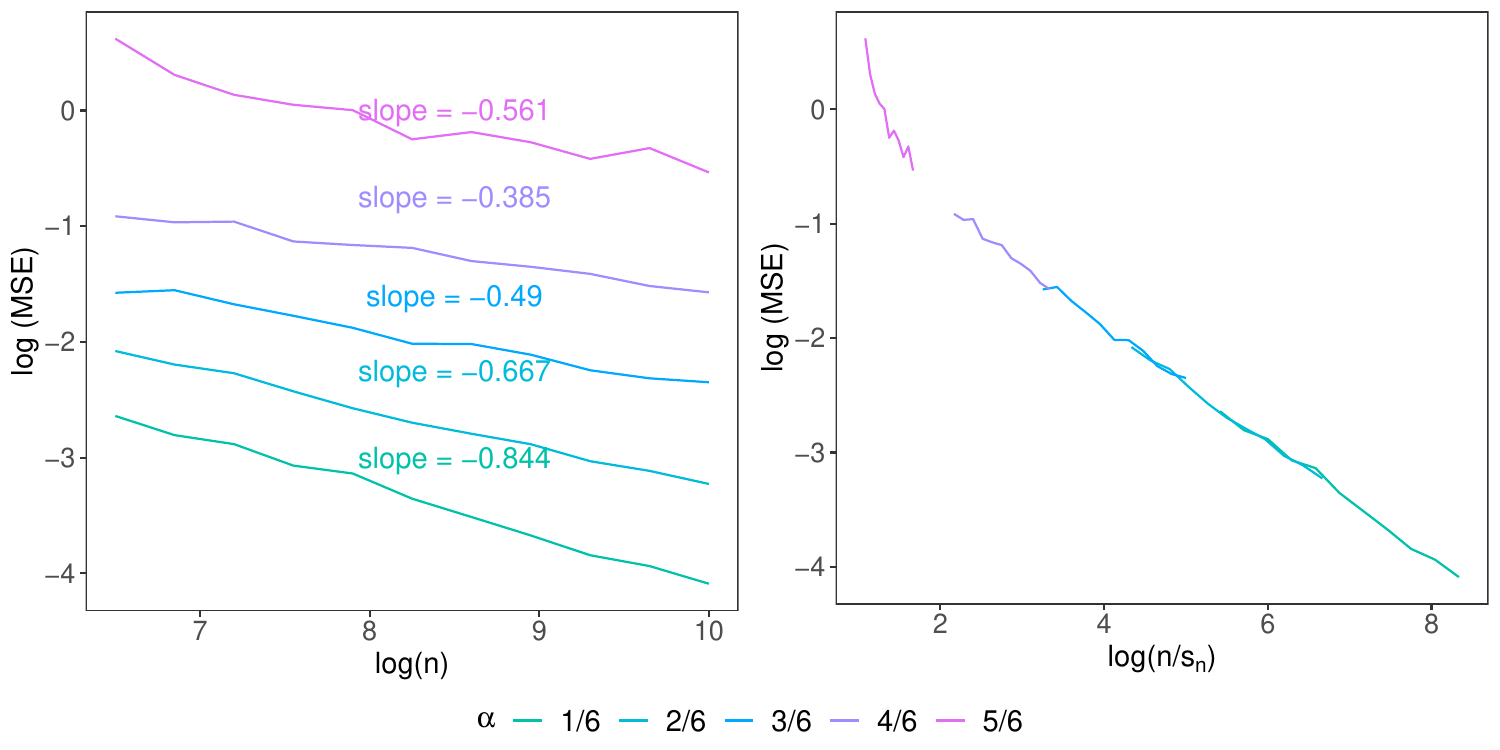}
    \caption{Analysis of MSE with same $\psi(x) = x^2\sign(x)$ with different $s_n = n^\alpha$. The left panel shows a plot of $\log(\widehat{\mathrm{MSE}})$ vs. $\log(n)$. Here, the slope of the best linear approximation of the $s_n= n^\alpha$ curve is close to $(\alpha -1)$. The right panel shows a plot of $\log(\widehat{\mathrm{MSE}})$ vs. $\log(n/s_n)$. Here, instead, we see that the slope of the best linear approximation of the $s_n = n^\alpha$ curve is the same, regardless of the choice of $s_n$. }
    \label{fig:rates-log(mse)}
\end{figure}
According to Theorem~\ref{thm:asymptotic-distribution-of-LSE}, $\sqrt{n/s_n}(\widehat{f}_{n}(0) - {f}_{0,n}(0))$ has a non-degenerate limit (which is the same for all choices of $s_n$ since we chose the same $\psi, h_0$ and $\sigma_0^2$), thus suggesting that
\[
\log(\widehat{\mathrm{MSE}}(n,s_n = n^\alpha)) \approx \text{constants} + (\alpha - 1) \log(n).
\]
This is exactly what is observed in Figure~\ref{fig:rates-log(mse)} left plot when plotting $\log(\widehat{\mathrm{MSE}})$ vs. $\log(n)$. For $s_n = n^\alpha$, the best linear approximation of $\log(\widehat{\mathrm{MSE}})$ with respect to $\log(n)$ has a slope close to $(\alpha - 1)$ as annotated in the plot. Whereas when plotting $\log(\widehat{\mathrm{MSE}})$ vs. $\log(n/s_n)$ in Figure~\ref{fig:rates-log(mse)} right plot, we see that the curves have similar slope regardless of the rate.
Finally, for Figure~\ref{fig:rates-qqplot}, we fix $n$ to a large number (approximately $e^{10}$). For each $s_n$, we plot a combined QQ-plot of 500 observations of $(1/2)\sqrt{n/s_n}(\hat{f}_{n}(0) - {f}_{0,n}(0))$ and 500 observations from the \textit{same} asymptotic distribution (given by the RHS of Theorem~\ref{thm:asymptotic-distribution-of-LSE}). The QQ plots all concentrate around $y= x$ line, indicating that the asymptotic distribution of $\sqrt{n/s_n}(\hat{f}_{n}(0) - {f}_{0,n}(0))$ is indeed the same regardless of choice of $s_n$.
\begin{figure}[H]
    \centering
    \includegraphics[width = 0.75\textwidth]{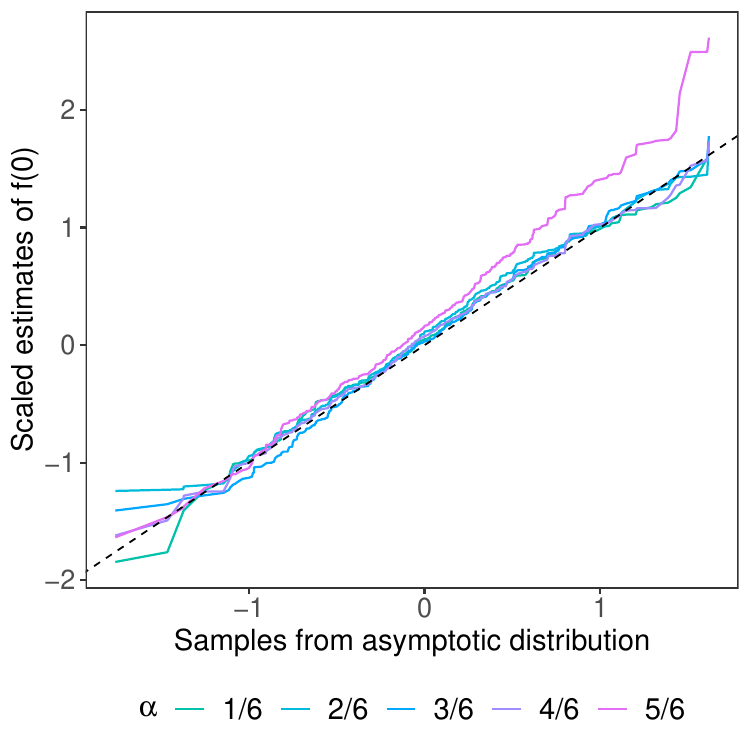}
    \caption{Combined QQ-plot of 500 observations of $1/2\cdot\sqrt{n/s_n}\cdot(\hat{f}_{n}(0) - {f}_{0,n}(0))$ for each choice of $s_n = n^\alpha$ vs. 500 observations from the \textit{same} asymptotic random variable,  $\slGCM[B(t) + |t|^3/24:t\in\mathbb{R}](0)$. The lines concentrate around the $y=x$ line.  }
    \label{fig:rates-qqplot}
\end{figure}
\subsection{Same Rate with Different Limiting Distributions}\label{subsec:sim-psi}
To study the effect of $\psi$, we first fix $s_n$ to $n^{1/3}$. We pick 4 choices of $\psi$ which vary in smoothness symmetry and regularity:
\begin{align}\label{eq:sim-diff-psi}
\psi_1(x) &= x,\quad \psi_2(x)= x[\mathbf{1}\{x \ge 0\}/2 + \mathbf{1}\{x < 0\}],\\
 \psi_3(x) &= (x^2/2)\mathbf{1}\{x \ge 0\} + (x^3/3)\mathbf{1}\{x < 0\}, \quad \psi_4(x) = \mathbf{1}\{|x|> 0.1\} \sign(x).
\end{align}
Before presenting the simulation results, let us pause to understand the properties of these functions. $\psi_1$ is a smooth odd function around $0$, which is continuously differentiable, leading to a symmetrical $\Psi_1$. $\psi_2$ is a continuous but not a differentiable function, with different left and right derivatives at 0. $\psi_3$ is also not an odd function, but with a higher order of smoothness. It has continuous first-order derivatives, but different left and right second-order derivatives. Moreover, the degree of smoothness (order of the polynomial) of the $\psi_3$ is different on the left and right sides of the $0$. $\psi_4$ on the other hand, is flat around $0$ but is discontinuous around it at $\{-0.1,0.1\}$.
As in Section~\ref{subsec:sim-s_n}, we choose $n$ growing from $665 \approx e^{6.5}$ to $22026 \approx e^{10}$, with equally spaced increments in the log scale. For each choice of $(n,\psi)$, we construct $f_{0,n}(x)$ as in~\eqref{eq:construction-of-f_0n} and generate $n$ IID observations from $X \sim \unif(-1,1)$ and $Y = f_{0,n}(X) + \xi$ where $\xi \sim N(0,1)$. For each set of observations, we calculate and store the estimate $\widehat{f}_{n}$ at $0$. We repeat this process 500 times, thus obtaining 500 observations of $\widehat{f}_{n}(0)$ for each choice of $(n,\psi)$.
First, we compare the MSE of the estimator via its estimator,
\begin{equation}
\widehat{\mathrm{MSE}}(n,\psi) = \frac{1}{500} \sum_{j = 1}^{500} (\widehat{f}_{n}^{(j)}(0) - {f}_{0,n}(0))^2,
\end{equation}
where $\widehat{f}_{n}^{(j)}(0)$ is the estimator at $0$ obtained from the $j$-th sample. According to Theorem~\ref{thm:asymptotic-distribution-of-LSE}, $\sqrt{n/s_n}(\hat{f}_{n}(0) - {f}_{0,n}(0))$ has a non-degenerate limit which is now different for all choices of $\psi$. This suggests that $\E[ (\hat{f}_{n}(0) - {f}_{0,n}(0))^2 ]$ is $C^*(\psi) s_n/n = C^*(\psi) n^{-2/3}$ where $C^*(\psi)$ is representing constants depending on the asymptotic distribution due to $\psi$ alone ($h_0 = 1/2, \sigma_0^2 = 1$). Hence, we should observe that$$\log(\widehat{\mathrm{MSE}}(n,\psi)) \approx \text{constant}(\psi) +  (-2/3) \log(n).$$
\begin{figure}[h]
    \centering
    \includegraphics[width = 0.7\textwidth]{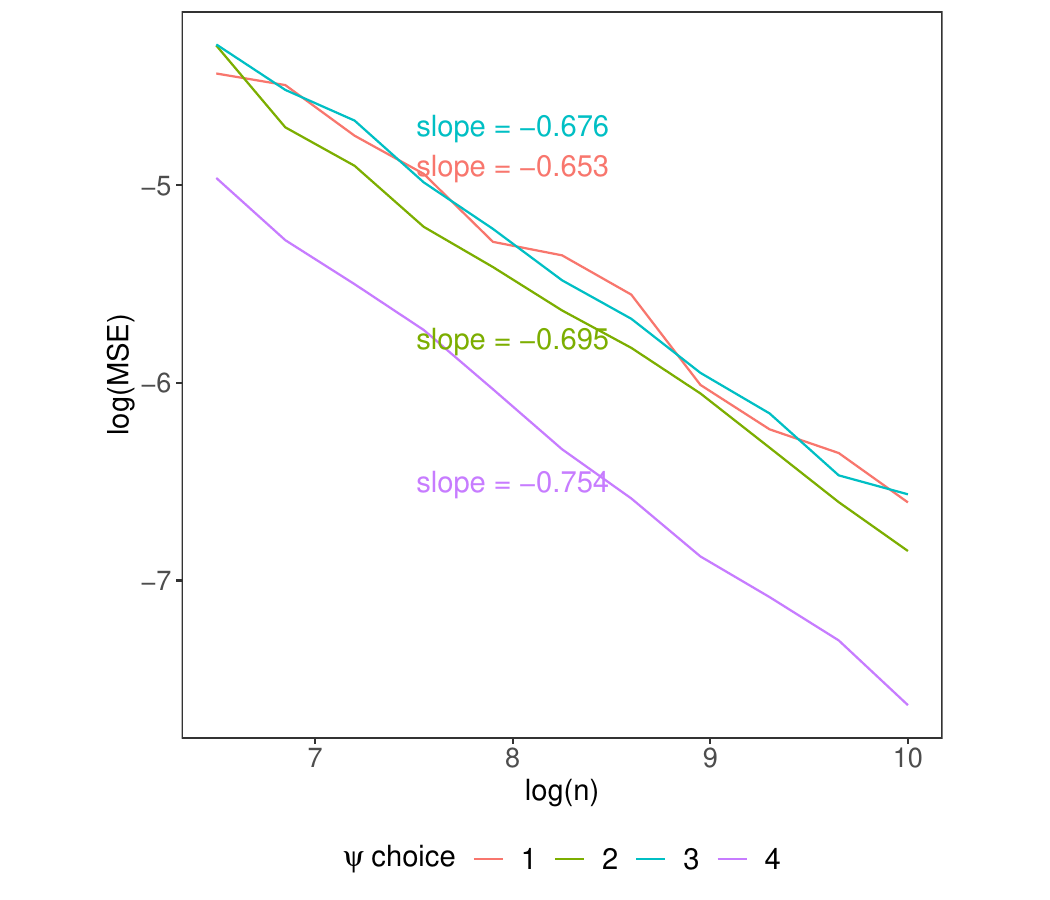}
    \caption{Analysis of MSE with same $s_n = n^{1/3}$
     with different $\psi$ as in (\ref{eq:sim-diff-psi}). This is a  plot of $\log(\widehat{\mathrm{MSE}})$ vs. $\log(n)$. Here, the slope of the best linear approximation of each $\psi$ curve is close to $-2/3 = -0.667$.}
    \label{fig:psi-log(mse)}
\end{figure}
This is demonstrated in Figure~\ref{fig:psi-log(mse)}, when plotting $\log(\widehat{\mathrm{MSE}})$ vs. $\log(n)$. For different choices of $\psi$, the best linear approximation of $\log(\widehat{\mathrm{MSE}})$ with respect to $\log(n)$ has a slope close to $-2/3$ (as annotated in the plot) with different intercepts for each $\psi$.
For Figure~\ref{fig:rates-qqplot2}, we fix $n$ to a large number (approximately $e^{10}$). For each choice of $\{\psi_k\}_{k = 1,2,3,4}$, we plot a separate QQ-plot of 500 observations of $C\sqrt{n/s_n}(\hat{f}_{n}(0) - {f}_{0,n}(0))$ (where $C$ is the appropriate constant from the LHS of Theorem~\ref{thm:asymptotic-distribution-of-LSE}) and 500 observations from the asymptotic distribution (given by the RHS of Theorem~\ref{thm:asymptotic-distribution-of-LSE}) which \textit{vary} with $\psi$. The QQ plots all concentrate around $y= x$ line, indicating that the asymptotic distribution of $\sqrt{n/s_n}(\hat{f}_{n}(0) - {f}_{0,n}(0))$ is indeed matching the proposed asymptotic distribution.
\begin{figure}[h]
    \centering
    \includegraphics[width = 0.75\textwidth]{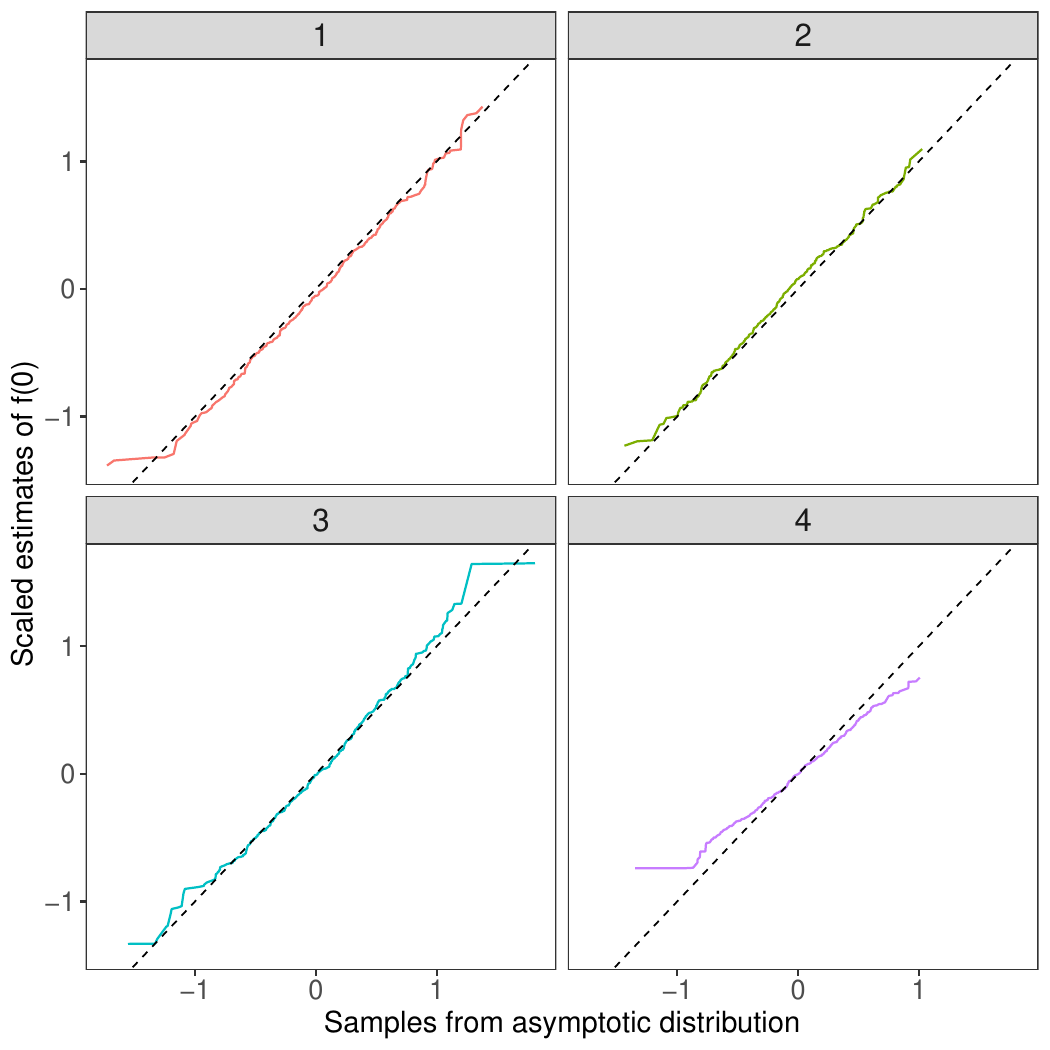}
    \caption{QQ-plot of 500 observations of $1/2\cdot\sqrt{n/s_n}\cdot(\hat{f}_{n}(0) - {f}_{0,n}(0))$ for each choice of ${\psi_k}$ vs. 500 observations from the asymptotic random variable,  $\slGCM[B(t) + \Psi_k(t/2):t\in\mathbb{R}](0)$ where $\Psi_k$ is defined by (\ref{eq:definition-Psi}) for $\psi_k$. The choice of $k \in \{1,2,3,4\}$ is on top of each panel. The lines concentrate around the $y=x$ line for each $\psi_k$.  }
    \label{fig:rates-qqplot2}
\end{figure}
\subsection{Comparison of Confidence Intervals}\label{subsec:sim-compare-CI}
To compare the performance of confidence intervals for monotone regression, we consider the example of $X\sim \unif(-1,1)$ and $Y = |x|^\theta \sign(x)  + \xi$ where $\xi\sim N(0,1)$. $\theta\ge0$ here is the flatness parameter, and our point of interest $x_0=0$.
We compare the following methods.
\begin{itemize}
   \item \textbf{DHZ oracle pivot method}: \cite{deng2021confidence} propose a pivotal statistic $\widehat{r}_n(\widehat{f}_n(x_0) - f_{0,n}(x_0))$  with $\widehat{r}_n$ calculated from the data; see Section 4.3 of~\cite{deng2021confidence}. The asymptotic distribution is the same for all functions with the same local flatness parameter. Thus, assuming that this parameter is known, one can draw large samples from this distribution, and obtain a high-accuracy estimate of any necessary quantile. This quantile can now be used to provide confidence intervals for $f_{0,n}(0)$. This is not a practical method, as $\theta$ is in general unknown. Theorem 3 of~\cite{deng2021confidence} suggests the use of the largest quantile over all flatness parameters, but no method computing such largest quantile is given.
      \item \textbf{DHZ data-driven pivot method}:  When the local flatness parameter is not known, \citet[Section 4.1.2]{deng2021confidence} suggest estimating the quantiles by simulating data from a smoothed LOESS estimator with $\sigma^2$ being estimated using the difference estimator \citep{rice1984bandwidth}. This method is the pivotal version of the smoothed bootstrap discussed in~\citet[Section 4.2.1]{guntuboyina2018nonparametric}. Although performing well in our simulations, this method currently has no theoretical guarantee, and we believe that its performance strongly depends on the underlying smoother used.
\item \textbf{Subsampling with unknown rate of convergence:} As mentioned in the introduction, classical subsampling is not applicable for isotonic LSE because of the unknown rate. \cite{bertail1999subsampling} proposes a subsampling method involving estimating the rate of convergence. We use subsample size $m = n^{1/2}$ for our comparison.
\item \textbf{HulC:} \cite{kuchibhotla2021hulc} provides a general tuning-free inference method that is valid for asymptotically zero median-bias distributions. Given our choice of $f_{0,n}(x) = |x|^\theta \sign(x)$, as seen in Example~\ref{exmp:Wright-1981}, we get $\psi(x) = |x|^\theta \sign(x)$ and $s_n = n^{1/(2\theta + 1)}$. By Theorem~\ref{thm:general-drifted-brownian-motion}, we know that the asymptotic distribution has zero median bias.
\end{itemize}
We take sample sizes $n$ to be in $\{50,100,250,1000\}$.
To see the effect of the flatness parameter, we pick $\theta$  ranging from $0.2$ to $10$.  For each $(\theta,n)$, we take $n$ observations $(X_i,Y_i), 1\le i\le n$ from the data-generating process mentioned above and construct confidence intervals for $f_{0,n}(x_0)$ with $0.95$ probability target. This process is replicated 1000 times to obtain estimates for the expected coverage and width for each method as follows,
\begin{align}
    \widehat{\E}[\text{Coverage(CI)}] &\approx \frac{1}{1000}\sum_{j = 1}^{1000} \text{I}\left( f(x_0) \in \widehat{\mathrm{CI}}^j\right), \notag\\
    \widehat{\E}[\text{Width(CI)}] &\approx \frac{1}{1000}\sum_{j = 1}^{1000} \text{Width}\left(\widehat{\mathrm{CI}}^j\right).
\end{align}
where $\widehat{\mathrm{CI}}^j$ is the confidence interval obtained from the $j$-th replication. In Figure~\ref{fig:compare}, we plot the coverage and width for all the methods across the flatness parameter $\theta$. Subsampling fails to hit the coverage target of $0.95$  across all sample sizes when the $f_{0,n}(x)$ has a low degree of smoothness. Both variants of the DHZ pivotal method perform well for small and large sample sizes. The HulC method, although missing the coverage target slightly for small $\theta$'s, performs well for large $\theta$'s, having better or comparable performance to other methods. Note that in this example, Theorem~\ref{thm:asymptotic-distribution-of-LSE} is not applicable when $\theta$ is allowed to converge to zero (and hence median unbiasedness may not hold). Finally, it is interesting to note that HulC has a comparable or shorter width compared to the DHZ methods, even for small values of $\theta$.
\begin{figure}[h]
    \centering
    \includegraphics[width = \textwidth]{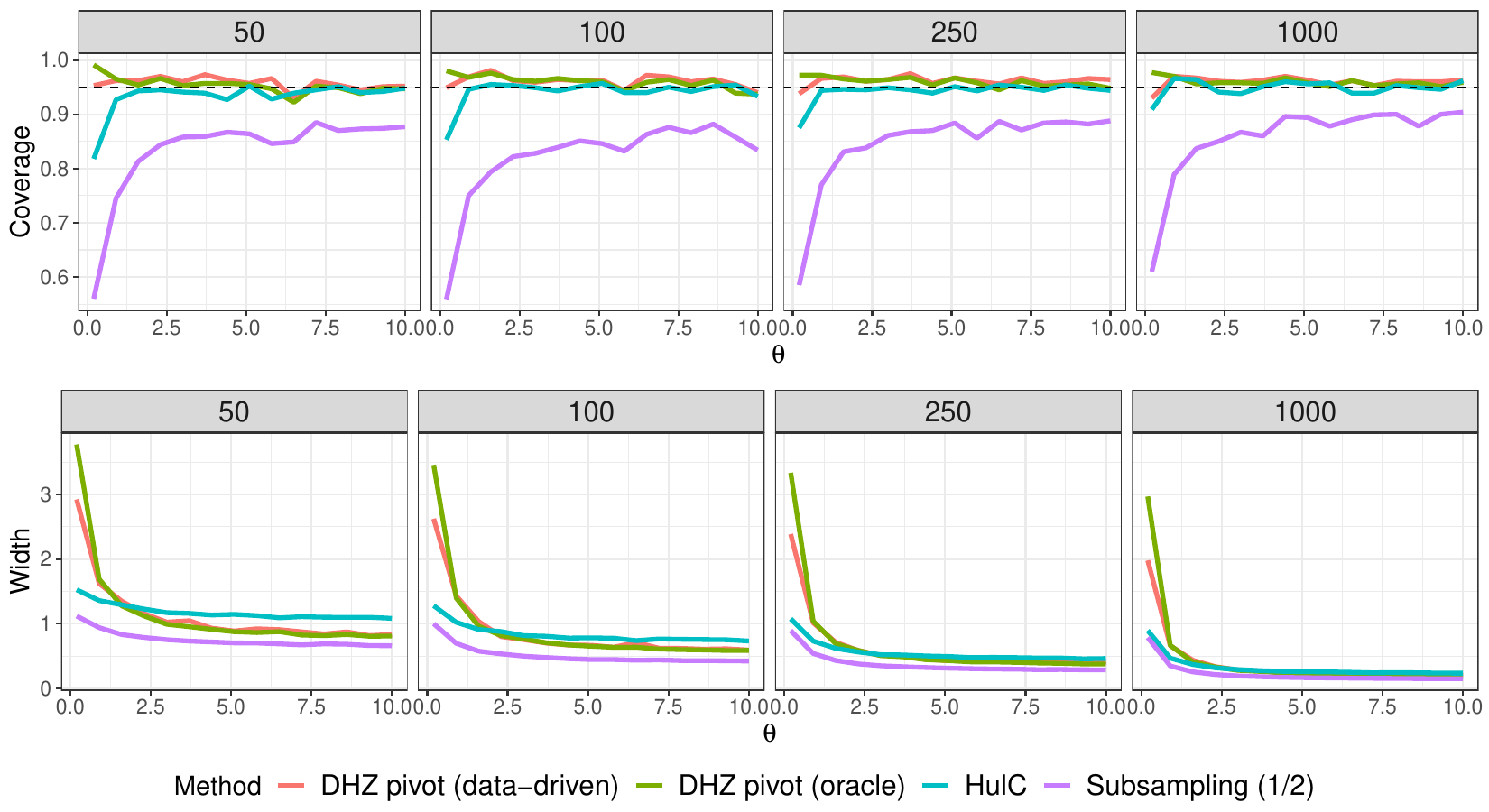}
    \caption{Comparison of methods to construct confidence intervals for monotonic regression estimator. The top of each panel is the sample size for which the comparison is done. The top plot compares the expected coverage of CI constructed vs. the degree of smoothness $\theta$. The bottom plot compares the expected width of CI constructed vs. the degree of smoothness $\theta$. }
    \label{fig:compare}
\end{figure}
\section{Discussion}\label{sec:discussions}
In this paper, we have significantly expanded the known collection of limiting distributions for isotonic regression. In a random design setting with independent observations, we have obtained the limiting distributions for isotonic LSE under a variety of local behavior assumptions on the underlying regression function. Furthermore, we studied the properties of the limiting distributions such as continuity, symmetry, and the median, and provided a simple sufficient condition for symmetry. With the help of such conditions, we constructed asymptotically valid adaptive confidence intervals for the isotonic regression function at a point $x_0$. To our knowledge, this is the first such uniformly valid confidence interval that does not even require the existence of a limiting distribution.
There are several interesting future venues to explore. Firstly, the set of possible limiting distributions obtained in this paper can be further expanded by relaxing the assumptions on the distribution of the covariates. This can be done by allowing the density of the covariates at the point of interest to be zero but restricting the local behavior of the distribution function. Secondly, we focused on the case of univariate isotonic regression in this paper. Two ways the setting can be generalized is by considering either multivariate isotonic regression problems or other shape-constrained problems including $k$-monotonicity~\citep{guntuboyina2015global} or both. At present, it is not obvious how either of these would evolve because, for the multivariate case, the LSE is not minimax optimal and an optimal block min-max estimator is more suitable, but the limiting distributions are no longer defined in terms of $\slGCM$~\citep{deng2020isotonic,deng2021confidence} and for other shape-constraints the limiting distribution involves more complicated functions of Brownian motion, such as the ``invelope'' function for convex regression~\citep{ghosal2017univariate}. More importantly, the technical tools to understand the distribution of random variables defined via optimization problems involving Brownian motion are not well-developed. Another interesting direction would be to understand the unimodality of the $\slGCM$ distribution. From the plots of the limiting distribution in our examples, we conjecture that unimodality at zero holds true for the $\slGCM$ for any non-negative strongly convex drift function which is zero at zero. This would allow us to use unimodal HulC from~\cite{kuchibhotla2021hulc} for asymptotically valid inference.
\section*{Acknowledgements}
This work was supported by the National Science Foundation (NSF) DMS-2210662.

\bibliographystyle{elsarticle-harv}
\bibliography{ref}

\clearpage
\begin{center}
  {\LARGE\bfseries Appendices\par}
\end{center}
\vspace{1.5em}

\appendix
\setcounter{section}{0}
\setcounter{equation}{0}
\setcounter{figure}{0}
\setcounter{table}{0}
\setcounter{definition}{0}
\setcounter{assumption}{0}
\setcounter{exmp}{0}
\renewcommand{\theequation}{\Alph{section}.\arabic{equation}}
\renewcommand{\thefigure}{\Alph{section}.\arabic{figure}}
\renewcommand{\thetable}{\Alph{section}.\arabic{table}}
\makeatletter
\@addtoreset{equation}{section}
\@addtoreset{figure}{section}
\@addtoreset{table}{section}
\@addtoreset{definition}{section}
\@addtoreset{assumption}{section}
\@addtoreset{exmp}{section}
\makeatother
\renewcommand{\thedefinition}{\Alph{section}.\arabic{definition}}
\renewcommand{\theassumption}{\Alph{section}.\arabic{assumption}}
\renewcommand{\theexmp}{\Alph{section}.\arabic{exmp}}
\section{Additional simulations}\label{appsec:add-sims}
In Section~\ref{subsec:sim-compare-CI}, we compare several inference procedures for the data being generated with a $Y = |X|^\theta \sign(X) + \eps$, where $X\sim \unif(-1,1)$ and $\eps\sim N(0,1)$. In this section, we aim to compare several procedures across other interesting settings, extending beyond our canonical example. We consider functions on the domain $[-1,1]$ with the point of interest $x_0 = 0$. Also, we center all the functions such that $f_{0,n}(0) = 0$.
\begin{enumerate}
    \item \textbf{Heteroscedastic error:}
    Assumption (A1) allows for heteroscedastic errors with some regularity conditions for the variance $\sigma^2(x)$. We consider the following case.
    \begin{align}\label{eq:hetero_err}
       \textbf{(1)}\ & Y = \sqrt{n/s_n}|s_n X|^{3/2}\sign(s_n X) + \eps,\quad s_n = n^{1/3},\\
       & X \sim \unif(-1,1),\ \eps|X \sim N(0, 2|X| + 1).
    \end{align}
    \item \textbf{Non-polynomial choice of $\psi$:}
    A key aspect of Theorem~\ref{thm:asymptotic-distribution-of-LSE} is that it allows for any choice of monotonic function $\psi$ (and convex $\Psi$). We pick the following cases to compare methods in a non-polynomial choice of $\psi$.
        \begin{align}\label{eq:non_poly}
        &Y = \sqrt{n/s_n}\psi\big( s_nX\big)  + \eps,\quad s_n = n^{1/3}, \\
        &X \sim \unif(-1,1),\ \eps|X \sim N(0,1/2). \\
       \text{\textbf{Choice of $\psi$ :}} \   &\textbf{(2)} \ 2x + 2\sin(x) \\
       &\textbf{(3)} \ 5\big(F_{1.2,2.45}(x/2+1/2) - F_{1.2,2.45}(0)\big).
    \end{align}
    where $F_{\theta_1,\theta_2}$ is the CDF of Beta$(\theta_1,\theta_2)$ distribution.
    Among our 2 choices of $\psi$, Model \textbf{(2)} provides a symmetrical asymptotic distribution, whereas Model \textbf{(3)} provides an asymmetric asymptotic distribution. This can affect the coverage performance of HulC \citep{kuchibhotla2021hulc}, which works with median-unbiasedness.
    \item \textbf{$f_{0,n}$ with asymptotically different degree of smoothness:}
    Here we consider a case mentioned in the broader setting of Example~\ref{exmp:near-flat-functions}, which is as follows.
    \begin{align}\label{eq:asymp-diff-smoothness}
      \textbf{(4)}\ &  Y = n^{-1/5}X + X^3 + \eps, \\
       & X \sim \unif(-1,1),\ \eps|X \sim N(0, 1/2).
    \end{align}
    As mentioned in the example, although for finite $n$, the first derivative is non-zero (i.e. $n^{-1/5}$) at 0, it vanishes as $n$ grows and therefore asymptotically behaves like $x^3$. This scenario is not accounted for by~\cite{wright1981asymptotic}.
    \item \textbf{Multiple non-vanishing derivates:}
    Here we propose a choice of $\psi$ with a polynomial with multiple non-vanishing derivatives.
        \begin{align}\label{eq:non_vanish_deriv}
        &Y = \sqrt{n/s_n}\psi\big( s_nX\big)  + \eps,\quad s_n = \log\log(n), \\
        & X \sim \unif(-1,1),\ \eps \sim N(0,1/2). \\
       \text{\textbf{Choice of $\psi$ :}} \   &\textbf{(5)} \ x + x^3 + x^5.
    \end{align}
    In this case, $\psi$ has non-zero first, third, and fifth derivatives.
    \item \textbf{Non-uniform distribution of $X$:}  Assumption (A2) allows for not only random $X$, but with a non-uniform distribution. We consider the following case.
    \begin{align}\label{eq:non_unif_X}
       \textbf{(6)}\ & Y = \sqrt{n/s_n}|s_n X|^{3/2}\sign(s_n X) + \eps,\quad s_n = n^{1/3},\\
       & X \sim 2\cdot\text{Beta}(2,3) - 1,\ \eps|X \sim N(0, 1/2).
    \end{align}
    \item \textbf{Asymmetric error distribution: } Theorem~\ref{thm:asymptotic-distribution-of-LSE} allows for an asymmetric and heteroscedastic error model (provided they obey some regularity conditions imposed by Assumption (A1)). We consider the following case where the error follows mean $0$ non-central $\chi^2$ distribution.
        \begin{align}\label{eq:asymm_err}
       \textbf{(7)}\ & Y = \sqrt{n/s_n}|s_n X|^{3/2}\sign(s_n X) + \eps,\quad s_n = n^{1/3},\\
       & X \sim \unif(-1,1),\ \eps|X \sim Z^2 - X^2 - 1 \text{ where }Z\sim N(|X|, 1)
    \end{align}
\end{enumerate}
All the models considered are summarized in Table~\ref{table:more_exmp}. We take sample sizes $n$ to be from $500$ to $5000$ with a spacing of $500$.
 For each choice of example and $n$, we take $n$ observations $(X_i,Y_i), 1\le i\le n$ from the data-generating process mentioned above and construct confidence intervals for $f_{0,n}(x_0)$ with $0.95$ probability target. This process is replicated 1000 times to obtain estimates for the expected coverage and width for each method as follows,
\begin{align}
    \widehat{\E}[\text{Coverage(CI)}] &\approx \frac{1}{1000}\sum_{j = 1}^{1000} \text{I}\left( f(x_0) \in \widehat{\mathrm{CI}}^j\right), \notag\\
    \widehat{\E}[\text{Width(CI)}] &\approx \frac{1}{1000}\sum_{j = 1}^{1000} \text{Width}\left(\widehat{\mathrm{CI}}^j\right).
\end{align}
where $\widehat{\mathrm{CI}}^j$ is the confidence interval obtained from the $j$-th replication. In Figure~\ref{fig:compare_more_exmp}, we plot the coverage and width for all the methods across the $n$, for the 7 examples mentioned above. Subsampling fails to hit the coverage target of $0.95$  across all examples. The data-driven DHZ pivotal method attains (sometimes overshooting) the coverage target when the error is heteroscedastic or asymmetric (Model \textbf{(1)} in \eqref{eq:hetero_err} and Model \textbf{(7)} in \eqref{eq:asymm_err}). But it fails to adapt to a non-standard true function $f_{0,n}$ or non-uniform distribution for $X$. The HulC method slightly undercovers in Model \textbf{(3)} in \eqref{eq:non_poly}. This is expected since the asymptotic distribution is asymmetric. The method performs well across all the other examples, where asymptotic symmetric behavior is present.
\begin{figure}[!htb]
    \centering
    \includegraphics[width = \textwidth]{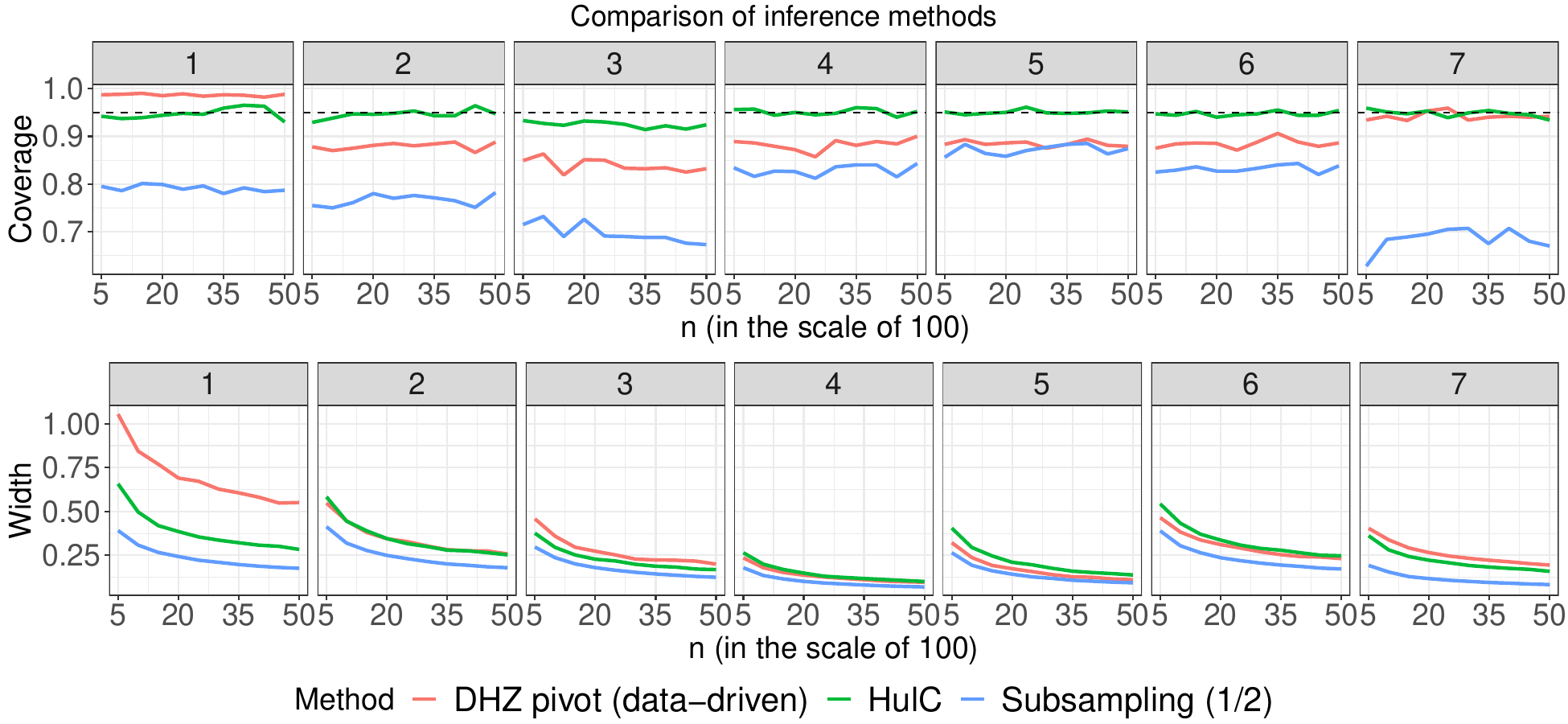}
    \caption{Comparison of methods to construct confidence intervals for monotonic regression estimator. The top of each panel is the model index as mentioned in Table~\ref{table:more_exmp}. The top plot compares the expected coverage of CI constructed vs. sample size $n$. The bottom plot compares the expected width of CI constructed vs. sample size $n$.
    }
    \label{fig:compare_more_exmp}
\end{figure}
\begin{table}[!htb]
\renewcommand{\arraystretch}{1.15}
\centering
\footnotesize
\begin{tabular}{cccccc}
 \hline
 Model  & $\psi(x)$ & $s_n$ & $X$  & $\eps|X$ & $f_{0,n}$ \\ [1ex]
 \hline
 1 & $|x|^{3/2}\sign(x)$ & $n^{1/3}$ & $\unif(-1,1)$ &  $N(0,2|X| + 1)$ & $ \sqrt{n/s_n}\psi\big( s_nX\big) $\\
 2 & $x + \sin(x)$ & $n^{1/3}$ & $\unif(-1,1)$ & $N(0,1/2)$ & $ \sqrt{n/s_n}\psi\big( s_nX\big) $\\
 3 & $5\big(F_{1.2,2.45}(x/2+1/2)$ & $n^{1/3}$ & $\unif(-1,1)$ & $N(0,1/2)$ & $ \sqrt{n/s_n}\psi\big( s_nX\big) $\\
& $- F_{1.2,2.45}(0)\big)$ & & &  &\\
 4 & $x$ & $ n^{1/5}$ & $\unif(-1,1)$ & $N(0,1/2)$  & $ n^{-1/5}X + X^3$ \\
 5 & $ x + x^3 + x^5$ & $\log\log(n)$ & $\unif(-1,1)$ & $N(0,1/2)$  & $ \sqrt{n/s_n}\psi\big( s_nX\big) $ \\
 6 &$|x|^{3/2}\sign(x)$ & $n^{1/3}$ & $2\cdot\text{Beta}(2,3) - 1$ & $N(0,1/2)$  & $ \sqrt{n/s_n}\psi\big( s_nX\big) $\\
 7 &  $|x|^{3/2}\sign(x)$ & $n^{1/3}$ & $\unif(-1,1)$ & $ (N(|X|, 1))^2 - X^2 - 1$  & $ \sqrt{n/s_n}\psi\big( s_nX\big) $\\ [1ex]
 \hline
\end{tabular}
\caption{Table of additional simulation models used to compare methods in Figure~\ref{fig:compare_more_exmp}.}
\label{table:more_exmp}
\end{table}
\section{Proof of Proposition~\ref{prop:uniform-boundedness}\label{appsec:proof-of-uniform-boundedness}}
Because $\widehat{f}_n(\cdot)$ is a non-decreasing function and $\widehat{f}_n(x) = \widehat{f}_n(X_{1:n})$ for all $x \le X_{1:n}$, $\widehat{f}_n(x) = \widehat{f}_n(X_{n:n})$ for all $x \ge X_{n:n}$, it follows that
\[
\|\widehat{f}_n\|_{\infty} = \max\{|\widehat{f}_n(X_{1:n})|,\,|\widehat{f}_n(X_{n:n})|\}.
\]
Moreover, from~\eqref{eq:min-max-formula}, we conclude that
\begin{align*}
\widehat{f}_n(X_{1:n}) &= \min_{j\ge1}\frac{1}{j}\sum_{k=1}^j \Yo{k}\\
\widehat{f}_n(X_{n:n}) &= \max_{i\ge1}\,\frac{1}{n-i+1}\sum_{k=i}^{n} \Yo{k} = \max_{j\ge1}\frac{1}{j}\sum_{k=1}^{j} \Yo{n-k+1}.
\end{align*}
If $n \le 3$, then $\|\widehat{f}_n\|_{\infty} \le \max_{1\le i\le n}|Y_i|$ and $\mathbb{E}[\|\widehat{f}_n\|_{\infty}^p] \le \sum_{i=1}^n \mathbb{E}[|\mu(X_i)|^p] + \sum_{i=1}^n \mathbb{E}[|\xi_i|^p] \le 3(\mathbb{E}[|\mu(X)|^p] + \mathbb{E}[\eta^p(X)])$. We now prove the result for $n\ge3$. Set $J_n = \lfloor\log_2(n)\rfloor - 1$. Therefore,
\begin{align*}
\|\widehat{f}_n\|_{\infty} ~&\le~ \max\left\{\max_{j\ge1}\,\left|\frac{1}{j}\sum_{k=1}^j Y_{[k:n]}\right|,\,\max_{j\ge1}\,\left|\frac{1}{j}\sum_{k=1}^j \Yo{n-k+1}\right|\right\}\\
&\le~ \max_{j\ge1}|\mu(\Xo{j})| + \max\left\{\max_{j\ge1}\,\left|\frac{1}{j}\sum_{k=1}^j \xio{k}\right|,\,\max_{j\ge1}\,\left|\frac{1}{j}\sum_{k=1}^j \xio{n-k+1}\right|\right\}\\
&\le~ \max_{j\ge1}|\mu(\Xo{j})| \\
&\quad+ \max\Bigg\{\max_{0 \le \ell \le J_n}\frac{\max_{2^\ell \le j< 2^{\ell + 1}}|\sum_{k=1}^{j}\xio{k}|}{2^{\ell}},\, \max_{2^{J_n+1} \le j \le n} \frac{|\sum_{k=1}^j \xio{k}|}{2^{J_n+1}},\\
&\quad \max_{0 \le \ell \le J_n}\frac{\max_{2^\ell \le j< 2^{\ell + 1}}|\sum_{k=1}^{j}\xio{n-k+1}|}{2^{\ell}},\, \max_{2^{J_n+1} \le j\le n} \frac{|\sum_{k=1}^j \xio{n-k+1}|}{2^{J_n+1}}
\Bigg\}.
\end{align*}
This implies that
\begin{align*}
&\left(\mathbb{E}[\|\widehat{f}_n\|_{\infty}^p]\right)^{1/p} \le \left(\mathbb{E}\left[\max_{j\ge1}|\mu(X_{j:n})|^p\right]\right)^{1/p}\\
&\quad+\Bigg(\sum_{\ell = 0}^{\lfloor \log_2(n)\rfloor}\frac{1}{2^{\ell p}}\mathbb{E}\left[\max_{2^{\ell} \le j < 2^{\ell + 1}}\left|\sum_{k=1}^j \xio{k}\right|^p\right] + \frac{1}{2^{(J_n+1) p}}\mathbb{E}\left[\max_{2^{J_n+1} \le j\le n}\left|\sum_{k=1}^j \xio{k}\right|^p\right]\\
&\quad+ \sum_{\ell = 0}^{\lfloor \log_2(n)\rfloor}\frac{1}{2^{\ell p}}\mathbb{E}\left[\max_{2^{\ell} \le j < 2^{\ell + 1}}\left|\sum_{k=1}^j \xio{n-k+1}\right|^p\right] + \frac{1}{2^{(J_n+1) p}}\mathbb{E}\left[\max_{2^{J_n+1} \le j\le n}\left|\sum_{k=1}^j \xio{n-k+1}\right|^p\right]\Bigg)^{1/p}.
\end{align*}
Note that $\xio{k}, 1\le k\le n$ are not independent random variables due to the potential dependence between errors and covariates, but they are conditionally independent given $X_1, \ldots, X_n$ under the continuity assumption of $x\mapsto \mathbb{P}(X \le x)$; see Lemma~1 of~\cite{bhattacharya1974convergence}. Furthermore, $\mathbb{P}(\xio{k} \le s|X_1, \ldots, X_n) = G_{\Xo{k}}(s)$, where $G_x(s):=\mathbb{P}(Y - \mu(X)\le s|X = x)$. This implies that conditional on $X_1, \ldots, X_n$, $\xio{k}, 1\le k\le n$ are mean zero independent random variables and $\mathbb{E}[|\xio{k}|^p|X_1, \ldots, X_n] = \eta^p(\Xo{k})$. This fact implies that $\{\sum_{k=1}^j \xio{k}: 2^\ell \le j < 2^{\ell + 1}\}$ is a martingale (conditional on $X_1, \ldots, X_n$) and Doob's maximal inequality implies, for example, that for $p > 1$,
\begin{align*}
\mathbb{E}\left[\max_{2^{\ell} \le j < 2^{\ell + 1}}\left|\sum_{k=1}^j \xio{k}\right|^p\right] &\le \mathbb{E}\left[\left|\sum_{k=1}^{2^{\ell + 1}} \xio{k}\right|^p\right].
\end{align*}
Theorem~2 of~\cite{von1965inequalities} implies that for $1 < p\le 2$,
\[
\mathbb{E}\left[\left|\sum_{k=1}^{2^{\ell + 1}} \xio{k}\right|^p\bigg|X_1, \ldots, X_n\right] \le 2\sum_{k=1}^{2^{\ell + 1}} \mathbb{E}[|\xio{k}|^p|X_1, \ldots, X_n] = 2\sum_{k=1}^{2^{\ell + 1}} \eta^p(X_{k:n}).
\]
Hence,
\[
\mathbb{E}\left[\max_{2^{\ell} \le j < 2^{\ell + 1}}\left|\sum_{k=1}^j \xio{k}\right|^p\right] \le 2\sum_{k=1}^{2^{\ell + 1}} \mathbb{E}\left[\eta^p(X_{k:n})\right].
\]
Similar calculations hold true for other terms as well. Therefore,
\begin{align*}
\left(\mathbb{E}[\|\widehat{f}_n\|_{\infty}^p]\right)^{1/p} &\le \left(\mathbb{E}\left[\max_{j\ge1}|\mu(X_{j:n})|^p\right]\right)^{1/p}\\
&\quad+2\Bigg(\sum_{\ell = 0}^{J_n}\frac{1}{2^{\ell p}}\sum_{k=1}^{2^{\ell + 1}}\mathbb{E}\left[\eta^p(\Xo{k})\right] + \frac{1}{2^{(J_n+1) p}}\sum_{k=1}^{n}\mathbb{E}\left[\eta^p(\Xo{k})\right]\\
&\quad+ \sum_{\ell = 0}^{J_n}\frac{1}{2^{\ell p}}\sum_{k=1}^{2^{\ell + 1}}\mathbb{E}\left[\eta^p(\Xo{n-k+1})\right] + \frac{1}{2^{(J_n+1) p}}\sum_{k=1}^{n}\mathbb{E}\left[\eta^p(\Xo{n-k+1})\right]\Bigg)^{1/p}.
\end{align*}
The first term is trivially bounded by $\|\mu\|_{\infty}$ and the second term is bounded by
\[
2^{1 + 1/p}\|\eta\|_{\infty}\left(\sum_{\ell = 0}^{J_n}\frac{2^{\ell + 1}}{2^{\ell p}} + \frac{n}{2^{(J_n + 1)p}}\right)^{1/p} \le 2^{1 + 2/p}\|\eta\|_{\infty}\left(\sum_{\ell = 0}^{\infty} \frac{1}{2^{\ell(p-1)}}\right)^{1/p} \le \|\eta\|_{\infty}\frac{2^{2 + 1/p}}{(2^{p-1} - 1)^{1/p}}.
\]
This completes the proof of the result.
\section{Proof of Theorem~\ref{thm:asymptotic-distribution-of-LSE}}\label{appsec:proof-of-thm-asymptotic-dist-LSE}
The basic structure of the proof of already presented in Section~\ref{subsec:proof-main-result}. Here we fill in the details.
Let $\widehat{H}_n$ denote the empirical CDF based on $\mathcal{X}$, i.e.,
\[
\widehat{H}_n(x) := \frac{1}{n}\sum_{i=1}^n \mathbf{1}\{X_i \le x\}.
\]
By assumption~\ref{assump:continuous-distribution-of-covariates}, for sufficiently large $n$, $\exists~\beta_l(n),~\beta_u(n)>0$, s.t.
\[H_n(x_0)-H_n(x_0-\beta_l(n))=H_n(x_0+\beta_u(n))-H_n(x_0)=\frac{c}{s_n}.\]
Again, we write $\beta_l(n)$ and $\beta_u(n)$ as $\beta_l$ and $\beta_u$ respectively. Finally, define
\begin{align*}
\delta_{l,-}(n)&:=|\{k:x_0-\beta_l(n)< X_{k:n}\leq x_0\}|,\\
\delta_{l,+}(n)&:=|\{k:x_0< X_{k:n}\leq x_0+\beta_l(n)\}|.
\end{align*}
The following lemma proves some properties of $\beta_l, \beta_u, \alpha_l, \alpha_u,$ and $\delta_{l,-}(n)$ under assumption~\ref{assump:continuous-distribution-of-covariates}.
\begin{lemma}\label{lem:smoothness-of-H}
Under assumption~\ref{assump:continuous-distribution-of-covariates}, for any fixed $c > 0$, as $n\to\infty$,
\[
s_n\beta_l(n)h_n(x_0) = c(1 + o(1)),\; s_n\beta_u(n)h_n(x_0) = c(1 + o(1)).
\]
Similarly, as $n\to\infty$
\[
s_n\alpha_l(n)h_n(x_0) = 2c(1 + o(1)),\; s_n\alpha_u(n)h_n(x_0) = 2c(1 + o(1)).
\]
The same conclusions continue to hold with $h_n(x_0)$ replaced by $h_0.$
Furthermore, $s_n\delta_{l,-}(n)/n=c(1 + o_p(1))$ and $s_n\delta_{l,+}/n~=~c(1 + o_p(1))$ as $n\to\infty$.
\end{lemma}
\begin{proof}
Observe by the mean value theorem, that
\begin{align*}
&H_n(x_0)-H_n(x_0-\beta_l(n))=H_n(x_0+\beta_u(n))-H_n(x_0)=\frac{c}{s_n}\\
&\implies H_n'(x_0+\xi_l(n))\beta_l(n)=H_n'(x_0+\xi_u(n))\beta_u(n)=\frac{c}{s_n}
\end{align*}
Here $\xi_l(n)\in(-\beta_l(n),~0),~\xi_u(n)\in(0,~\beta_u(n))~\forall~n\in\mathbb{N}$. Note that, as $\beta_l(n)\to 0$ and $\beta_u(n)\to 0$ as $n\to\infty$, $\xi_l(n)\to 0$ and $\xi_u(n)\to 0$ as $n\to\infty$. Because $H'$ is continuous in a neighborhood $\mathcal{N}'(x_0)$ of $x_0$, for sufficiently large enough $n$, $H_n'(x_0+\xi_l(n))=H_n'(x_0)+\eta_l(n) $ and $H_n'(x_0+\xi_u(n))=H_n'(x_0)+\eta_u(n)$, where $\lim_{n\to\infty}\eta_l(n)=\lim_{n\to\infty}\eta_u(n)=0$.  Hence,
\begin{equation}
s_n\beta_l(n)h_n(x_0)=c(1+o(1))\quad\mbox{and}\quad s_n\beta_u(n)h_n(x_0)=c(1+o(1)),\label{beta1}
\end{equation}
as $n\to\infty$.
Using the same logic,
\begin{equation}s_n\alpha_l(n)h_n(x_0)=2c(1+o(1))\quad\mbox{and}\quad s_n\alpha_u(n)h_n(x_0)=2c(1+o(1)),\label{alpha1}\end{equation}
as $n\to\infty$.
To prove the limit for $\delta_{l,-}(n)$, note that $\delta_{l,-}(n)\sim\mbox{Binomial}(n, c/s_n)$. Hence,
\[
\mathbb{E}\left[\frac{s_n\delta_{l,-}(n)}{n}\right] = c,\quad\mbox{and}\quad \mbox{Var}\left(\frac{s_n\delta_{l,-}(n)}{n}\right) = \frac{s_n^2}{n^2}n(c/s_n)(1 - c/s_n) \le c\frac{s_n}{n}.
\]
Because $cs_n/n\to0$ as $n\to\infty$, the variance converges to zero and we conclude that $s_n\delta_{l,-}(n)/(nc) - 1 = O_p(\sqrt{s_n/n}) = o_p(1)$.
The proof for $\delta_{l,+}$ is nearly identical and is omitted.
\end{proof}
\begin{lemma}\label{lem:control-of-v_n}
Define
\[
v_n = \frac{1}{\delta_{l,-}(n)}\sum_{k:x_0-\beta_l(n)< X_{k:n}\leq x_0}(f_{0,n}(X_{k:n})-f_{0,n}(x_0-\beta_l(n))).
\]
Under assumptions~\ref{assump:continuous-distribution-of-covariates} and~\ref{assump:continuity-of-f_{0,n}},
\begin{equation}\label{eq:main-lemma-limit-of-v_n}
\sqrt{\frac{n}{s_n}}v_n ~\overset{p}{\to}~ \left[-\psi(-c/h_0) - (h_0/c)\Psi(-c/h_0)\right]\quad\mbox{as}\quad n\to\infty.
\end{equation}
\end{lemma}
\begin{proof}
Note that
\begin{align*}
\sqrt{\frac{n}{s_n}}v_n&=\sqrt{\frac{n}{s_n}}\frac{\sum_{k:x_0-\beta_l(n)\leq X_{k:n}\leq x_0}(f_{0,n}(X_{k:n})-f_{0,n}(x_0-\beta_l(n)))}{|k:x_0-\beta_l(n)\leq X_{k:n}\leq x_0|}\\
&=\frac{n}{s_n\delta_{l,-}(n)}s_n\sqrt{\frac{n}{s_n}}\int_{\left(x_0-\beta_l(n),~x_0\right]}\left(f_{0,n}(x)-f_{0,n}(x_0-\beta_l(n))\right)d\widehat{H}_n(x)\\
&=\frac{n}{s_n\delta_{l,-}(n)}\sqrt{ns_n}\int_{\left(x_0-\beta_l(n),~x_0\right]}\left(f_{0,n}(x)-f_{0,n}(x_0-\beta_l(n))\right)dH_n(x)\\&\quad+\frac{n}{s_n\delta_{l,-}(n)}\sqrt{ns_n}\int_{\left(x_0-\beta_l(n),~x_0\right]}\left(f_{0,n}(x)-f_{0,n}(x_0-\beta_l(n))\right)d(\widehat{H}_n - H_n)(x)\\
&=\frac{n}{s_n\delta_{l,-}(n)}(I_1+I_2)~~(\text{say}).
\end{align*}
By Lemma~\ref{lem:smoothness-of-H}, $s_n\delta_{l,-}(n)/n = c(1 + o_p(1))$ as $n\to\infty$ and hence, it suffices to show that as $n\to\infty$,
\[
I_1 := \sqrt{ns_n}\int_{\left(x_0-\beta_l(n),~x_0\right]}\left(f_{0,n}(x)-f_{0,n}(x_0-\beta_l(n))\right)dH_n(x) ~\to~ -c\psi(-c/h_0) - h_0\Psi(-c/h_0),
\]
and
\[
I_2 := \sqrt{ns_n}\int_{\left(x_0-\beta_l(n),~x_0\right]}\left(f_{0,n}(x)-f_{0,n}(x_0-\beta_l(n))\right)d(\widehat{H}_n - H_n)(x) \overset{p}{\to} 0.
\]
Clearly (from~\ref{assump:data-model-assumption}), $I_1 \ge 0$. Because $(x_0 - \beta_l(n), x_0]$ belongs to $\mathcal{N}(x_0)$ for large enough $n$, we conclude that for large enough $n$,
\begin{equation}\label{eq:I_1-bounds}
I_1 \in \sqrt{ns_n}\int_{(x_0 - \beta_{l}(n), x_0]} (f_{0,n}(x) - f_{0,n}(x_0 - \beta_l(n)))dx\times\left[\inf_{|x - x_0| \le \beta_l(n)}h_n(x),\, \sup_{|x - x_0| \le \beta_l(n)}h_n(x)\right].
\end{equation}
Because $\beta_l(n)\to0$ as $n\to\infty$, we obtain as $n\to\infty$,
\begin{align*}
\inf_{|x - x_0| \le \beta_l(n)}h_n(x) &\ge h_n(x_0) - \sup_{|x - x_0| \le \beta_l(n)}|h_n(x) - h_n(x_0)| \to h_0,\\
\sup_{|x - x_0| \le \beta_l(n)}h_n(x) &\le h_n(x_0) + \sup_{|x - x_0| \le \beta_l(n)}|h_n(x) - h_n(x_0)| \to h_0.
\end{align*}
Now consider the integral on the right-hand side of~\eqref{eq:I_1-bounds} can be controlled using assumption~\ref{assump:continuity-of-f_{0,n}} and Proposition~\ref{prop:assumptions-implications}:
\begin{align*}
    &\int_{(x_0 - \beta_{l}(n), x_0]} (f_{0,n}(x) - f_{0,n}(x_0 - \beta_l(n)))dx\\
    &= \int_{-\beta_l(n)}^{0} (f_{0,n}(x_0 + t) - f_{0,n}(x_0))dt + \beta_l(n)(f_{0,n}(x_0) - f_{0,n}(x_0 - \beta_l(n)))\\
    &= \beta_l(n)\int_{-1}^0 (f_{0,n}(x_0 + t{s_n\beta_l(n)}/{s_n}) - f_{0,n}(x_0))dt+\beta_l(n)(f_{0,n}(x_0) - f_{0,n}(x_0 - \beta_l(n))).
\end{align*}
Because $s_n\beta_l(n)\to c/h_0$ as $n\to\infty$ by Lemma~\ref{lem:smoothness-of-H}, we obtain from Part 3 of Proposition~\ref{prop:assumptions-implications} that
\[
\sqrt{\frac{n}{s_n}}\int_{-1}^0 (f_{0,n}(x_0 + t{s_n\beta_l(n)}/{s_n}) - f_{0,n}(x_0))dt \to -\frac{\Psi(-c/h_0)}{(c/h_0)},\quad\mbox{as}\quad n\to\infty,
\]
and by assumption~\ref{assump:continuity-of-f_{0,n}},
\[
\sqrt{\frac{n}{s_n}}(f_{0,n}(x_0) - f_{0,n}(x_0 - s_n\beta_l(n)/s_n)) \to -\psi(-c/h_0),\quad\mbox{as}\quad n\to\infty.
\]
Therefore, as $n\to\infty$, we get
\[
\sqrt{ns_n}\int_{(x_0 - \beta_{l}(n), x_0]} (f_{0,n}(x) - f_{0,n}(x_0 - \beta_l(n)))dx \to -(c/h_0)\psi(-c/h_0) - \Psi(-c/h_0).
\]
To control $I_2$, note that
\begin{align*}
\mbox{Var}\left(I_2\right) &= s_n\int_{(x_0 - \beta_l(n), x_0]} (f_{0,n}(x) - f_{0,n}(x_0 - \beta_l(n)))^2dH_n(x)\\
&\le s_n(f_{0,n}(x_0) - f_{0,n}(x_0 - \beta_l(n)))\int_{(x_0 - \beta_l(n), x_0]} (f_{0,n}(x) - f_{0,n}(x_0 - \beta_{l}(n)))dH_n(x)\\
&= \sqrt{\frac{s_n}{n}}(f_{0,n}(x_0) - f_{0,n}(x_0 - \beta_l(n)))I_1\\
&= \frac{s_n}{n}\sqrt{\frac{n}{s_n}}(f_{0,n}(x_0) - f_{0,n}(x_0 - \beta_l(n)))I_1
\end{align*}
By assumption~\ref{assump:continuity-of-f_{0,n}}, we have
\[
\sqrt{\frac{n}{s_n}}(f_{0,n}(x_0) - f_{0,n}(x_0 - \beta_l(n))) ~\to~ -\psi(-c/h_0),\quad\mbox{as}\quad n\to\infty.
\]
Because $s_n/n\to0$, we conclude that $\mbox{Var}(I_2) \to 0$ and because $\mathbb{E}[I_2] = 0$, we obtain that $I_2\overset{p}{\to} 0$ as $n\to\infty.$
\end{proof}
\begin{lemma}\label{lem:Doob's-maximal-inequality}
Suppose $Z_i$ are mean zero independent random variables with finite second moment. Let $S_j^{'}=\sum_{i=1}^j Z_i$. Then,
\[\mathbb{P}\left(\max_{m\leq j\leq n}|S_j'|>t\right)\leq\frac{4\sum_{i=1}^n \mathbb{E}[Z_j^2]}{t^2}\]
\end{lemma}
\begin{proof}
By Chebyshev's inequality, we have
\[
\mathbb{P}\left(\max_{m\leq j\leq n}|S_j'|>t\right) \le \frac{1}{t^2}\mbox{Var}\left(\max_{m \le j\le n}|S_j'|\right).
\]
Because $\{S_j'\}_{j\ge1}$ is a martingale, we get that $\{|S_j'|\}_{j\ge1}$ is a submartingale and hence, Doob's maximal inequality implies that
\[
\mbox{Var}\left(\max_{m \le j\le n}|S_j'|\right) \le \mathbb{E}\left[\max_{m\le j\le n}|S_j'|^2\right] \le \mathbb{E}\left[\max_{1\le j\le n}|S_j'|^2\right] \le 4\mathbb{E}[|S_n'|^2] = 4\sum_{i=1}^n \mathbb{E}[Z_j^2].
\]
\end{proof}
\begin{lemma}\label{lem:Hajek-Renyi-special-case}
Suppose $Z_1, \ldots, Z_n$ are mean zero independent random variables such that $\mathbb{E}[Z_j^2] \le \sigma^2$ for all $1\le j\le n$.  Then
\[\mathbb{P}\left(\max_{m\leq j\leq n}\frac{S_j'}{j}>t\right)\leq\frac{4\sigma^2}{t^2m}\]
\end{lemma}
\begin{proof}
Observe that
\begin{align*}
\mathbb{P}\left(\max_{m\leq j\leq n}\frac{S_j'}{j}>t\right)&\le\mathbb{P}\left(\bigcup_{k=1}^{\lceil \log_2\left(\frac{n}{m}\right)\rceil}\left\{\max_{2^{k-1}m\leq j\leq 2^km}\frac{S_j'}{j}>t\right\}\right)\\
&\leq \sum_{k=1}^{\lceil \log_2\left(\frac{n}{m}\right)\rceil}\mathbb{P}\left(\max_{2^{k-1}m\leq j\leq 2^km}\frac{S_j'}{j}>t\right)\\
&\leq \sum_{k=1}^{\lceil \log_2\left(\frac{n}{m}\right)\rceil}\mathbb{P}\left(\max_{2^{k-1}m\leq j\leq 2^km}S_j'>2^{k-1}mt\right)\\
&\leq 4\sum_{k=1}^{\lceil \log_2\left(\frac{n}{m}\right)\rceil}\frac{\sigma^22^{k-1}m}{2^{2k-2}m^2t^2}~~(\text{Using Lemma~\ref{lem:Doob's-maximal-inequality}})\\
&=\frac{16\sigma^2}{t^2m}\sum_{k=1}^{\lceil \log_2\left(\frac{n}{m}\right)\rceil}\frac{1}{2^k}\leq\frac{16\sigma^2}{t^2m}.
\end{align*}
\end{proof}
\begin{lemma}\label{lem:unequalness-of-fn-star-and-fhat}
For all $n\ge1$, we have
\begin{align*}
\{f_{n,c}^{*}(x_0)\neq \hat{f}_n(x_0)\}&\subseteq\left\{\min_{z\geq x_0}\mathrm{Av}_n\left(\left(x_0-\beta_l(n),~z\right]\right)< f_{0,n}(x_0-\beta_l(n))\right\}\\
&\quad\cup\left\{\max_{y\leq x_0-\alpha_l(n)}\mathrm{Av}_n\left([y,~x_0-\beta_l(n)]\right)> f_{0,n}(x_0-\beta_l(n))\right\}\\
&\quad\cup\left\{\max_{y\leq x_0}\mathrm{Av}_n\left(\left(y,~x_0+\beta_u(n)\right]\right)> f_{0,n}(x_0+\beta_u(n))\right\}\\
&\quad\cup\left\{\min_{z\geq x_0+\alpha_u(n)}\mathrm{Av}_n\left([x_0+\beta_u(n),~z]\right)<f_{0,n}(x_0+\beta_u(n))\right\}.
\end{align*}
\end{lemma}
Lemma~\ref{lem:unequalness-of-fn-star-and-fhat} shows that the event that the restricted estimator $f_{n,c}^{*}(x_0)$ is not equal to the original one $\hat{f}(x_0)$, is a subset of the union of certain events. The fact that these events on the RHS have a small probability, will be shown in the next lemma.
\begin{proof}[Proof of Lemma~\ref{lem:unequalness-of-fn-star-and-fhat}]
Observe that
\begin{align*}
\{\hat{f}_n(x_0)=f_{n,c}^{*}(x_0)\}\supseteq &\left\{\min_{z\geq x_0}\text{Av}_n\left(\left(x_0-\beta_l(n),~z\right]\right)\geq \max_{y\leq x_0-\alpha_l(n)}\text{Av}_n\left([y,~x_0-\beta_l(n)]\right)\right\}\cap\\
&\left\{\max_{y\leq x_0}\text{Av}_n\left(\left(y,~x_0+\beta_u(n)\right]\right)\leq \min_{z\geq x_0+\alpha_u(n)}\text{Av}_n\left([x_0+\beta_u(n),~z]\right)\right\}.
\end{align*}
To see how this holds, let us deal with the first event on the RHS:
\begin{align*}
&\min_{z\geq x_0}\text{Av}_n\left(\left(x_0-\beta_l(n),~z\right]\right)\geq \max_{y\leq x_0-\alpha_l(n)}\text{Av}_n\left([y,~x_0-\beta_l(n)]\right)\\
\implies &\text{Av}_n\left(\left(x_0-\beta_l(n),~z\right]\right)\geq \text{Av}_n\left([y,~x_0-\beta_l(n)]\right)~~\forall~z\geq x_0,~y\leq x_0-\alpha_l(n)\\
\implies &\text{Av}_n\left(\left(x_0-\beta_l(n),~z\right]\right)\geq \min_{y\leq x_0-\alpha_l(n)}\text{Av}_n\left([y,~z]\right)~~\forall~z\geq x_0.
\end{align*}
Now, for any $z\geq x_0$,
\begin{align*}
\max_{y\leq x_0} \text{Av}_n\left([y,~z]\right)&=\max\left\{\max_{y\leq x_0-\alpha_l(n)} \text{Av}_n\left([y,~z]\right),~\max_{x_0-\alpha_l(n)<y\leq x_0}\text{Av}_n\left([y,~z]\right)\right\}\\
&\leq\max\left\{\text{Av}_n\left(\left(x_0-\beta_l(n),~z\right]\right),~\max_{x_0-\alpha_l(n)<y\leq x_0}\text{Av}_n\left([y,~z]\right)\right\}\\
&=\max_{x_0-\alpha_l(n)<y\leq x_0}\text{Av}_n\left([y,~z]\right)~~~(\text{as }x_0-\beta_l(n)\in\left(x_0-\alpha_l(n),~x_0\right]).
\end{align*}
But, as the maximum over a restricted set is no larger than that on the entire set, we have
\[\max_{y\leq x_0} \text{Av}_n\left([y,~z]\right)\geq \max_{x_0-\alpha_l(n)<y\leq x_0}\text{Av}_n\left([y,~z]\right).\]
Combining the above two results, for any $z\geq x_0$, we have
\begin{equation}
\max_{y\leq x_0} \text{Av}_n\left([y,~z]\right)= \max_{x_0-\alpha_l(n)<y\leq x_0}\text{Av}_n\left([y,~z]\right).\label{eq1}
\end{equation}
Similarly, the second event on the RHS can be dealt with as follows:
\begin{align*}
&\max_{y\leq x_0}\text{Av}_n\left(\left(y,~x_0+\beta_u(n)\right]\right)\leq \min_{z\geq x_0+\alpha_u(n)}\text{Av}_n\left([x_0+\beta_u(n),~z]\right)\\
\implies &\text{Av}_n\left(\left(y,~x_0+\beta_u(n)\right]\right)\leq \text{Av}_n\left([x_0+\beta_u(n),~z]\right)~~\forall~z\geq x_0+\alpha_u(n),~y\leq x_0\\
\implies &\text{Av}_n\left(\left(y,~x_0+\beta_u(n)\right]\right)\leq \min_{z\geq x_0+\alpha_u(n)}\text{Av}_n\left([y,~z]\right)~~\forall~y\leq x_0.
\end{align*}
Thus, for any $y\leq x_0$,
\begin{align*}
\min_{z\geq x_0} \text{Av}_n\left([y,~z]\right)&=\min\left\{\min_{z\geq x_0+\alpha_u(n)} \text{Av}_n\left([y,~z]\right),~\min_{x_0+\alpha_u(n)>z\geq x_0}\text{Av}_n\left([y,~z]\right)\right\}\\
&\geq\min\left\{\text{Av}_n\left(\left(y,~x_0+\beta_u(n)\right]\right),~\min_{x_0+\alpha_u(n)>z\geq x_0}\text{Av}_n^{'}\left([y,~z]\right)\right\}\\
&=\min_{x_0+\alpha_u(n)>z\geq x_0}\text{Av}_n\left([y,~z]\right)~~~(\text{as }x_0+\beta_u(n)\in\left(x_0, ~x_0+\alpha_u(n)\right]).
\end{align*}
But, as the minimum over a restricted set is no less than that on the entire set, for any $y\leq x_0$, we have
\begin{equation}
\min_{z\geq x_0} \text{Av}_n\left([y,~z]\right)=\min_{x_0+\alpha_u(n)>z\geq x_0}\text{Av}_n\left([y,~z]\right)\label{eq2}
\end{equation}
From \eqref{eq1} and \eqref{eq2}, we have:
\begin{align*}
\hat{f}_n(x_0)&=\max_{i:X_{(i)}\leq x_0}\min_{j:x_0\leq X_{(j)}}\text{Av}_n\left([X_{(i)},~X_{(j)}]\right)\\
&=\max_{i:X_{(i)}\leq x_0}\left[\min_{j:x_0\leq X_{(j)}<x_0+\alpha_u(n)}\text{Av}_n\left([X_{(i)},~X_{(j)}]\right)\right]\\
&=\min_{j:x_0\leq X_{(j)}<x_0+\alpha_u(n)}\left[\max_{i:X_{(i)}\leq x_0}\text{Av}_n\left([X_{(i)},~X_{(j)}]\right)\right]\\
&=\min_{j:x_0\leq X_{(j)}<x_0+\alpha_u(n)}\left[\max_{i:x_0-\alpha_l(n)<X_{(i)}\leq x_0}\text{Av}_n\left([X_{(i)},~X_{(j)}]\right)\right]=f_{n,c}^{*}(x_0)
\end{align*}
Thus,
\begin{align*}
\{\hat{f}(x_0)\neq f_{n,c}^{*}(x_0)\}\subseteq &\left\{\min_{z\geq x_0}\text{Av}_n\left(\left(x_0-\beta_l(n),~z\right]\right)< \max_{y\leq x_0-\alpha_l(n)}\text{Av}_n\left([y,~x_0-\beta_l(n)]\right)\right\}\cup\\
&\left\{\max_{y\leq x_0}\text{Av}_n\left(\left(y,~x_0+\beta_u(n)\right]\right)> \min_{z\geq x_0+\alpha_u(n)}\text{Av}_n\left([x_0+\beta_u(n),~z]\right)\right\}
\end{align*}
Now, observe that
\begin{align*}
    &\left\{\min_{z\geq x_0}\text{Av}_n\left(\left(x_0-\beta_l(n),~z\right]\right)< \max_{y\leq x_0-\alpha_l(n)}\text{Av}_n\left([y,~x_0-\beta_l(n)]\right)\right\}\\&\hspace{2cm}\subseteq
    \left\{\min_{z\geq x_0}\text{Av}_n\left(\left(x_0-\beta_l(n),~z\right]\right)< f_{0,n}(x_0-\beta_l(n))\right\}\\
    &\hspace{4cm}\cup\left\{\max_{y\leq x_0-\alpha_l(n)}\text{Av}_n\left([y,~x_0-\beta_l(n)]\right)> f_{0,n}(x_0-\beta_l(n))\right\},
\end{align*}
and
\begin{align*}
    &\left\{\max_{y\leq x_0}\text{Av}_n\left(\left(y,~x_0+\beta_u(n)\right]\right)> \min_{z\geq x_0+\alpha_u(n)}\text{Av}_n\left([x_0+\beta_u(n),~z]\right)\right\}\\
    &\hspace{2cm}\subseteq
\left\{\max_{y\leq x_0}\text{Av}_n\left(\left(y,~x_0+\beta_u(n)\right]\right)> f_{0,n}(x_0+\beta_u(n)\right\}\\
&\hspace{4cm}\quad\cup\left\{\min_{z\geq x_0+\alpha_u(n)}\text{Av}_n^{'}\left([x_0+\beta_u(n),~z]\right)<f_{0,n}(x_0+\beta_u(n)\right\}.
\end{align*}
This completes the proof of the lemma.\\\\
\end{proof}
\begin{lemma}\label{lem:prob-of-unequalness-goes-to-zero}
Under assumptions~\ref{assump:data-model-assumption},~\ref{assump:continuous-distribution-of-covariates}, and~\ref{assump:continuity-of-f_{0,n}}, for all $c > 0$, we have
\begin{align*}
\limsup_{n\to\infty} \mathbb{P}(f_{n,c}^{*}(x_0) \neq \hat{f}_n(x_0))
&\le \frac{4\overline{\sigma}^2}{c}\Big[\frac{1}{\Gamma_1(c/h_0)} + \frac{1}{\Gamma_1(-c/h_0)}\\
&\quad+ \frac{1}{\Gamma_2(c/h_0)} + \frac{1}{\Gamma_2(-c/h_0)}\Big].
\end{align*}
Under~\ref{assump:continuity-of-f_{0,n}},
\[
\limsup_{c\to\infty}\limsup_{n\to\infty} \mathbb{P}(f_{n,c}^{*}(x_0) \neq \hat{f}_n(x_0)) = 0.
\]
\end{lemma}
\begin{proof}[Proof of Lemma~\ref{lem:prob-of-unequalness-goes-to-zero}]
Consider the bounds obtained in Lemma~\ref{lem:unequalness-of-fn-star-and-fhat}. Showing that the probability of each of these bounds goes to zero will suffice. We will only show that the first of these probabilities goes to zero and the remaining ones can be completed by a similar argument.
For notational convenience, define
\begin{align*}
\text{Av}_n^{*}([y,~z])&:=\frac{\sum_{k:y\leq X_{k:n}\leq z}(Y_{[k:n]}-f_{0,n}(X_{
(
k)}))}{|\{k:y\leq X_k\leq z\}|}\\
&= \frac{\sum_{k:y\leq X_{k:n}\leq z}\xi_{[k:n]}}{|\{k:y\leq X_k\leq z\}|},\quad\mbox{and}\\
\text{Av}_n^{f_{0,n}}([y,~z])&:=\frac{\sum_{k:y\leq X_{k:n}\leq z}f_{0,n}(X_{k:n})}{|\{k:y\leq X_{k:n}\leq z\}|}.
\end{align*}
Consider the first event.
\begin{align*}
\mathcal{E}_{1n} &= \left\{\min_{z\geq x_0}\text{Av}_n\left(\left(x_0-\beta_l(n),~z\right]\right)< f_{0,n}(x_0-\beta_l(n))\right\}\\
&=\left\{\min_{z\geq x_0}\left[\text{Av}_n^{*}\left(\left(x_0-\beta_l(n),~z\right]\right)+\text{Av}_n^{f_{0,n}}\left(\left(x_0-\beta_l(n),~z\right]\right)\right]< f_{0,n}(x_0-\beta_l(n))\right\}\\
&\subseteq \left\{\min_{z\geq x_0}\left[\text{Av}_n^{*}\left(\left(x_0-\beta_l(n),~z\right]\right)\right]+\min_{z\geq x_0}\left[\text{Av}_n^{f_{0,n}}\left(\left(x_0-\beta_l(n),~z\right]\right)\right]< f_{0,n}(x_0-\beta_l(n))\right\}\\
&\overset{(a)}{=}\left\{\min_{z\geq x_0}\left[\text{Av}_n^{*}\left(\left(x_0-\beta_l(n),~z\right]\right)\right]+\text{Av}_n^{f_{0,n}}\left(\left(x_0-\beta_l(n),~x_0\right]\right)< f_{0,n}(x_0-\beta_l(n))\right\}\\
&=\left\{\max_{z\geq x_0}\left[-\text{Av}_n^{*}\left(\left(x_0-\beta_l(n),~z\right]\right)\right]> -f_{0,n}(x_0-\beta_l(n))+\text{Av}_n^{f_{0,n}}\left(\left(x_0-\beta_l(n),~x_0\right]\right)\right\}\\
&=\left\{\max_{z\geq x_0}\left[\text{Av}_n^{**}\left(\left(x_0-\beta_l(n),~z\right]\right)\right]> v_n\right\},
\end{align*}
where equality~(a) follows from~\ref{assump:data-model-assumption}.
Here \[\text{Av}_n^{**}\left(\left(x_0-\beta_l(n),~z\right]\right):=-\text{Av}_n^{*}\left(\left(x_0-\beta_l(n),~z\right]\right)\]
and
\begin{align*}v_n&=-f_{0,n}(x_0-\beta_l(n))+\text{Av}_n^{f_{0,n}}\left(\left(x_0-\beta_l(n),~x_0\right]\right)\\
&=\frac{\sum_{k:x_0-\beta_l(n)\leq X_{k:n}\leq x_0}(f_{0,n}(X_{k:n})-f_{0,n}(x_0-\beta_l(n)))}{|\{k:x_0-\beta_l(n)\leq X_{k:n}\leq x_0\}|}.
\end{align*}
Note that $v_n$ is random only through $\mathcal{X} = \{X_1, \ldots, X_n\}$. We control the probability of $\mathcal{E}_{1n}$ by first conditioning on $\mathcal{X}$. In the definition of $\text{Av}_n^{**}(\left(x_0-\beta_l(n),~z\right]$, we only consider covariate observations which are greater than $x_0-\beta_l(n)$. This implies that we can write
\[
\text{Av}_n^{**}((x_0 - \beta_l(n), z]) = \max_{\delta_{l,-}(n) \le j\le n}\frac{1}{j}\sum_{i=k}^{k+j} (-\xi_{[i:n]}),
\]
where $k = \min\{1\le m\le n:\, X_{k:n} > x_0-\beta_l(n)\}$.
Then Lemma~\ref{lem:Hajek-Renyi-special-case} and assumption~\ref{assump:data-model-assumption} implies that
\begin{equation}\label{eq:Markov-inequality}
\begin{split}
   \mathbb{P}(\mathcal{E}_{1n}|\mathcal{X}) &=\mathbb{P}\left(\max_{z\geq x_0} \text{Av}_n^{**}(\left(x_0-\beta_l(n),~z\right])>v_n~|~\mathcal{X}\right)\\
    &\le \mathbb{P}\left(\max_{\delta_{l,-}(n) \le j\le n}\frac{1}{j}\sum_{i=k}^{k+j} (-\xi_{[i:n]})>v_n|~\mathcal{X}\right)\\
    &\le \min\left\{\frac{4\overline{\sigma}^2}{v_n^2\delta_{l,-}(n)}, 1\right\}.
\end{split}
\end{equation}
Therefore,
\[
\mathbb{P}(\mathcal{E}_{1n}) \le \mathbb{E}\left[\min\left\{\frac{4\overline{\sigma}^2}{v_n^2\delta_{l,-}(n)}, 1\right\}\right].
\]
Lemma~\ref{lem:smoothness-of-H} implies that $s_n\delta_{l,-}(n)/n \overset{p}{\to} c$ as $n\to\infty$. In Lemma~\ref{lem:control-of-v_n}, we show that
\begin{equation}\label{eq:limit-of-v_n}
\sqrt{\frac{n}{s_n}}v_n ~\overset{p}{\to}~ \left[-\psi(-c/h_0) - (h_0/c)\Psi(-c/h_0)\right]\quad\mbox{as}\quad n\to\infty.
\end{equation}
Combining these two, we get that
\begin{align*}
\frac{4\overline{\sigma}^2}{v_n^2\delta_{l,-}(n)}
&\overset{p}{\to} \frac{4\overline{\sigma}^2}{c[-\psi(-c/h_0) - (h_0/c)\Psi(-c/h_0)]}
\quad\mbox{as}\quad n\to\infty.
\end{align*}
Because $x\mapsto \min\{x, 1\}$ is a continuous bounded function, we conclude that for all $c > 0$,
\begin{align*}
\limsup_{n\to\infty}\,\mathbb{P}(\mathcal{E}_{1n}) &\le \min\left\{\frac{4\overline{\sigma}^2}{c[-\psi(-c/h_0) - (h_0/c)\Psi(-c/h_0)]},\,1\right\}\\
&\le \frac{4\overline{\sigma}^2}{c[-\psi(-c/h_0) - (h_0/c)\Psi(-c/h_0)]}.
\end{align*}
In dealing with the remaining three events from Lemma~\ref{lem:unequalness-of-fn-star-and-fhat}, we encounter the following analogues of $v_n$:
\begin{align*}
v_{2,n} &:= f_{0,n}(x_0 - \beta_l(n)) - \max_{y \le x_0 - \alpha_l(n)}\text{Av}_n([y, x_0 - \beta_l(n)])\\
&= f_{0,n}(x_0 - \beta_l(n)) - \text{Av}_n([x_0 - \alpha_l(n), x_0 - \beta_l(n)]),\\
v_{3,n} &:= f_{0,n}(x_0 + \beta_u(n)) - \text{Av}_n((x_0, x_0 + \beta_u(n)])\\
v_{4,n} &:= \text{Av}_n([x_0 + \beta_u(n), x_0 + \alpha_u(n)]) - f_{0,n}(x_0 + \beta_u(n)).
\end{align*}
Following the same proof technique as in Lemma~\ref{lem:control-of-v_n}, it can be proved that as $n\to\infty$,
\begin{align*}
\sqrt{\frac{n}{s_n}}v_{2,n} &\overset{p}{\to} [\psi(-c/h_0) + (h_0/c)(\Psi(-2c/h_0) - \Psi(-c/h_0))],\\
\sqrt{\frac{n}{s_n}}v_{3,n} &\overset{p}{\to} [\psi(c/h_0) - (h_0/c)\Psi(c/h_0)],\\
\sqrt{\frac{n}{s_n}}v_{4,n} &\overset{p}{\to} [-\psi(c/h_0) + (h_0/c)(\Psi(2c/h_0) - \Psi(c/h_0))].
\end{align*}
Therefore,
\begin{align*}
\limsup_{n\to\infty} \mathbb{P}(f_{n,c}^{*}(x_0) \neq \hat{f}_n(x_0))
&\le \frac{4\overline{\sigma}^2}{c}\Big[\frac{1}{\Gamma_1(c/h_0)} + \frac{1}{\Gamma_1(-c/h_0)}\\
&\quad+ \frac{1}{\Gamma_2(c/h_0)} + \frac{1}{\Gamma_2(-c/h_0)}\Big].
\end{align*}
\end{proof}
\begin{lemma}\label{lem:brownian-motion}
Consider the process $U_n(\cdot)$ defined in~\eqref{eq:linear-interpolation-scaled}.
Under assumption~\ref{assump:data-model-assumption}, as $n\to\infty$,
\[
\left(\sqrt{\frac{2n/s_n}{D_0\sigma_0^2}}(U_n(t) - \mathbb{E}[U_n(t)|\mathcal{X}])\right)_{t\in[0,2cD_0]}\overset{d}{\to} (B(t))_{t\in[0, 2cD_0]},\]
where $B(t)$ denotes the standard Brownian motion on $\mathbb{R}$.
\end{lemma}
\begin{proof}
Setting $\xi_{[l:\gamma_n]}' := Y_{[l:\gamma_n]}'-f_{0,n}(X_{(l)}')$, from the definition of $U_n(\cdot)$, it follows that
\[
U_{n}(t)-\mathbb{E}[U_n(t)|\mathcal{X}]=2cD_0\left(\frac{1}{\gamma_n}\sum_{l=1}^{\lfloor \gamma_nt/(2cD_0)\rfloor} \xi_{[l:\gamma_n]}' + \frac{\gamma_nt/(2cD_0)-\lfloor \gamma_nt/(2cD_0)\rfloor}{\gamma_n}\xi_{[\lfloor \gamma_nt/(2cD_0)\rfloor+1:\gamma_n]}'\right).
\]
The right-hand side is a scaled average of independent random variables with mean zero conditional on $X_1, \ldots, X_n.$ Hence, we get
\begin{align*}
\mbox{Var}\left(U_n(t)|\mathcal{X}\right) &= (2cD_0)^2\left(\frac{1}{\gamma_n^2}\sum_{l=1}^{\lfloor \gamma_nt/(2cD_0)\rfloor} \sigma_n^2(X_{(l)}')\right)\\
&\quad+ (2cD_0)^2\left(\frac{\gamma_nt/(2cD_0)-\lfloor \gamma_nt/(2cD_0)\rfloor}{\gamma_n}\right)^2\sigma_n^2(X_{[\lfloor \gamma_nt/(2cD_0)\rfloor +1:\gamma_n]}).
\end{align*}
This implies that for any $t\in[0, 2cD_0]$,
\[
\left|\mbox{Var}\left(\frac{\gamma_n^{1/2}U_n(t)}{(2cD_0)^{1/2}}\big|\mathcal{X}\right) - \frac{2cD_0}{\gamma_n}\sum_{l=1}^{\lfloor \gamma_nt/(2cD_0)\rfloor} \sigma_n^2(X_{(l)}')\right| \le \frac{\overline{\sigma}^2}{\gamma_n}\to 0,\quad\mbox{as}\quad n\to\infty.
\]
Because $\alpha_l, \alpha_u$ both converge to zero as $n\to\infty$ (from Lemma~\ref{lem:smoothness-of-H}), we conclude that from assumption~\ref{assump:data-model-assumption},
\[
\sup_{1\le l\le \gamma_n}|\sigma_n^2(X_{(l)}') - \sigma^2_0| \le \sup_{-\alpha_l \le x - x_0 \le \alpha_u}|\sigma_n^2(x) - \sigma_n^2(x_0)| + |\sigma_n^2(x_0) - \sigma_0^2| \to 0.
\]
Combining this with the variance expression, we obtain that
\[
\left|\mbox{Var}\left(\frac{\gamma_n^{1/2}U_n(t)}{(2cD_0)^{1/2}}\big|\mathcal{X}\right) - \frac{2cD_0\lfloor \gamma_nt/(2cD_0)\rfloor}{\gamma_n}\sigma_0^2\right| \to 0,\quad\mbox{as}\quad n\to\infty.
\]
Because $\lfloor \gamma_nt/(2cD_0)\rfloor/\gamma_n\to t/(2cD_0)$ as $n\to\infty$, we conclude that $$\mbox{Var}(\gamma_n^{1/2}U_n(t)/(2cD_0\sigma_0^2)^{1/2}|\mathcal{X}) \to t,\quad\mbox{as}\quad n\to\infty.$$ One can now follow the proof of Donsker's invariance theorem to claim that
\[
\left(\frac{\gamma_n^{1/2}}{(2cD_0\sigma_0^2)^{1/2}}(U_n(t) - \mathbb{E}[U_n(t)|\mathcal{X}])\right)_{t\in[0, 2cD_0]} ~\overset{d}{\to}~ (B(t))_{t\in[0,2cD_0]},
\]
conditional on $\mathcal{X}$, where $(B(t))_{t\in[0,2cD_0]}$ is the standard Brownian motion on $[0, 2cD_0]$. Because the limiting process does not depend on $\mathcal{X}$, and because $\gamma_n\sim\mbox{Binomial}(n, 4c/s_n)$ (which implies $s_n\gamma_n/(4cn) \overset{p}{\to}1$), we get
\[
\left(\sqrt{\frac{2n/s_n}{D_0\sigma_0^2}}(U_n(t) - \mathbb{E}[U_n(t)|\mathcal{X}])\right)_{t\in[0,2cD_0]} \overset{d}{\to} (B(t))_{t\in[0, 2cD_0]}.
\]
This completes the proof of Lemma~\ref{lem:brownian-motion}.
\end{proof}
\begin{lemma}\label{lem:deterministic-part-final}
Under assumptions~\ref{assump:data-model-assumption}--\ref{assump:continuity-of-f_{0,n}}, for any $c > 0$, as $n\to\infty$,
\[
\sup_{t\in[0, 2cD_0]}\left|\sqrt{\frac{n}{s_n}}(\mathbb{E}[U_n(t)|\mathcal{X}] - f_{0,n}(x_0)t) - \frac{D_0h_0}{2}\left[-\Psi\left(-\frac{2c}{h_0}\right) + \Psi\left(\frac{1}{h_0}\left(\frac{2t}{D_0} - 2c\right)\right)\right]\right| \overset{p}{\to} 0.
\]
\end{lemma}
\begin{proof}
Note that $U_n(t)$ can be alternatively written as
\[
U_n(t) = 2cD_0\frac{1}{\gamma_n}\sum_{l=1}^{\gamma_n} I_l(t)Y'_{[l:\gamma_n]},
\]
where
\[
I_l(t) := \begin{cases}1, &\mbox{if }l \le \lfloor \gamma_nt/(2cD_0)\rfloor,\\
\gamma_nt/(2cD_0) - \lfloor\gamma_nt/(2cD_0)\rfloor, &\mbox{if }l = \lfloor\gamma_n t/(2cD_0)\rfloor + 1,\\
0, &\mbox{if }l > \lfloor\gamma_n t/(2cD_0)\rfloor + 1.\end{cases}
\]
Observe that $\sum_{l=1}^{\gamma_n} I_l(t) = \gamma_nt/(2cD_0)$ and hence
\[
R(t) := \mathbb{E}[U_n(t)|\mathcal{X}] - f_{0,n}(x_0)t = \frac{2cD_0}{\gamma_n}\sum_{l=1}^{\gamma_n}I_l(t)\left[f_{0,n}(X_{(l)}') - f_{0,n}(x_0)\right].
\]
Because $I_l(\cdot), 1\le l\le \gamma_n$ are non-negative and $f_{0,n}(\cdot)$ is non-decreasing, there exists a $T\in[0,2cD_0]$ such that $t\mapsto R(t)$ is non-increasing for $t \le T$ and non-decreasing for $t \ge T$. In fact, $T = (2cD_0/\gamma_n)\min\{l\in\{1,2,\ldots,\gamma_n\}:\,X_{(l)}'\ge x_0\}$. This fact implies that it suffices to study the behavior of $R(\cdot)$ at $t = (2cD_0k/\gamma_n)$ because one can sandwich $R(\cdot)$ at $t\in[2cD_0k/\gamma_n, 2cD_0(k+1)/\gamma_n)$ by $R(\cdot)$ at either endpoint. Fix $1\le k\le \gamma_n$. Observe that
\[
R(2cD_0k/\gamma_n) = \frac{2cD_0}{\gamma_n}\sum_{l=1}^k [f_{0,n}(X_{(l)}') - f_{0,n}(x_0)].
\]
This implies that
\begin{equation}\label{eq:representation-inverse-CDF}
\left(R(2cD_0k/\gamma_n)\right)_{k=1,2,\ldots,\gamma_n} = \frac{2cD_0n}{\gamma_n}\left(\mathbb{M}_n(X_{k:n}')\right)_{k=1,2,\ldots,\gamma_n},
\end{equation}
where
\begin{equation}\label{eq:stochastic-M-n-defined}
\mathbb{M}_n(\zeta) = \int_{x_0 - \alpha_l}^{\zeta} [f_{0,n}(x) - f_{0,n}(x_0)]d\widehat{H}_n(x),\quad\zeta\in(x_0 - \alpha_l, x_0 + \alpha_u).
\end{equation}
Note that the piecewise constant interpolation of $(R(2cD_0k/\gamma_n))_{k=1,2,\ldots,\gamma_n}$ is exactly equal to $(\mathbb{M}_n(\zeta))_{\zeta\in(x_0 - \alpha_l, x_0 + \alpha_u)}$, because $\zeta\mapsto \mathbb{M}_n(\zeta)$ is a constant on $[X_{k:n}', X_{(k+1)}')$ for all $1\le k\le \gamma_n - 1$. Define
\[
M_n(\zeta) = \int_{x_0 - \alpha_l}^{\zeta} [f_{0,n}(x) - f_{0,n}(x_0)]dH_n(x),\quad\zeta\in(x_0 - \alpha_l, x_0 + \alpha_u).
\]
Clearly, $M_n(\zeta) = \mathbb{E}[\mathbb{M}_n(\zeta)]$ for any $\zeta$.
Lemma~\ref{lem:study-of-stochastic-part} implies that as $n\to\infty$,
\begin{equation}\label{eq:stochastic-part-implication}
\sqrt{\frac{n}{s_n}}\sup_{t\in[0, 2cD_0]}\left|R(t) - \frac{2cD_0n}{\gamma_n}M_n\left(X'_{(\lfloor \gamma_n t/(2cD_0)\rfloor)})\right)\right| = o_p(1).
\end{equation}
To see this, note that the limits in~\eqref{eq:lemma-in-deterministic-part} prove that $(R(t))_{t\in[0, 2cD_0]}$ is asymptotically the same as the piecewise constant interpolation of $(R(2cD_0k/\gamma_n))_{k=1,2,\ldots,\gamma_n}$, in the uniform sense, at the rate of $\sqrt{n/s_n}(n/\gamma_n)$. Let us call the piecewise constant interpolation as $(\tilde{R}(t))_{t\in[0, 2cD_0]}$. Formally, we have
\[
\tilde{R}(t) = \frac{2cD_0n}{\gamma_n}\int_{x_0 - \alpha_l}^{X'_{(k)}} [f_{0,n}(x) - f_{0,n}(x_0)]d\widehat{H}_n(x), \quad\mbox{for}\quad t\in[2cD_0k/\gamma_n, 2cD_0(k+1)/\gamma_n).
\]
Now by the first limit of~\eqref{eq:lemma-in-deterministic-part}, we get
\[
\sqrt{\frac{n}{s_n}}\sup_{t\in[0, 2cD_0]}\left|\tilde{R}(t) - \frac{2cD_0n}{\gamma_n}M_n(X'_{(\lfloor \gamma_nt/(2cD_0)\rfloor)})\right| = o_p(1),
\]
which implies~\eqref{eq:stochastic-part-implication}.
To prove the result, it now suffices to study the convergence of $M_n(X'_{(\lfloor \gamma_n t/(2cD_0)\rfloor)})$. The limit statement~\eqref{eq:limit-of-M_n} of Lemma~\ref{lem:study-of-stochastic-part} implies
\[
\sup_{t\in[0, 2cD_0]}\left|\sqrt{ns_n}M_n(X'_{(\lfloor\gamma_n t/(2cD_0\rfloor))}) - h_0\left[-\Psi\left(-\frac{2c}{h_0}\right) + \Psi\left(s_n(X'_{(\lfloor\gamma_n t/(2cD_0\rfloor))} - x_0)\right)\right]\right| = o_p(1).
\]
Moreover, since $s_n\gamma_n/n = 4c(1 + o_p(1))$, this is equivalent to
\[
\sup_{t\in[0, 2cD_0]}\left|\sqrt{\frac{n}{s_n}}\frac{2cD_0n}{\gamma_n}M_n(X'_{(\lfloor\gamma_n t/(2cD_0\rfloor))}) - \frac{D_0h_0}{2}\left[-\Psi\left(-\frac{2c}{h_0}\right) + \Psi\left(s_n(X'_{(\lfloor\gamma_n t/(2cD_0\rfloor))} - x_0)\right)\right]\right| = o_p(1).
\]
Because $\Psi(\cdot)$ is a continuous function, it is uniformly continuous on bounded intervals, and hence, by the final limit of~\eqref{eq:lemma-in-deterministic-part}, we get
\[
\sup_{t\in[0, 2cD_0]}\left|\Psi\left(s_n(X'_{(\lfloor\gamma_n t/(2cD_0\rfloor))} - x_0)\right) - \Psi\left(\frac{1}{h_0}\left(\frac{2t}{D_0} - 2c\right)\right)\right| = o_p(1).
\]
This completes the proof.
\end{proof}
\begin{lemma}\label{lem:study-of-stochastic-part}
Consider the setting in the proof of Lemma~\ref{lem:deterministic-part-final}. Then as $n\to\infty$,
\begin{equation}\label{eq:lemma-in-deterministic-part}
\begin{split}
\sqrt{\frac{n}{s_n}}\frac{n}{\gamma_n}\sup_{\zeta\in(x_0 - \alpha_l, x_0 + \alpha_u)}|\mathbb{M}_n(\zeta) - M_n(\zeta)| ~&\overset{p}{\to}~ 0,\\
\sqrt{\frac{n}{s_n}}\frac{n}{\gamma_n}\sup_{1\le k\le \gamma_n - 1}|M_n(X_{(k+1)}') - M_n(X_{(k)}')| ~&\overset{p}{\to}~ 0,\\
\sqrt{\frac{n}{s_n}}\sup_{t\in[0, 2cD_0]}\left|\frac{1}{h_0}\left(\frac{2t}{D_0} - 2c\right) - s_n\left(X'_{(\lfloor \gamma_nt/(2cD_0)\rfloor)} - x_0\right)\right| ~&=~ O_p(1).
\end{split}
\end{equation}
Moreover, for any $c > 0$, as $n\to\infty$
\begin{equation}\label{eq:limit-of-M_n}
\sup_{u\in[-2c-1, 2c + 1]}\left|\sqrt{ns_n}M_n\left(x_0 + \frac{u}{s_nh_0}\right) - h_0\left[-\Psi\left(-\frac{2c}{h_0}\right) + \Psi\left(\frac{u}{h_0}\right)\right]\right| \to 0.
\end{equation}
\end{lemma}
\begin{proof}
For the proof of the first limit in~\eqref{eq:lemma-in-deterministic-part}, note that
\begin{align*}
\sup_{-\alpha_l < \zeta - x_0 < \alpha_u}|\mathbb{M}_n(\zeta) - M_n(\zeta)| \le \sup_{-\alpha_l < s < \alpha_u}\left|\int_{x_0 - \alpha_l}^{x_0 + s} [f_{0,n}(x) - f_{0,n}(x_0)]d(\widehat{H}_n - H_n)(x)\right|.
\end{align*}
Also, observe that for all $x\in(x_0 - \alpha_l, x_0 + \alpha_u)$,
\[
f_{0,n}(x_0 - \alpha_l) - f_{0,n}(x_0) \le f_{0,n}(x) - f_{0,n}(x_0) \le f_{0,n}(x_0 + \alpha_u) - f_{0,n}(x_0).
\]
This implies that
\begin{align*}
&\sup_{-\alpha_l < x - x_0 < \alpha_u}\sqrt{\frac{n}{s_n}}|f_{0,n}(x) - f_{0,n}(x_0)|\\
&\quad\le \sqrt{\frac{n}{s_n}}|f_{0,n}(x_0 - \alpha_l) - f_{0,n}(x_0)| + \sqrt{\frac{n}{s_n}}|f_{0,n}(x_0 + \alpha_u) - f_{0,n}(x_0)|,
\end{align*}
and by assumption~\ref{assump:continuity-of-f_{0,n}} (combined with Lemma~\ref{lem:smoothness-of-H}), the right hand side converges to $2\psi(2c/h_0)$ as $n\to\infty.$ Hence, there exists an $N \ge 1$ such that for all $n \ge N$, we have
\begin{equation}\label{eq:bound-for-f_{0n}}
\sqrt{\frac{n}{s_n}}\sup_{x_0 - \alpha_l < x < x_0 + \alpha_u}|f_{0,n}(x) - f_{0,n}(x_0)| ~\le~ (2\psi(2c/h_0) + 1).
\end{equation}
Define $r_n(x) = \sqrt{n/s_n}(f_{0,n}(x) - f_{0,n}(x_0))$. Under assumption~\ref{assump:data-model-assumption}, it is clear that $r_n(\cdot)$ is a non-decreasing function and we also have $|r_n(x)| \le (2\psi(2c/h_0) + 1)$ for all $x\in(x_0 - \alpha_l, x_0 + \alpha_u)$ and $n\ge N$. Note that
\begin{align*}
&\mathbb{E}\left[\sqrt{\frac{n}{s_n}}\sup_{-\alpha_l < \zeta - x_0 < \alpha_u} |\mathbb{M}_n(\zeta) - M_n(\zeta)|\right]\\ &\le \mathbb{E}\left[\sup_{-\alpha_l < s < \alpha_u}\left|\int_{x_0 - \alpha_l}^{x_0 + s} r_n(x)d(\widehat{H}_n - H_n)(x)\right|\right]\\
&\le \frac{2}{\sqrt{n}}\mathbb{E}\left[\sup_{-\alpha_l < s < \alpha_u}\left|\frac{1}{\sqrt{n}}\sum_{i=1}^n \varepsilon_ir_n(X_i)\mathbf{1}\{x_0 - \alpha_l < X_i \le x_0 + s\}\right|\right],
\end{align*}
where for independent Rademacher variables $\varepsilon_i, 1\le i\le n$ (i.e., $\mathbb{P}(\varepsilon_i = -1) = \mathbb{P}(\varepsilon_i = 1) = 1/2$), the last inequality follows from the symmetrization inequality~\cite[Lemma 2.3.6]{van1996weak}. Now, we deal with the right-hand side by first conditioning on $X_1, \ldots, X_n$. Theorem 3.1.17 of~\cite{gine2021mathematical} implies that
\begin{align*}
&\mathbb{E}\left[\sup_{-\alpha_l < s < \alpha_u}\left|\frac{1}{\sqrt{n}}\sum_{i=1}^n \varepsilon_ir_n(X_i)\mathbf{1}\{x_0 - \alpha_l < X_i \le x_0 + s\}\right|\big|\mathcal{X}_n\right]\\
&\quad\le (2\psi(2c/h_0) + 1)\mathbb{E}\left[\sup_{-\alpha_l < s < \alpha_u}\left|\frac{1}{\sqrt{n}}\sum_{i=1}^n \varepsilon_i\mathbf{1}\{x_0 - \alpha_l < X_i \le x_0 + s\}\right|\big|\mathcal{X}_n\right].
\end{align*}
Hence, by the symmetrization inequality,
\begin{align*}
&\mathbb{E}\left[\sqrt{\frac{n}{s_n}}\sup_{-\alpha_l < \zeta - x_0 < \alpha_u}|\mathbb{M}_n(\zeta) - M_n(\zeta)|\right]\\
&\quad\le 4(2\psi(2c/h_0) + 1)\mathbb{E}\left[\sup_{-\alpha_l < s < \alpha_u}\int_{-\infty}^{\infty}\mathbf{1}\{-\alpha_l < x - x_0 < s\}d(\widehat{H}_n - H_n)(x)\right].
\end{align*}
The class of intervals $\{[x_0 - \alpha_l, x_0 + s]:\, -\alpha_l < s < \alpha_u\}$ is a VC class with VC dimension 2; see Section 2.2 of~\url{http://maxim.ece.illinois.edu/teaching/fall14/notes/VC.pdf}. This implies that the class of functions $\{x\mapsto \mathbf{1}\{x\in[x_0 - \alpha_l, x_0 + s]:\, -\alpha_l < s < \alpha_u\}\}$ is a VC class of VC dimension 2 with the envelope function $F(x) = \mathbf{1}\{x\in[x_0 - \alpha_l, x_0 + \alpha_u]\}$. Therefore, Theorem 2.14.1 of~\cite{van1996weak} implies the existence of a universal constant $\mathfrak{C} \in (0, \infty)$ such that
\begin{align*}
\mathbb{E}\left[\frac{n}{s_n^{1/2}}\sup_{-\alpha_l < \zeta - x_0 < \alpha_u}|\mathbb{M}_n(\zeta) - M_n(\zeta)|\right] \le 4\mathfrak{C}(2\psi(2c/h_0) + 1)(H_n(x_0 + \alpha_u) - H_n(x_0 - \alpha_l))^{1/2},
\end{align*}
which by definition of $\alpha_l, \alpha_u$ is equal to $4\mathfrak{C}(2\psi(2c/h_0) + 1)\sqrt{2c/s_n}$. Hence, for any $c > 0$, as $n\to\infty$,
\begin{align*}
\sqrt{\frac{n}{s_n}}\frac{n}{\gamma_n}\sup_{-\alpha_l < \zeta - x_0 < \alpha_u}|\mathbb{M}_n(\zeta) - M_n(\zeta)| = O_p\left(\sqrt{\frac{n}{s_n}}\frac{2cD_0n}{\gamma_n}\frac{4\mathfrak{C}(2\psi(2c/h_0) + 1)\sqrt{2c}}{n}\right) = o_p(1),
\end{align*}
because $\gamma_n$ scales like $n/s_n$ and $\gamma_n \to \infty$ as $n\to\infty$ (which follows from the assumption that $s_n/n\to0$ as $n\to\infty$).
Now to prove the second limit in~\eqref{eq:lemma-in-deterministic-part}, note that from~\eqref{eq:bound-for-f_{0n}} for $n > N$,
\[
|M_n(X_{(k+1)}') - M_n(X_{(k)}')| \le \sqrt{\frac{s_n}{n}}(2\psi(2c/h_0) + 1)(H_n(X_{(k+1)}') - H_n(X_{(k)}')).
\]
Furthermore, observe that we can write $X_1, \ldots, X_n$ as $X_i = H_n^{-1}(U_i)$ for standard uniform random variables $U_1, \ldots, U_n$ and hence write $H_n(X_{(l)}') = U_{(l)}'$ for the subset of uniform order statistics $U_{(1)} \le U_{(2)} \le \cdots \le U_{(n)}$ belonging to the interval $(H_n(x_0 - \alpha_l), H_n(x_0 + \alpha_u))$.
It is also well-known that
\begin{equation}\label{eq:represetation-uniform-order-statistics}
(U_{(1)}, \ldots, U_{(n)}) ~\overset{d}{=}~ \left(\frac{W_1}{\sum_{j=1}^{n+1}W_j}, \cdots, \frac{W_1 + \cdots + W_n}{\sum_{j=1}^{n+1}W_j}\right),
\end{equation}
for independent standard exponential random variables $W_1, \ldots, W_n$. Combining these facts, we obtain
\begin{align*}
&\sup_{1\le k\le \gamma_n - 1}\left|M_n(X_{(k+1)}') - M_n(X_{(k)}')\right| \le \sqrt{\frac{s_n}{n}}(2\psi(2c/h_0) + 1)\frac{\max_{2\le l\le \gamma_n} W_l}{\sum_{j=1}^{n+1}W_j},
\end{align*}
where we abused the equality in distribution in~\eqref{eq:represetation-uniform-order-statistics} to mean equality almost surely which can be done without loss of generality by defining uniform order statistics through that relation. Because $\max_{2\le l\le \gamma_n} W_l = O_p(\log\gamma_n)$ and $(n+1)^{-1}\sum_{j=1}^{n+1}W_j = O_p(1)$, we conclude that
\[
\sup_{1\le k\le \gamma_n - 1}\left|M_n(X_{(k+1)}') - M_n(X_{(k)}')\right| = O_p\left(\sqrt{\frac{s_n}{n}}\frac{\log\gamma_n}{n}\right),
\]
and hence, as $n\to\infty,$
\[
\sqrt{\frac{n}{s_n}}\frac{n}{\gamma_n}\sup_{1\le k\le \gamma_n - 1}\left|M_n(X_{(k+1)}') - M_n(X_{(k)}')\right| = o_p(1).
\]
This completes the proof of the second limit of~\eqref{eq:lemma-in-deterministic-part}. Finally, for the last limit of~\eqref{eq:lemma-in-deterministic-part}, note that
\[
\widehat{H}_n(X'_{(\lfloor \gamma_nt/(2cD_0)\rfloor})) - \widehat{H}_n(x_0 - \alpha_l) = \frac{1}{n}\left\lfloor\frac{\gamma_n t}{2cD_0}\right\rfloor.
\]
Moreover, setting $\widehat{H}_n((a, b]) = \widehat{H}_n(b) - \widehat{H}_n(a)$ and similarly for $H_n(\cdot)$, Theorem 2.14.1 of~\cite{van1996weak} implies
\begin{equation}\label{eq:VC-CDF-diff}
\sup_{\zeta\in(x_0 - \alpha_l, x_0 + \alpha_u)}|\widehat{H}_n((x_0 - \alpha_l, \zeta]) - H_n((x_0 - \alpha_l, \zeta])| = O_p\left(\frac{1}{\sqrt{ns_n}}\right).
\end{equation}
(See the proof of Lemma~\ref{lem:study-of-stochastic-part} for a very similar application of Theorem 2.14.1 of~\cite{van1996weak}.) Therefore,
\[
\sqrt{ns_n}\sup_{t\in[0, 2cD_0]}\left|\frac{\gamma_n t}{2cD_0n} - H_n\left((x_0 - \alpha_l, X'_{(\lfloor \gamma_n t/(2cD_0)\rfloor)}]\right)\right| = O_p(1).
\]
Because $X'_{(\lfloor \gamma_nt/(2cD_0)\rfloor)}\in(x_0 - \alpha_l, x_0 + \alpha_u)$ and the density of covariates is uniformly close to $h_0$ on this interval as $n\to\infty$, we get that
\[
\sup_{\zeta\in(x_0 - \alpha_l, x_0 + \alpha_u)}\left|\frac{H_n((x_0 - \alpha_l, \zeta])}{(\zeta - x_0 + \alpha_l)h_0} - 1\right| = o(1),
\]
which further implies that
\[
\sqrt{ns_n}\sup_{t\in[0, 2cD_0]}\left|\frac{\gamma_n t}{2cD_0n} - h_0\left(X'_{(\lfloor \gamma_nt/(2cD_0)\rfloor)} - x_0 + \alpha_l\right)\right| = O_p(1).
\]
Equivalently,
\[
\sqrt{\frac{n}{s_n}}\sup_{t\in[0, 2cD_0]}\left|\frac{s_n\gamma_n t}{2cD_0n} - s_nh_0\left(X'_{(\lfloor \gamma_nt/(2cD_0)\rfloor)} - x_0) - s_n\alpha_l\right)\right| = O_p(1).
\]
From Lemma~\ref{lem:smoothness-of-H}, we know that $s_nh_0\alpha_l = 2c(1 + o(1))$ and from~\eqref{eq:VC-CDF-diff} (with $\zeta = x_0 + \alpha_u$), $s_n\gamma_n/n = 4c(1 + o(1))$ as $n\to\infty$. This implies the final limit of~\eqref{eq:lemma-in-deterministic-part}.
To prove~\eqref{eq:limit-of-M_n}, we first note that it suffices to prove pointwise convergence: as stated in Proposition~\ref{prop:verification-of-A3}, Section~0.1 of~\cite{resnick2008extreme} implies that pointwise convergence of monotone functions to a continuous function is in fact uniform convergence. This implies that pointwise convergence of a sequence of functions that are all unimodal at a fixed point is also uniform convergence. In our case, the functions $M_n(\zeta) = \int_{x_0 - \alpha_l}^\zeta [f_{0,n}(x) - f_{0,n}(x_0)]dH_n(x)$ are decreasing on $(x_0 - \alpha_l, x_0]$ and are increasing on $[x_0, x_0 + \alpha_u)$ for all $n\ge1$. Therefore, to show uniform convergence of $\sqrt{ns_n}M_n(x_0 + u/(s_nh_0))$ to a continuous function, it suffices to study pointwise convergence.
First, consider the case $u \le 0$. Observe that because $f_{0,n}(x) - f_{0,n}(x_0)$ does not change sign on $x\in[x_0-\alpha_l, x_0]$, we get that
\begin{align*}
    M_n\left(x_0 + \frac{u}{s_nh_0}\right) &= \int_{x_0 - \alpha_l}^{x_0 + u/(s_nh_0)} [f_{0,n}(x) - f_{0,n}(x_0)]dH_n(x)\\
    &\in \int_{x_0 - \alpha_l}^{x_0 + u/(s_nh_0)} [f_{0,n}(x) - f_{0,n}(x_0)]dx\times\left[\inf_{x\in(x_0 - \alpha_l, x_0]}h_n(x),\,\sup_{x\in(x_0 - \alpha_l, x_0]}h_n(x)\right].
\end{align*}
Both endpoints of the interval above converge to $h_0$ as $n\to\infty$ by assumption~\ref{assump:continuous-distribution-of-covariates}. To prove the convergence of the integral, note that
\begin{align*}
    \int_{x_0 - \alpha_l}^{x_0 + u/(s_nh_0)} [f_{0,n}(x) - f_{0,n}(x_0)]dx &= \int_{-\alpha_l}^{u/(s_nh_0)} [f_{0,n}(x_0 + x) - f_{0,n}(x_0)]dx\\
    &= \alpha_l\int_{-1}^0 [f_{0,n}(x_0 + x\alpha_l) - f_{0,n}(x_0)]dx\\
    &\quad - \frac{1}{s_n}\int_{u/h_0}^0 [f_{0,n}(x_0 + x/s_n) - f_{0,n}(x_0)]dx.
\end{align*}
By Part 3 of Proposition~\ref{prop:assumptions-implications}, we obtain that as $n\to\infty$,
\begin{align*}
\sqrt{\frac{n}{s_n}}\int_{-1}^0 [f_{0,n}(x_0 + {x\alpha_l}) - f_{0,n}(x_0)]dx &\to -\frac{\Psi(-2c/h_0)}{(2c/h_0)}\\
\sqrt{\frac{n}{s_n}}\int_{u/h_0}^0 [f_{0,n}(x_0 + x/s_n) - f_{0,n}(x_0)]dx &\to -\Psi(u/h_0).
\end{align*}
Therefore, (using $s_n\alpha_l \to 2c/h_0$ as $n\to\infty$), for all $u \le 0$
\begin{equation}\label{eq:negative-u-limit}
\sqrt{ns_n}M_n(x_0 + u/(s_nh_0)) ~\to~ h_0\left[-\Psi(-2c/h_0) + \Psi(u/h_0)\right],\quad\mbox{as}\quad n\to\infty.
\end{equation}
Now, consider the case $u > 0$. Then we note that
\[
M_n(x_0 + u/(s_nh_0)) = M_n(x_0) + \int_{x_0}^{x_0 + u/(s_nh_0)} [f_{0,n}(x) - f_{0,n}(x_0)]dH_n(x).
\]
From the above calculation, we know
\[
\sqrt{ns_n}M_n(x_0) \to -h_0\Psi(-2c/h_0),\quad\mbox{as}\quad n\to\infty.
\]
To analyze the second integral, observe that because $f_{0,n}(x) - f_{0,n}(x_0)$ does not change sign on $x\in[x_0, x_0 + \alpha_u]$,
\begin{align*}
&\int_{x_0}^{x_0 + u/(s_nh_0)} [f_{0,n}(x) - f_{0,n}(x_0)]dH_n(x)\\
&\in \int_{x_0}^{x_0 + u/(s_nh_0)} [f_{0,n}(x) - f_{0,n}(x_0)]dx\times \left[\inf_{x\in[x_0, x_0 + \alpha_u]}h_n(x),\, \sup_{x\in[x_0, x_0 + \alpha_u]}h_n(x)\right].
\end{align*}
Both endpoints of the interval above converge to $h_0$ as $n\to\infty$ by assumption~\ref{assump:continuous-distribution-of-covariates}. To prove the convergence of the integral, note that
\begin{align*}
\sqrt{ns_n}\int_{x_0}^{x_0 + u/(s_nh_0)} [f_{0,n}(x) - f_{0,n}(x_0)]dx &= \sqrt{\frac{n}{s_n}}\int_{0}^{u/h_0} [f_{0,n}(x_0 + x/s_n) - f_{0,n}(x_0)]dx\\ &\to \Psi(u/h_0),\quad\mbox{as}\quad n\to\infty.
\end{align*}
Therefore, for $u > 0$,
\begin{equation}\label{eq:positive-u-limit}
\sqrt{ns_n}M_n(x_0 + u/(s_nh_0)) ~\to~ h_0[-\Psi(-2c/h_0) + \Psi(u/h_0)],\quad\mbox{as}\quad n\to\infty.
\end{equation}
Combining~\eqref{eq:negative-u-limit} and~\eqref{eq:positive-u-limit}, we get that for any $u\in\mathbb{R}$, as $n\to\infty$,
\[
\sqrt{ns_n}M_n(x_0 + u/(s_nh_0)) \to h_0[-\Psi(-2c/h_0) + \Psi(u/h_0)].
\]
This completes the proof.
\end{proof}
\section{Additional proofs for Section~\ref{sec:applications-asymptotic-distribution-of-LSE}}\label{appsec:applications}
\subsection{Verification for Example~\ref{exmp:Wright-1981}}
We verify assumption~\ref{assump:continuity-of-f_{0,n}} from the condition in Example~\ref{exmp:Wright-1981}.
Suppose $\{c_n\}$ is a real sequence such that $c_n\to c\in\mathbb{R}\backslash\{0\}$. Fix any sequence $s_n \ge 0$ with $s_n\to\infty$. Then,
\begin{align*}
\sqrt{\frac{n}{s_n}}\left(f_{0,n}\left(x_0+\frac{c_n}{s_n}\right)-f_{0,n}(x_0)\right)
&=\sqrt{\frac{n}{s_n}}A\left|x_0+\frac{c_n}{s_n}-x_0\right|^{\theta}\mathrm{sign}(c_n)(1+o(1))\\
&=\frac{n^{1/2}}{s_n^{\theta+1/2}}A\cdot |c_n|^{\theta}\mathrm{sign}(c_n)(1+o(1)).
\end{align*}
Because $c_n\to c\neq0$, we get $|c_n|^{\theta}\mathrm{sign}(c_n) \to |c|^{\theta}\mathrm{sign}(c)\neq 0$ as $n\to\infty$. Hence, ensuring that $n^{1/2}/s_n^{\theta + 1/2}$ converges to a positive constant provides the rate of convergence. Without loss of generality, taking the limiting constant to be $1$, we get that $s_n = n^{1/(2\theta + 1)}$ suffices. Thus,
\begin{equation}\label{eq:psi-and-s_n-Wright-app}
\psi(c)=A\cdot\mathrm{sign}(c)|c|^{\theta}\quad\mbox{and}\quad s_n = n^{1/(2\theta+1)}.
\end{equation}
This proves the verification of assumption~\ref{assump:continuity-of-f_{0,n}} for Example~\ref{exmp:Wright-1981}.
\subsection{Details for Example~\ref{exmp:regular-varying-functions}}
\label{appsec:regular-varying-details}
We claim that there exists a sequence $\{s_n\}_{n\ge1}$ such that assumption~\ref{assump:continuity-of-f_{0,n}} holds true with $\psi(c) = A|c|^{\theta}\mathrm{sign}(c)$. Because this $\psi(\cdot)$ is a continuous function, by Proposition~\ref{prop:verification-of-A3}, it suffices to verify~\ref{assump:continuity-of-f_{0,n}} with $c_n = c$ for all $n\ge1$.
For any $c\in\mathbb{R}$, we have
\begin{align*}
\sqrt{\frac{n}{s_n}}\left(f_{0,n}\left(x_0+\frac{c}{s_n}\right)-f_{0,n}(x_0)\right)
&=\sqrt{\frac{n}{s_n}}A\left|x_0+\frac{c}{s_n}-x_0\right|^{\theta}L\left(\frac{|c|}{s_n}\right)\mathrm{sign}(c)(1+o(1))\\
&=\frac{n^{1/2}}{s_n^{\theta+1/2}}A|c|^{\theta}L\left(|c|/{s_n}\right)\mathrm{sign}(c)(1+o(1))\\
&\overset{(a)}{=} \frac{n^{1/2}}{s_n^{\theta+1/2}}A|c|^{\theta}L\left(1/{s_n}\right)\mathrm{sign}(c)(1+o(1)),
\end{align*}
where equality (a) follows from the fact that $L(\cdot)$ is slowly varying at $0$ and $1/s_n\to 0$ as $n\to\infty$: for any fixed $c\neq 0$,
\[
\frac{L(|c|/s_n)}{L(1/s_n)} \to 1.
\]
Therefore, by choosing $s_n$ such that $n^{1/2}L(1/s_n)/s_n^{\theta+1/2} \to 1$ as $n\to\infty$, we can take
\[
\psi(c) = A|c|^{\theta}\mathrm{sign}(c),\quad c\in\mathbb{R}.
\]
This completes the verification of assumption~\ref{assump:continuity-of-f_{0,n}} for Example~\ref{exmp:regular-varying-functions}. As a concrete example for the rate, take $L(x)=\ln\left(|x|^{-1}\right)$. Then, taking $s_n = (\sqrt{n}\ln n/(2\theta+1))^{2/(2\theta+1)}$ satisfies $n^{1/2}L(1/s_n)/s_n^{\theta + 1/2} \to 1$ as $n\to\infty$. \\
\begin{figure}[!h]
    \centering
    \includegraphics[width = \textwidth]{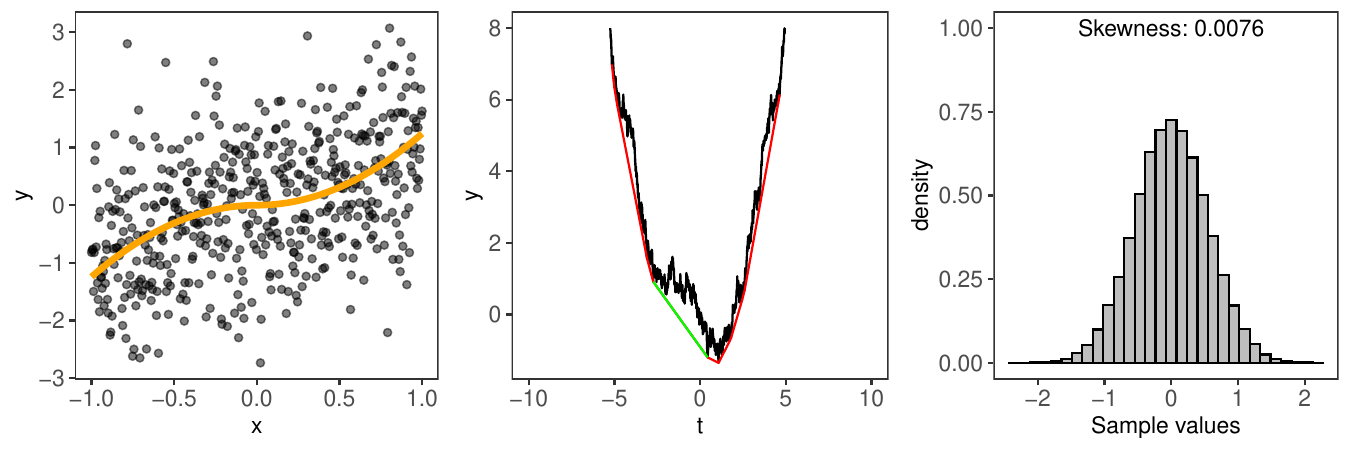}
    \caption{The left panel shows the scatter plot of data from $Y = f_{0,n}(X) + \xi,$ with $n = 500, X\sim\mathrm{Unif}(-1, 1), \xi|X\sim N(0, 1)$ and $f_{0,n}(x) = (\ln{n}/n^2)^{1/5}\psi((\sqrt{n}\ln n/5)^{2/5}X)$ along with the plot of $f_{0,n}(\cdot)$. The middle panel shows the plot of a two-sided Brownian motion with drift $|t|^3/24$ along with the greatest convex minorant; the green line shows the slope from the left of the greatest convex minorant at 0. The right panel shows the histogram of $10^5$ observations from the random variable $\slGCM.$}
    \label{fig:psi_eg2}
\end{figure}
\noindent To illustrate, we present the plot of limiting distribution when $\theta = 2, A = 1$ giving us,
\begin{equation}\label{eq:psi-and-s_n-exmp2}
\psi(c)=|c|^{2}\cdot\mathrm{sign}(c)\quad\mbox{and}\quad s_n = (\sqrt{n}\ln n/5)^{2/5}.
\end{equation}
We generate data using the same process as in Example~\ref{exmp:Wright-1981} in \eqref{eq:data-gen-process}, with $f_{0,n}$ defined as in~\eqref{eq:construction-of-f_0n} for our choice of $(\psi, s_n)$ in~\eqref{eq:psi-and-s_n-exmp2}. In Figure~\ref{fig:psi_eg2}, we present a scatterplot of a sample of size $n = 500$, the plot of $B(t) + |(t/2)|^{3}/3$ along with its greatest convex minorant ($B(\cdot)$ is the two-sided Brownian motion), and a histogram of $10^5$ observations from the random variable $\slGCM$, which represents the limiting distribution from Theorem~\ref{thm:asymptotic-distribution-of-LSE} in this example.
\subsection{Details for Example~\ref{exmp:locally-asymmetric}}
\label{appsec:locally-asymmetric-details}
It is clear that $\psi(\cdot)$ is a continuous function on $\mathbb{R}\setminus\{0\}$ and hence, by Proposition~\ref{prop:verification-of-A3}, it suffices to verify~\ref{assump:continuity-of-f_{0,n}} with $c_n = c \neq 0$ for all $n\ge1$.
Fix any $c\in\mathbb{R}\setminus\{0\}$. If $c > 0$, then
\begin{align*}
\sqrt{\frac{n}{s_n}}\left(f_{0,n}\left(x_0+\frac{c}{s_n}\right)-f_{0,n}(x_0)\right)
&=\sqrt{\frac{n}{s_n}}A_1\left(x_0+\frac{c}{s_n}-x_0\right)^{\theta_1}L_1\left(\frac{c}{s_n}\right)(1+o(1))\\
&=\frac{n^{1/2}}{s_n^{\theta_1+1/2}} A_1 c^{\theta_1}L_1\left(\frac{c}{s_n}\right)(1+o(1))\\
&=\frac{n^{1/2}}{s_n^{\theta_1+1/2}} L_1\left(\frac{1}{s_n}\right)A_1 c^{\theta_1}(1+o(1)),
\end{align*}
where the last equality follows from the assumption that $L_1(\cdot)$ is a slowly varying function at $0$, i.e.,
\[
\frac{L_1(c/s_n)}{L_1(1/s_n)} \to 1 \qquad (n\to\infty).
\]
Similarly, if $c<0$,
\begin{align*}
\sqrt{\frac{n}{s_n}}\left(f_{0,n}\left(x_0+\frac{c}{s_n}\right)-f_{0,n}(x_0)\right)
&=-\sqrt{\frac{n}{s_n}}A_2\left(-x_0-\frac{c}{s_n}+x_0\right)^{\theta_2}L_2(-c/s_n)(1+o(1))\\
&=-\frac{n^{1/2}}{s_n^{\theta_2+1/2}} A_2 (-c)^{\theta_2}L_2\left(\frac{-c}{s_n}\right)(1+o(1))\\
&=-\frac{n^{1/2}}{s_n^{\theta_2+1/2}}L_2\left(\frac{1}{s_n}\right)A_2 (-c)^{\theta_2}(1+o(1)),
\end{align*}
where the last equality follows from the assumption that $L_2(\cdot)$ is slowly varying at $0$. Define $s_n$ based on $\theta_1$ and $\theta_2$ as follows:
\begin{equation}\label{eq:rate-and-s_n-local-asymmetric-app}
\text{Choose $s_n$ s.t. }
\begin{cases}{n^{1/2}L_1(1/s_n)}/{s_n^{\theta_1+1/2}}\to 1, &\text{if }\theta_1\geq\theta_2,\\
{n^{1/2}L_2(1/s_n)}/{s_n^{\theta_2+1/2}}\to 1, &\text{if }\theta_1<\theta_2
\end{cases}
\end{equation}
Note that $s_n$ scales like $n^{1/(2\theta + 1)}$ up to a slowly varying (at $\infty$) factor depending on $n$, with $\theta := \max\{\theta_1, \theta_2\}$. This implies that if $\theta_1 \ge \theta_2$, then $n^{1/2}L_2(1/s_n)/s_n^{\theta_2 + 1/2} \to \infty\mathbf{1}_{\theta_1>\theta_2}+\mathbf{1}_{\theta_1=\theta_2}$. This, in turn, leads to $\psi(\cdot)$ defined in~\eqref{eq:psi-definition-local-asymmetry} as follows: if $\theta_1\ge \theta_2$, then the positive-side limit is finite and equals $A_1c^{\theta}$ for $c>0$, whereas the negative-side limit is $-A_2(-c)^{\theta}$ when $\theta_1=\theta_2$ and diverges to $-\infty$ when $\theta_1>\theta_2$. The case $\theta_1<\theta_2$ is analogous and yields divergence to $+\infty$ for $c>0$ while retaining a finite limit on the negative side. Hence, the pointwise limits are exactly those in \eqref{eq:psi-definition-local-asymmetry}, verifying assumption~\ref{assump:continuity-of-f_{0,n}} for this example.\\
\noindent For a better understanding of $s_n$ and the rate of convergence of isotonic LSE at $x_0$, consider as in the previous example of $L(x)=\ln\left(|x|^{-1}\right)$. We obtain with $\theta := \max\{\theta_1, \theta_2\},$
\begin{align*}
    s_n&=(\sqrt{n}\ln n/(2\theta+1))^{2/(2\theta+1)}\\
    \Rightarrow\quad \sqrt{\frac{n}{s_n}}&=n^{\theta/(2\theta+1)}((2\theta+1)/\ln n)^{1/(2\theta+1)}.
\end{align*}
\begin{figure}[!h]
    \centering
    \includegraphics[width = \textwidth]{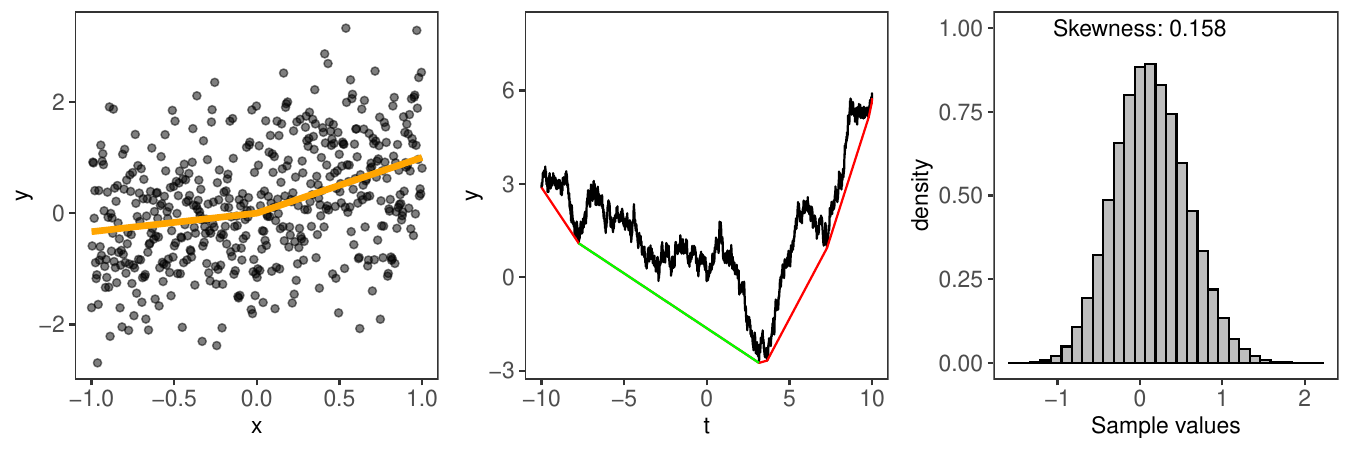}
    \caption{The left panel shows the scatter plot of data from $Y = f_{0,n}(X) + \xi,$ with $n = 500, X\sim\mathrm{Unif}(-1, 1), \xi|X\sim N(0, 1)$ and $f_{0,n}(x) = n^{-1/3}\psi_\text{asym}(n^{1/3}x)$ along with the plot of $f_{0,n}(\cdot)$. The middle panel shows the plot of a two-sided Brownian motion with an asymmetrical drift $\Psi_\text{asym}(t/2)$ along with the greatest convex minorant; the green line shows the slope from the left of the greatest convex minorant at 0. The right panel shows the histogram of $10^5$ observations from the random variable $\slGCM.$ $\psi_\text{asym}$ and $\Psi_\text{asym}$ are defined in~\eqref{eq:psi-and-s_n-exmp3}.}
    \label{fig:psi_eg3}
\end{figure}
To illustrate, we present the plot of limiting distribution when $A_1 = 1, A_2 = 1/3 $. Thus,
\begin{equation}\label{eq:psi-and-s_n-exmp3}
    \psi_{\mathrm{asym}}(t) = \begin{cases}
t & t \geq 0 \\
t/3 & t< 0 \\
\end{cases}, \quad s_n = n^{1/3}
\end{equation}
\[
\Psi_{\mathrm{asym}}(t) = \begin{cases}
t^2/2 & t \geq 0 \\
t^2/6 & t< 0 \\
\end{cases}.
\]
We generate data using the same process as in Example~\ref{exmp:Wright-1981} in~\eqref{eq:data-gen-process}, with $f_{0,n}$ defined as in \eqref{eq:construction-of-f_0n} for our choice of $(\psi_{\mathrm{asym}}, s_n)$ in~\eqref{eq:psi-and-s_n-exmp3}. In Figure~\ref{fig:psi_eg3}, we present a scatterplot of a sample of size $n = 500$, the plot of $B(t) + \Psi_{\mathrm{asym}}(t/2)$ along with its greatest convex minorant ($B(\cdot)$ is the two-sided Brownian motion), and a histogram of $10^5$ observations from the random variable $\slGCM$, which represents the limiting distribution from Theorem~\ref{thm:asymptotic-distribution-of-LSE} in this example.
\subsection{Details for Example~\ref{exmp:near-flat-functions}}
\label{appsec:near-flat-details}
We illustrate this case by generating data using the same process as in Example~\ref{exmp:Wright-1981} in~\eqref{eq:data-gen-process}, with $f_{0,n}$ defined in \eqref{eq:f_0n-exmp4}. Note that, unlike earlier illustrations, $f_{0,n}$ changes with $n$ here.
In Figure~\ref{fig:psi_eg4}, we present a scatterplot of a sample of size $n = 500$, the plot of $B(t) + |t/2|^2/2$ along with its greatest convex minorant ($B(\cdot)$ is the two-sided Brownian motion), and a histogram of $10^5$ observations from the random variable $\slGCM$, which represents the limiting distribution from Theorem~\ref{thm:asymptotic-distribution-of-LSE} in this example.
\begin{figure}[!h]
    \centering
    \includegraphics[width = \textwidth]{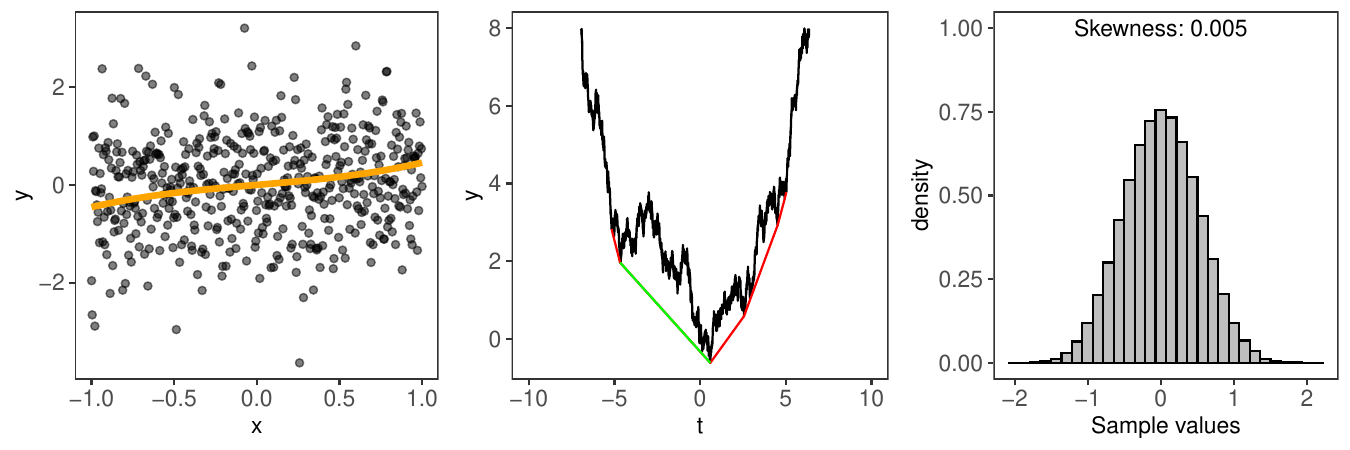}
  \caption{The left panel shows the scatter plot of data from $Y = f_{0,n}(X) + \xi,$ with $n = 500, X\sim\mathrm{Unif}(-1, 1), \xi|X\sim N(0, 1)$ and $f_{0,n}(x) = x/n^{1/5} + x^3/6$ along with the plot of $f_{0,n}(\cdot)$. The middle panel shows the plot of a two-sided Brownian motion with drift $|t|^2/8$ along with the greatest convex minorant; the green line shows the slope from the left of the greatest convex minorant at 0. The right panel shows the histogram of $10^5$ observations from the random variable $\slGCM.$}
    \label{fig:psi_eg4}
\end{figure}
\section{Proof of Corollary \ref{cor:triangular-asymptotic-distribution-of-LSE}}\label{appsec:proof-of-cor-traingular-asymptotic-distribution-of-LSE}
\begin{proof}
This can be shown by noting that, for any $c\neq 0$,
\begin{align*}
    \phi(c) &= \lim_{n\to\infty}\sqrt{\frac{k_n}{s_{k_n}}}\left(f_{0,n}\left(x_0 + \frac{c_n}{s_{k_n}}\right) - f_{0,n}(x_0)\right) \\
    &= \lim_{n\to\infty}\sqrt{\frac{k_n}{n}\cdot\frac{s_n}{s_{k_n}}}\sqrt{\frac{n}{s_{n}}}\left(f_{0,n}\left(x_0 + \frac{c_n}{s_{k_n}}\right) - f_{0,n}(x_0)\right)\\
    &=\sqrt{\frac{\tau_1}{\tau_2}}\lim_{n\to\infty}\sqrt{\frac{n}{s_{n}}}\left(f_{0,n}\left(x_0 + \frac{c_n}{s_{n}}\frac{s_n}{s_{k_n}}\right) - f_{0,n}(x_0)\right)\\
    &=\sqrt{\frac{\tau_1}{\tau_2}}\lim_{n\to\infty}\sqrt{\frac{n}{s_{n}}}\left(f_{0,n}\left(x_0 + \frac{c_n'}{s_{n}}\right) - f_{0,n}(x_0)\right)\\
    &=\sqrt{\frac{\tau_1}{\tau_2}}\psi\left(\frac{c}{\tau_2}\right).
\end{align*}
The last equality follows due to \ref{assump:continuity-of-f_{0,n}} and the fact that $\lim c_n' = c/\tau_2$. Thus,
\begin{align*}
\Phi(c) &= \mathbf{1}\{t> 0\}\int_0^t \phi(s)ds - \mathbf{1}\{t< 0\}\int_t^0 \phi(s)ds \\
&= \sqrt{\tau_1\tau_2}\Psi(t/\tau_2).
\end{align*}
The proof now follows along the same lines as Theorem \ref{thm:asymptotic-distribution-of-LSE} but with sample size $k_n$. To account for this difference in sample size, we have a rescaled drift in the limiting distribution.
\end{proof}
\section{Proof of Proposition~\ref{prop:assumptions-implications}}\label{appsec:proof-of-prop-assumptions-implications}
\begin{enumerate}
\item  From assumption~\ref{assump:continuity-of-f_{0,n}}, we obtain
\begin{equation}\label{eq:psi-fixed-c}
\psi(c) = \lim_{n\to\infty}\, \sqrt{\frac{n}{s_n}}\left(f_{0,n}\left(x_0 + \frac{c}{s_n}\right) - f_{0,n}(x_0)\right),
\end{equation}
and the function $c\mapsto \sqrt{n/s_n}(f_{0,n}(x_0 + c/s_n) - f_{0,n}(x_0))$ is monotone non-decreasing because $f_{0,n}(\cdot)$ is monotone non-decreasing under~\ref{assump:data-model-assumption}. Since the limit of a sequence of monotone functions is monotone, this implies the result. \\\\
For the second part observe that: $\limsup_{|c|\to\infty} |c(\psi(c) - \psi(3c/2))| = \infty$ implies that there exists $c_1<0, c_2>0$ such that $\psi(c_1)<0$ and $\psi(c_2)>0$ (if not, then $\psi$ is identically zero either for all $c\geq 0$ or for all $c\leq 0$ or both. In each of these cases, either $\limsup_{c\to\infty} |c(\psi(c) - \psi(3c/2))|=0$ or $\limsup_{c\to-\infty} |c(\psi(c) - \psi(3c/2))|=0$ or both, which is a contradiction).
\\\\Now using the fact that $\psi$ is non-decreasing,
\[\psi(c)\leq\psi(c_1)<0~~\forall~c\leq c_1~~\text{and}~~\psi(c)\geq\psi(c_2)>0~~\forall~c\geq c_2\]
This shows that $\psi$ is bounded away from zero as $|c|\to\infty$. Also, for $c>c_2$
\begin{align*}
    \Psi(c)&= \int_{0}^c \psi(s)ds\geq\int_{c_2}^c \psi(s)ds\\&\geq\int_{c_2}^c \psi(c_2)ds=(c-c_2)\psi(c_2)>0
\end{align*}
This implies that $\frac{1}{c}\Psi(c)$ is bounded away from zero as $c\to\infty$. The case $c\to-\infty$ can be dealt with similarly.
\item From~\ref{assump:continuity-of-f_{0,n}}, we know that $\tilde{\psi}$ is not identically equal to zero (as observed in the part above) or $\infty$ on $\mathbb{R}\backslash\{0\}$. Thus, $\exists~c_1\neq 0$ s.t. $0<|\tilde{\psi}(c_1)|<\infty$. Suppose $c_1>0$. The case $c_1<0$ is similar. Also, arguing as in part 1, there exists $c_0\in\mathbb{R}\backslash\{0\}$ such that $0<\psi(c_0)$. Fix any such $c_0$. Clearly $c_0>0$. Take any sequence $\{c_n\}$ satisfying $c_n\to c_0$.
We first show that $\limsup_{n\to\infty} s_n/s_n' \in(0,~\infty)$. To that end, first observe that as $s_n,~s_n'\geq 0~\forall~n$, $\limsup_{n\to\infty} s_n/s_n'\in[0, \infty]$. Now, if possible, suppose $\limsup_{n\to\infty} s_n/s_n'=\infty$. This means, $\exists$ a subsequence $\{n_k\}$, such that $\lim_{k\to\infty} s_{n_k}/s_{n_k}'=\infty$. Thus, $\lim_{k\to\infty} s_{n_k}'/s_{n_k}= 0$. Now, as $c_n\to c_0>0$, $\exists~ N_0\in\mathbb{N}$ s.t. $c_n>0~\forall~ n\geq N_0$. Thus, for sufficiently large $k$, $d_k=c_{n_k}s_{n_k}'/s_{n_k}\downarrow 0$. Now as $\tilde{\psi}(c_1)$ is finite, take any sequence $c_n'\to c_1>0$. For large enough $k$,
\begin{align*}
&f_{0,n}\left(x_0+\frac{c_{n_k}s'_{n_k}}{s_{n_k}s'_{n_k}}\right)\leq f_{0,n}\left(x_0+\frac{c'_{n_k}}{s'_{n_k}}\right)\\
&\implies \sqrt{\frac{n_k}{s_{n_k}'}}\left(f_{0,n}\left(x_0+\frac{c_{n_k}s'_{n_k}}{s_{n_k}s'_{n_k}}\right)-f_{0,n}(x_0)\right)\\
&\qquad\leq \sqrt{\frac{n_k}{s_{n_k}'}}\left(f_{0,n}\left(x_0+\frac{c'_{n_k}}{s'_{n_k}}\right)-f_{0,n}(x_0)\right)\\
&\qquad\leq (\tilde{\psi}(c_1)+1) < \infty.
\end{align*}
Hence, for large $k$,
\begin{align*}
    0 < \psi(c_0)&= \lim_{n\to\infty} \sqrt{\frac{n}{s_n}}\left(f_{0,n}\left(x_0+\frac{c_n}{s_n}\right)-f_{0,n}(x_0)\right)\\
    &=\lim_{k\to\infty} \sqrt{\frac{n_k}{s_{n_k}}}\left(f_{0,n}\left(x_0+\frac{c_{n_k}}{s_{n_k}}\right)-f_{0,n}(x_0)\right)\\
    &=\lim_{k\to\infty}\sqrt{\frac{s_{n_k}'}{s_{n_k}}}\sqrt{\frac{n_k}{s_{n_k}'}}\left(f_{0,n}\left(x_0+\frac{c_{n_k}s'_{n_k}}{s_{n_k}s'_{n_k}}\right)-f_{0,n}(x_0)\right)\\
    &\leq 0\times(\tilde{\psi}(c_1)+1) = 0.
\end{align*}
Thus, we get $\psi(c_0)=0$, which is a contradiction since $\psi(c_0)$ was assumed to be non-zero. Now suppose $\limsup_{n\to\infty} s_n/s_n'=0$. Choose $c_0$ such that $\tilde{\psi}(c_0)>0$. Again, we can get a subsequence $\{n_k\}$ such that $\lim_{n\to\infty} s_{n_k}/s_{n_k}'=0$ (of course, this subsequence is different from the previous one but we use the same notation for simplicity). This case is also ruled out using the same argument as above by simply swapping the roles of $s_{n_k}$ and $s_{n_k}'$ in the above calculations to obtain $\tilde{\psi}(c_0)=0$, a contradiction. Note that we needed $\psi$ and $\tilde{\psi}$ to be finite at only some point on either side of the real line.\\\\
Thus, $\limsup_{n\to\infty}s_n/s_n'= a_1 \in (0, \infty)$. As in the previous calculation, we get hold of a subsequence $\{n_k\}$ such that $\lim_{k\to\infty}s_{n_k}/s_{n_k}'= a_1$ and note that
\begin{align*}
    \tilde{\psi}(c_0)&= \lim_{k\to\infty} \sqrt{\frac{n_k}{s'_{n_k}}}\left(f_{0,n}\left(x_0+\frac{c_{n_k}}{s'_{n_k}}\right)-f_{0,n}(x_0)\right)\\
    &=\lim_{k\to\infty}\sqrt{\frac{s_{n_k}}{s'_{n_k}}}\sqrt{\frac{n_k}{s_{n_k}}}\left(f_{0,n}\left(x_0+\frac{c_{n_k}s_{n_k}}{s'_{n_k}s_{n_k}}\right)-f_{0,n}(x_0)\right)\\
    &= \sqrt{a_1}\psi(a_1c_0),\\
    \implies \tilde{\psi}(c_0)&=\sqrt{a_1}\psi(a_1c_0)~~\forall~c_0\in\mathbb{R}.
\end{align*}
Note that the entire argument can be repeated with $\limsup_{n\to\infty} s_n/s_n'$ replaced by $a_0=\liminf_{n\to\infty} s_n/s_n'$. Hence, we conclude that $a_0=\liminf_{n\to\infty} s_n/s_n'\in(0,~\infty)$ and:
\[\tilde{\psi}(c_0)=\sqrt{a_0}\psi(a_0c_0)~~\forall~c_0\in\mathbb{R}\]
Thus, what we get finally is:
\begin{equation}\sqrt{a_1}\psi(a_1c_0)=\sqrt{a_0}\psi(a_0c_0)~~\forall~c_0\in\mathbb{R}.\label{limsup-liminf}
\end{equation}
We will first show that $a_0=a_1$. Suppose not. Suppose $a_1>a_0$. Since $|\psi|$ is not identically equal to zero or $\infty$, consider $c_0'$ such that $0<|\psi(c_0')|< \infty$. Suppose $c_0'>0$. Then, as $\psi$ is non-decreasing, we can get $c_0$ such that $0<\psi(a_0c_0)< \infty$ (take for e.g. $c_0=c_0'/a_0>0$, as $a_0\neq 0$). Now, again as $\psi$ is non-decreasing, and $a_1> a_0$,
\begin{align*}
    a_1c_0> a_0c_0&\implies \psi(a_1c_0)\geq \psi(a_0c_0)\\
    &\implies \sqrt{a_1}\psi(a_1c_0)\geq \sqrt{a_1}\psi(a_0c_0)>\sqrt{a_0}\psi(a_0c_0)
\end{align*}
This is a contradiction to \eqref{limsup-liminf} as equality must be ensured. Thus, we must have $a_0=a_1$. The case $c_0'<0$ can be dealt with in a similar manner. Hence $a=\lim_{n\to\infty} s_n/s_n'$ exists and belongs to $(0,~\infty)$. Thus, we have what we needed to show:
\[\tilde{\psi}(c)=\sqrt{a}\psi(ac)~~\forall~c\in\mathbb{R}\]
 Now for the second part, suppose $a=1$. Then $\tilde{\psi}(c)=\psi(c)$ is trivial. Now suppose $\tilde{\psi}\equiv\psi$. We show that $a = 1$ is the only permissible limit. Suppose, if possible, $a \neq 1$. Take any $c_0 \neq 0$ such that $0<|\tilde{\psi}(c_0)| <\infty$. Suppose $c_0>0$. If $a > 1$, then $ac_0 > c_0\implies \psi(ac_0)\geq\psi(c_0)>0$ and thus, $\tilde{\psi}(c_0)=\sqrt{a}\psi(ac_0)>\psi(c_0)$ which is a contradiction as equality must hold. Similarly, if $a < 1$, then $ac_0 < c_0$, which implies, $\psi(ac_0)\leq\psi(c_0)$ and thus, $\sqrt{a}\psi(ac_0)<\psi(c_0)$ which is again a contradiction. Note that for $a<1$, in the above step, $\psi(ac_0)$ can be zero and the strict inequality would still be attained. The case with $c_0<0$ is similar. Thus, we have $\lim_{n\to\infty} s_n/s_n'=1$.
\item First note that the result is trivially true if either $c = 0$ or $t = 0$. Fix $t\neq 0$ and $c \neq 0$. Define for notational convenience,
\[
g_n(y) = \sqrt{\frac{n}{s_n}}\left(f_{0,n}\left(x_0+\frac{y}{s_n}\right)-f_{0,n}(x_0)\right).
\]
We want to prove that
\[
\lim_{n\to\infty} \int_0^t cg_n(c_nx)dx = \Psi(ct)\quad\mbox{for}\quad t > 0;~~
\lim_{n\to\infty} \int_t^0 cg_n(c_nx)dx = \Psi(ct)\quad\mbox{for}\quad t < 0,
\]
By a change of variable $s = cx$ in both integrals, they are equivalent to showing
\begin{align*}
\lim_{n\to\infty} \int_0^{ct} g_n\left(\frac{c_n}{c}s\right)ds &= \Psi(ct)\quad\mbox{for}\quad t > 0;\\
\lim_{n\to\infty} \int_{ct}^0 g_n\left(\frac{c_n}{c}s\right)ds &= \Psi(ct)\quad\mbox{for}\quad t < 0.
\end{align*}
Take any sequence $c_{n}\to c$. Thus, $c_nt\to ct$ as $n\to\infty$. This implies that there exists an $N_0 \ge 1$ (that can depend on $\{c_n\}$) such that $\mbox{sign}(c_nx) = \mbox{sign}(cx)$ for all $n\ge N_0$ and all $x\in\mathbb{R}$. Then from~\ref{assump:continuity-of-f_{0,n}}, we know
\[
\lim_{n\to\infty}g_n(c_nt) = \psi(ct),
\]
exists and is finite. This implies that there exists an $N_1 \ge 1$ (that can depend on $\{c_n\}, c, t$) such that for all $n\ge N_1$,
\[
\psi(ct) - 1 \le g_n(c_nt) \le \psi(ct) + 1.
\]
If $c > 0$ and $t > 0$, then for all $x\in[0, t]$ and all $n\ge N_0$, we have $c_n > 0$ and $c_nx \geq 0$. Hence, from the assumption~\ref{assump:data-model-assumption} that $f_{0,n}$ is monotone non-decreasing we conclude
\[
0 \le g_n(c_nx) \le g_n(c_nt) \le \psi(ct) + 1,\quad\mbox{for all}\quad x\in[0, t].
\]
If $c < 0$ and $t > 0$, then for all $x\in[0, t]$ and all $n\ge N_0$, we have $c_n < 0$ and $c_nx < 0$. Hence, from assumption~\ref{assump:data-model-assumption}, we conclude
\[
\psi(ct) - 1 \le g_n(c_nt) \le g_n(c_nx) \le 0,\quad\mbox{for all}\quad x\in[0, t].
\]
Therefore, by the bounded convergence theorem, we get
\[
\lim_{n\to\infty} \int_0^t g_n(c_nx)dx = \int_0^t \lim_{n\to\infty} g_n(c_nx)dx = \int_0^t \psi(cx)dx = \frac{1}{c}\int_0^{ct} \psi(s)ds = \frac{\Psi(ct)}{c}.
\]
Similarly, if $c > 0$ and $t < 0$, then for all $x\in[t, 0]$ and all $n\ge N_0$, we have $c_n > 0$ and $c_nx < 0$. Hence, from assumption~\ref{assump:data-model-assumption}, we conclude
\[
\psi(ct) - 1 \le g_n(c_nt) \le g_n(c_nx) \le 0\quad\mbox{for all}\quad x\in[t, 0].
\]
Again, for $c < 0$ and $t < 0$, and all $x\in[t, 0]$ and all $n\ge N_0$, we have $c_n < 0$ and $c_nx > 0$. Hence, from assumption~\ref{assump:data-model-assumption}, we conclude
\[
0 \le g_n(c_nx) \le g_n(c_nt) \le \psi(ct)+1\quad\mbox{for all}\quad x\in[t, 0].
\]
Therefore, again by the bounded convergence theorem, we get
\[
\lim_{n\to\infty} \int_t^0 g_n(c_nx)dx = \int_t^0 \lim_{n\to\infty} g_n(c_nx)dx = \int_t^0 \psi(cx)dx = \frac{1}{c}\int_{ct}^0 \psi(s)ds = -\frac{\Psi(ct)}{c}.
\]
Thus the limit exists and is finite $\forall~t$.

\item 
First consider $\Gamma_1(\cdot)$. Recall that $\psi(\cdot)$ is a non-decreasing function. Thus, for $c > 0$,
\[
\Gamma_1(c) = \psi(c) - \frac{\Psi(c)}{c} = \frac{1}{c}\int_0^c [\psi(c) - \psi(t)]dt \ge 0.
\]
Moreover, $\Gamma_1(c) \ge \frac{1}{c}\int_0^{2c/3} [\psi(c) - \psi(t)]dt \ge 2(\psi(c) - \psi(2c/3))/3$. Hence,
\[
\limsup_{c\to\infty} c\Gamma_1(c) \ge \limsup_{c\to\infty} (2c/3)(\psi(c) - \psi(2c/3)) = \infty,
\]
by the assumption.
For $c < 0$,
\[
\Gamma_1(c) = -\psi(c) + \frac{\Psi(c)}{c} = -\psi(c) - \frac{1}{c}\int_{c}^0 \psi(t)dt = -\frac{1}{c}\int_c^0 [\psi(t) - \psi(c)]dt \ge 0,
\]
because $\psi(t) \ge \psi(c)$ for all $t\in[c, 0]$. Moreover,
\[
\Gamma_1(c) \ge -\frac{1}{c}\int_{2c/3}^0 [\psi(t) - \psi(c)]dt \ge \frac{2}{3}[\psi(2c/3) - \psi(c)].
\]
Hence,
\[
\limsup_{c\to-\infty} (-c)\Gamma_1(c) \ge \limsup_{c\to-\infty} (-2c/3)[\psi(2c/3) - \psi(c)] = \infty.
\]
Now consider $\Gamma_2(\cdot)$. For $c > 0$,
\[
\Gamma_2(c) = -\psi(c) + \frac{\Psi(2c) - \Psi(c)}{c} = -\psi(c) + \frac{1}{c}\int_c^{2c} \psi(t)dt = \frac{1}{c}\int_c^{2c} [\psi(t) - \psi(c)]dt \ge 0.
\]
Moreover, $\Gamma_2(c) \ge (1/c)\int_{3c/2}^{2c} [\psi(t) - \psi(c)]dt \ge (1/2)[\psi(3c/2) - \psi(c)]$. Hence, the result follows. For $c < 0$, 
\[
\Gamma_2(c) = \psi(c) + \frac{\Psi(2c) - \Psi(c)}{c} = \psi(c) - \frac{1}{c}\int_{2c}^{c} \psi(t)dt = -\frac{1}{c}\int_{2c}^c [\psi(c) - \psi(t)]dt \ge 0.
\]
Moreover, $\Gamma_2(c) \ge (-1/c)\int_{2c}^{3c/2} [\psi(c) - \psi(t)]dt \ge (1/2)(\psi(c) - \psi(3c/2))$. Hence, the result follows.
\end{enumerate}

\section{Proof of Lemma~\ref{lem:unif-validity-under-property}}\label{appsec:proof-of-unif-validity-under-property}
\begin{proof} 
By the definition of the infimum, for each $n\in\N$, $\exists~f_{0,n}\in \mathcal{F}, P_n\in\mathcal{P}(f_{0,n})$ such that:
\begin{align}
\label{eq:3-5-subsequence}
\inf_{f\in\mathcal{F}} \inf_{P\in\mathcal{P}(f)} \Pr_P\left(f(x_0)\in \widehat{\mathrm{CI}}_{n,\alpha}(x_0)\right) 
&\leq \left(1+\frac{1}{n}\right) \Pr_{P_n}\left(f_{0,n}(x_0)\in \widehat{\mathrm{CI}}_{n,\alpha}(x_0)\right)
\end{align}

 As $\{f_{0,n}\}\subseteq \mathcal{F}$, due to property $\mathcal{U}$, for the sequence $\{P_n\in\mathcal{P}(f_{0,n})\}_{n\geq 1}$, using Theorem 2 of~\cite{kuchibhotla2021hulc} and the assumption that $\limsup_{n\to\infty} \Delta_{n,\alpha}=0$, we get that 
\begin{align*}
&\mathbb{P}_{P_n}\left(f_{0,n}(x_0) \notin \widehat{\mathrm{CI}}_{n,\alpha}(x_0)\right) ~\le~ \alpha\left(1 + 2(B_{\alpha}\Delta_{n,\alpha})^2e^{2B_{\alpha}\Delta_{n,\alpha}}\right)\\
\implies \limsup_{n\to\infty} ~&\mathbb{P}_{P_n}\left(f_{0,n}(x_0) \notin \widehat{\mathrm{CI}}_{n,\alpha}(x_0)\right) ~\le~ \limsup_{n\to\infty}\alpha\left(1 + 2(B_{\alpha}\Delta_{n,\alpha})^2e^{2B_{\alpha}\Delta_{n,\alpha}}\right) =\alpha.
\end{align*}  This implies,
\begin{align*}
    &\liminf_{n\to\infty} \Pr_{P_n} \left(f_{0,n}(x_0)\in \widehat{\mathrm{CI}}_{n,\alpha}(x_0)\right)=1-\alpha\\
    &\implies\liminf_{n\to\infty} \left(1+\frac{1}{n}\right) \Pr_{P_n} \left(f_{0,n}(x_0)\in \widehat{\mathrm{CI}}_{n,\alpha}(x_0)\right)=1-\alpha\\
    &\overset{(\text{by \eqref{eq:3-5-subsequence}})}{\implies} \liminf_{n\to\infty} \inf_{f\in\mathcal{F}} \inf_{P\in\mathcal{P}(f)}\Pr_{P} \left(f(x_0)\in \widehat{\mathrm{CI}}_{n,\alpha}(x_0)\right)=1-\alpha.
\end{align*}
\end{proof}

\section{Proof of Proposition~\ref{prop:med-unb-unif-validity}}
\label{appsec:proof-of-med-unb-unif-validity}
\begin{proof}
    Assumption~\ref{assump:locally-bounded} implies that $\psi_{f,\,\rho_f}(\cdot,h),\,0<h\le 1$ are uniformly bounded on bounded sets. Because $f(\cdot)$ is non-decreasing, assumption~\ref{assump:locally-bounded} is equivalent to:
\[
\sup_{f\in\mathcal{F}(x_0)}\ \sup_{0<h\le 1}\ \max\{|\psi_{f,\,\rho_f}(-C,h)|,\,|\psi_{f,\,\rho_f}(C,h)|\}\le \mathfrak{B}_C
\quad~\forall~C\ge 0.
\]
Combined with~\ref{assump:locally-symmetric}, this can be further reduced to $\sup_{f\in\mathcal{F}(x_0)}\sup_{0<h\le 1}|\psi_{f,\,\rho_f}(C,h)| < \infty$ for all $C \ge 0$. Assumption~\ref{assump:locally-symmetric} implies that $f(\cdot)$ is locally anti-symmetric. Suppose $\{f_{0,n}\}_{n\geq 1}$ is any sequence of functions from $\mathcal{F}$. Under~\ref{assump:locally-bounded}, Helly's selection theorem implies that for any subsequence $\{j_n\}_{n\ge1}$ of $\{n\}_{n\ge1}$, there exists a further subsequence $\{k_n\}_{n\ge1}$, a sequence $h_{k_n}\downarrow 0$, and a non-decreasing function $\psi(\cdot)$ such that
    \[
    \psi_{f_{0,k_n},\,\rho_{f_{0,k_n}}}(c,h_{k_n}) \to \psi(c)\quad \text{as }n\to\infty~\forall~c\in\mathbb{R}.
    \]

\noindent Assumption~\ref{assump:locally-symmetric} applied to the sequence $f_{0,k_n}$ now implies that $\psi(-c) = -\psi(c)$ for all $c \ge 0$. This is because, by assumption~\ref{assump:locally-symmetric}, for any fixed $c>0$,
\[
\sup_{f\in\mathcal{F}(x_0)}\left|\limsup_{h\downarrow 0}\,\frac{\psi_{f,\,\rho_f}(c,h)}{\psi_{f,\,\rho_f}(-c,h)}+1\right|=0.
\]
In particular, along the sequence $h_{k_n}\downarrow 0$ there exists a further subsequence (not relabeled) such that
\[
\left|\frac{\psi_{f_{0,k_n},\,\rho_{f_{0,k_n}}}(c,h_{k_n})}{\psi_{f_{0,k_n},\,\rho_{f_{0,k_n}}}(-c,h_{k_n})}+1\right|\to 0.
\]
Therefore, for any $c>0$ at which $\psi$ is continuous at $\pm c$,
\[
\left|\frac{\psi(c)}{\psi(-c)}+1\right|
=\lim_{n\to\infty}\left|\frac{\psi_{f_{0,k_n},\,\rho_{f_{0,k_n}}}(c,h_{k_n})}{\psi_{f_{0,k_n},\,\rho_{f_{0,k_n}}}(-c,h_{k_n})}+1\right|
=0,
\]
and hence $\psi(-c)=-\psi(c)$. Since $\psi$ is monotone, it is a.s. continuous and therefore $\psi(-c)=-\psi(c)$ for all $c\ge 0$. Next, we use assumption~\ref{assump:tail-nonconstancy} to verify the required non-constancy in assumption~\ref{assump:continuity-of-f_{0,n}} of Theorem~\ref{thm:asymptotic-distribution-of-LSE}. Fix $M>0$. By~\ref{assump:tail-nonconstancy}, there exists $C_M<\infty$ such that for all $C\ge C_M$,
    \[
    \sup_{f\in\mathcal{F}(x_0)}\ \sup_{0<h\le 1}\ \sup_{|c|>C}
    \frac{1}{\big|c(\psi_{f,\,\rho_f}(3c/2,h)-\psi_{f,\,\rho_f}(c,h))\big|}
    \le \frac{1}{M}.
    \]
    In particular, for any $c$ with $|c|\ge C_M$, taking $C=C_M$ yields
    \[
    \sup_{0<h\le 1}\frac{1}{\big|c(\psi_{f_{0,n},\,\rho_{f_{0,n}}}(3c/2,h)-\psi_{f_{0,n},\,\rho_{f_{0,n}}}(c,h))\big|}
    \le \frac{1}{M},
    \]
    i.e.
    \[
    \big|c(\psi_{f_{0,n},\,\rho_{f_{0,n}}}(3c/2,h)-\psi_{f_{0,n},\,\rho_{f_{0,n}}}(c,h))\big|\ge M
    \quad \text{for all }0<h\le 1.
    \]
    Passing to the subsequence $\{k_n\}$ and then letting $n\to\infty$ gives, for any such $c$ at which $\psi$ is continuous at $c$ and $3c/2$,
    \[
    \big|c(\psi(3c/2)-\psi(c))\big|
    =\lim_{n\to\infty}\big|c(\psi_{f_{0,k_n},\,\rho_{f_{0,k_n}}}(3c/2,h_{k_n})-\psi_{f_{0,k_n},\,\rho_{f_{0,k_n}}}(c,h_{k_n}))\big|
    \ge M.
    \]
    Since $M>0$ is arbitrary, this implies $\big|c(\psi(3c/2)-\psi(c))\big|\to\infty$ as $|c|\to\infty$, i.e., the non-constancy required by (A3) holds for the subsequential limit drift $\psi(\cdot)$.\\

\noindent Hence, along the subsequence $\{k_n\}_{n\ge1}$, under~\ref{assump:continuous-distribution-of-covariates-uniform},~\ref{assump:data-model-assumption-uniform}, Theorem~\ref{thm:asymptotic-distribution-of-LSE} combined with Theorem~\ref{thm:general-drifted-brownian-motion} implies $\sqrt{k_n/s_{k_n}}(\widehat{f}_{k_n}(x_0) - f_{0,k_n}(x_0))$ converges in distribution to a symmetric distribution. Using Corollary~\ref{cor:triangular-asymptotic-distribution-of-LSE}, the median bias as defined in \eqref{eq:median-bias-definition} goes to zero. Thus, this class satisfies property $\mathcal{U}$.
\end{proof}

\section{Proof of Theorem~\ref{thm:general-drifted-brownian-motion}}\label{appsec:proof-of-thm-general-drifted-brownian-motion}
\begin{lemma}\label{lem:symmetry-switch-relation}
Let $g:I_0\to\mathbb{R}$ be a continuous function and define:
\[\argminp_{x\in I_0}(g(x))=\sup_{x\in I_0}\{g(x)=\min_{y\in I_0} g(y)\}~~;~~\argmaxp_{x\in I_0}(g(x))=\sup_{x\in I_0}\{g(x)=\max_{y\in I_0} g(y)\}\]
Suppose $\slGCM[g](x_0)$ (and $\slLCM[g](x_0)$) denote the slope from the left (left derivative) of the Greatest Convex Minorant (and Least Concave Majorant) of the function $f$ evaluated at the point $x_0$. Then,
\[\slGCM[g](x_0)\leq a\iff \argminp_{x\in I_0}(g(x)-ax)\geq x_0\]
\[\slLCM[g](x_0)< a\iff \argmaxp_{x\in I_0}(g(x)-ax)<x_0\]
\end{lemma}
\begin{proof}
This lemma follows directly from Chapter 3, Lemma 3.2 in \cite{groeneboom2014nonparametric}. The first equivalence relation is the above-stated lemma itself as $g$ continuous means it is lower semi-continuous. For the second one, observe that:
\begin{align*}
\slLCM[g](x_0)< a &\iff -\slGCM[-g](x_0)< a\\
&\iff \slGCM[-g](x_0)> -a\\
&\iff \text{argmin}_{x\in I_0}^{+}(-g(x)+ax)< x_0
\end{align*}
Note that the last equivalence holds from the same cited lemma above, as $-g$ is lower semi-continuous. Thus,
\begin{align*}
\slLCM[g](x_0)< a 
&\iff \text{argmin}_{x\in I_0}^{+}(-g(x)+ax)< x_0\\
&\iff \text{argmax}_{x\in I_0}^{+}(g(x)-ax)< x_0
\end{align*}
\end{proof}
\begin{lemma}\label{lem:symmetry-uniqueness}
Let 
\[\mathcal{S}(t)=B(t)+d(t)\]
where $\{B(t)\}_{t\in I}$ is a standard two-sided Brownian motion on $I\subseteq\mathbb{R}$ which is $\sigma$-compact and $d$ is a deterministic continuous function. Then the location of the maximum (if exists) is a.s. unique.
\end{lemma}
\begin{proof} 
This follows from Lemma 2.6 in \cite{kim1990cube}. 
\end{proof}
\begin{proof}[Proof of Theorem~\ref{thm:general-drifted-brownian-motion}]
Let $\mathcal{F}_{t_0}$ and $\mathcal{G}_{-t_0}$ denote the distribution functions of $Y_1=\slGCM[\mathcal{B}_1](t_0)$ and $Y_2=\slLCM[\mathcal{B}_2](-t_0)$ respectively. Let \[Y_1^c=\slGCM_{[-c,~c]}[\mathcal{B}_1](t_0) ~\text{and}~ Y_2^c=\slLCM_{[-c,~c]}[\mathcal{B}_2](-t_0)\] denote the corresponding random variables on the compact set $I=[-c,~c]$. Note that on this compact set, $Y_1^c$ and $Y_2^c$ can be written in terms of an argmin functional as in Lemma \ref{lem:symmetry-switch-relation}. Also, we get $\mathbb{P}(Y_1\neq Y_1^c)\to 0$ and $\mathbb{P}(Y_2\neq Y_2^c)\to 0$ as $c\to\infty$ by slightly modifying Lemma 6.2 in \cite{rao1969estimation}, as we did in the last part of the proof of Theorem \ref{thm:asymptotic-distribution-of-LSE}. Thus,
\begin{align*}
\mathcal{F}_{t_0}(a)&=\mathbb{P}(Y_1\leq a)=\mathbb{P}(Y_1\leq a,~Y_1=Y_1^c)+\mathbb{P}(Y_1\leq a,~Y_1\neq Y_1^c)\\
&=\mathbb{P}(Y_1^c\leq a)+\mathbb{P}(Y_1\leq a,~Y_1\neq Y_1^c)\\
&\overset{\text{(a)}}{=}\mathbb{P}\left(\underset{{u\in I}}{\text{argmin}^{+}}(\mathcal{B}_1(u)-au)\geq t_0\right)+\mathbb{P}(Y_1\leq a,~Y_1\neq Y_1^c)\\
&\overset{\text{(b)}}{=}\mathbb{P}\left(\underset{{u\in I}}{\text{argmin}}(\mathcal{B}_1(u)-au)\geq t_0\right)+\mathbb{P}(Y_1\leq a,~Y_1\neq Y_1^c)\\
&=\mathbb{P}\left(-\underset{{u\in I}}{\text{argmin}}(\mathcal{B}_1(u)-au)\leq -t_0\right)+\mathbb{P}(Y_1\leq a,~Y_1\neq Y_1^c)\\
&\overset{\text{(c)}}{=}\mathbb{P}\left(\underset{{u\in I}}{\text{argmin}}(\mathcal{B}_1(-u)+au)\leq -t_0\right)+\mathbb{P}(Y_1\leq a,~Y_1\neq Y_1^c)
\end{align*}
\begin{align*}
\mathcal{F}_{t_0}(a)
&\overset{\text{(d)}}{=}\mathbb{P}\left(\underset{{u\in I}}{\text{argmax}}(-\mathcal{B}_1(-u)-au)\leq -t_0\right)+\mathbb{P}(Y_1\leq a,~Y_1\neq Y_1^c)\\
&=\mathbb{P}\left(\underset{{u\in I}}{\text{argmax}}(-B_1(-u)-d(-u)-au)\leq -t_0\right)+\mathbb{P}(Y_1\leq a,~Y_1\neq Y_1^c)
\end{align*}
where (a) follows from Lemma~\ref{lem:symmetry-switch-relation}; (b) from Lemma~\ref{lem:symmetry-uniqueness}; (c) since $\text{argmin}_{x\in I_0}f(x)=-\text{argmin}_{x\in I_0}f(-x)$ for $I_0$ symmetric around $0$; and (d) since $\text{argmin}_x f(x)=\text{argmax}_x -f(x)$.
Now, observe that the processes $\{B_1(u)\}_{u\in\mathbb{R}}$ and $\{-B_1(-u)\}_{u\in\mathbb{R}}$ are identically distributed as stochastic processes and they have the same distribution as $\{B_2(u)\}_{u\in\mathbb{R}}$ Thus,
\begin{align*}
\mathcal{F}_{t_0}(a)&=\mathbb{P}\left(\underset{{u\in I}}{\text{argmax}}(B_2(u)-d(-u)-au)\leq -t_0\right)+\mathbb{P}(Y_1\leq a,~Y_1\neq Y_1^c)\\
&\overset{\text{(e)}}{=}\mathbb{P}\left(\underset{{u\in I}}{\text{argmax}}(B_2(u)-d(u)-au)\leq -t_0\right)+\mathbb{P}(Y_1\leq a,~Y_1\neq Y_1^c)\\
&\overset{\text{(f)}}{=}\mathbb{P}\left(\underset{{u\in I}}{\text{argmax}^{+}}(\mathcal{B}_2(u)-au)\leq -t_0\right)+\mathbb{P}(Y_1\leq a,~Y_1\neq Y_1^c)\\
&\overset{\text{(g)}}{=}\mathbb{P}(Y_2^c \leq -t_0)+\mathbb{P}(Y_1\leq a,~Y_1\neq Y_1^c)
\end{align*}
where (e) uses $d(-u)=d(u)$ for all $u\in I$; (f) follows from Lemma~\ref{lem:symmetry-uniqueness}; and (g) follows from Lemmas~\ref{lem:symmetry-switch-relation} and SA-1 in \cite{cattaneo2023bootstrap}.
Now taking $c\to\infty$ on the RHS,
\begin{align*}
\mathcal{F}_{t_0}(a)&=\lim_{c\to\infty} \mathbb{P}(Y_2^c \le -t_0)=\lim_{c\to\infty} [\mathbb{P}(Y_2^c \leq -t_0, Y_2^c=Y_2)+\mathbb{P}(Y_2^c \leq -t_0, Y_2^c\neq Y_2)]\\
&=\lim_{c\to\infty} \mathbb{P}(Y_2 \le -t_0)=\mathbb{P}(Y_2 \le -t_0)\\
&=\mathcal{G}_{-t_0}(a)
\end{align*}
This proves the first part of the result. Now setting $t_0=0$, gives us:
\[\slGCM[\mathcal{B}_1](0)\overset{d}{=}\slLCM[\mathcal{B}_2](0)\]
Also,
\begin{align*}
\slLCM[\mathcal{B}_2](0)&=-\slGCM[-\mathcal{B}_2(t)](0)=-\slGCM[-B_2(t)+d(t)](0)\\
&\overset{d}{=}-\slGCM[B_1(t)+d(t)](0)=-\slGCM[\mathcal{B}_1](0)
\end{align*}
This shows the second part of our result.
\end{proof}

\end{document}